\documentclass[reqno]{amsart}
\usepackage{algorithm}
\usepackage{algorithmic}
\usepackage{bm}
\usepackage{amsrefs}
\usepackage{appendix}
\usepackage[shortlabels]{enumitem}
\usepackage{hyperref}
\usepackage[capitalize]{cleveref}
\usepackage{tikz-cd}
\usepackage{tikz-network}
\usepackage[margin=1in]{geometry}
\usepackage{xcolor}
\usepackage{siunitx}

\DeclareMathOperator{\tr}{tr}
\DeclareMathOperator{\ran}{ran}

\DeclareMathOperator{\spn}{span}
\DeclareMathOperator{\Id}{Id}
\DeclareMathOperator{\divr}{div}
\DeclareMathOperator{\supp}{supp}

\newtheorem{thm}{Theorem}
\newtheorem{cor}[thm]{Corollary}

\newtheorem{lem}[thm]{Lemma}

\theoremstyle{definition}
\newtheorem{defn}[thm]{Definition}
\newtheorem*{example*}{Example}

\theoremstyle{remark}
\newtheorem{rk}[thm]{Remark}

\AddToHook{env/thm/begin}{\crefalias{thm}{thm}}
\crefname{thm}{Theorem}{Theorems}
\Crefname{thm}{Theorem}{Theorems}

\AddToHook{env/lem/begin}{\crefalias{thm}{lem}}
\crefname{lem}{Lemma}{Lemmas}
\Crefname{lem}{Lemma}{Lemmas}

\AddToHook{env/prop/begin}{\crefalias{thm}{prop}}
\crefname{prop}{Proposition}{Propositions}
\Crefname{prop}{Proposition}{Propositions}

\AddToHook{env/cor/begin}{\crefalias{thm}{cor}}
\crefname{cor}{Corollary}{Corollaries}
\Crefname{cor}{Corollary}{Corollaries}

\AddToHook{env/rk/begin}{\crefalias{thm}{rk}}
\crefname{rk}{Remark}{Remarks}
\Crefname{rk}{Remark}{Remarks}

\crefname{prty}{property}{properties}
\Crefname{prty}{Property}{Properties}
\crefname{alg}{algorithm}{algorithms}
\Crefname{alg}{Algorithm}{Algorithms}

\AddToHook{env/appendices/begin}{\crefalias{section}{appendix}}
\AddToHook{env/appendices/begin}{\crefalias{subsection}{appendix}}

\tikzset{
  shaded/.style = {fill=red!10!blue!20!gray!30!white},
  unshaded/.style = {fill=white},
  shadedw/.style = {fill=white},
  shadedb/.style = {fill=blue!120!gray!30!white},
  shadedr/.style = {fill=red!120!gray!30!white},
  shadedg/.style = {fill=green!120!gray!30!white},
  shadedy/.style = {fill=yellow!120!gray!30!white},
  Tcirc/.style = {circle, draw, thick, fill=white, opaque},
  Tellip/.style = {ellipse, draw, thick, fill=white, opaque},
  Tbox/.style = {rounded corners,rectangle, draw, thick, fill=vertexfill, opaque},
    align/.style = {scale=.7,baseline},
    align.7/.style = {scale=.7, baseline},
    align1/.style = {scale=1.05, baseline},
    align1.5/.style = {scale=1.5,baseline},
  every picture/.style=semithick
}

\usetikzlibrary{
  hobby,
  decorations.pathreplacing,
  decorations.pathmorphing,
  shapes.geometric,
  calc
}

\tikzset{
  knot diagram/every strand/.append style={black, thick},
  stock/.style={consider self intersections=true, end tolerance=1pt, clip width=5pt, clip radius=5pt},
  stockthick/.style={consider self intersections=true, end tolerance=1pt, clip width=3pt, clip radius=20pt},
  shaded/.style = {fill=red!10!blue!20!gray!30!white},
  unshaded/.style = {fill=white},
}

\title{Tensor network approximation of Koopman operators}
\author[D.\ Giannakis et al.]{Dimitrios Giannakis}
\address{Department of Mathematics, Dartmouth College, Hanover, NH 03755, USA.}
\email{dimitrios.giannakis@dartmouth.edu}
\author[]{Mohammad Javad Latifi Jebelli}
\address{Department of Mathematics, Dartmouth College, Hanover, NH 03755, USA.}
\email{mohammad.javad.latifi.jebelli@dartmouth.edu}
\author[]{Michael Montgomery}
\address{Department of Mathematics, Dartmouth College, Hanover, NH 03755, USA.}
\email{michael.r.montgomery@dartmouth.edu}
\author[]{Philipp Pfeffer}
\address{Department of Mechanical Engineering, Technische Universit\"at Ilmenau, D-98684 Ilmenau, Germany.}
\email{philipp.pfeffer@tu-ilmenau.de}
\author[]{J\"org Schumacher}
\address{Department of Mechanical Engineering, Technische Universit\"at Ilmenau, D-98684 Ilmenau, Germany.}
\email{joerg.schumacher@tu-ilmenau.de}
\author[]{Joanna Slawinska}
\address{Department of Mathematics, Dartmouth College, Hanover, NH 03755, USA.}
\email{joanna.m.slawinska@dartmouth.edu}

\subjclass[2020]{47A35, 37Nxx, 65Pxx, 47B32, 47A80}

\begin{document}

\date{}

\begin{abstract}
	We propose a tensor network framework for approximating the evolution of observables of measure-preserving ergodic systems.
	Our approach is based on a spectrally-convergent approximation of the skew-adjoint Koopman generator by a diagonalizable, skew-adjoint operator $W_\tau$ that acts on a reproducing kernel Hilbert space $\mathcal H_\tau$ with coalgebra structure and Banach algebra structure under the pointwise product of functions.
	Leveraging this structure, we lift the unitary evolution operators $e^{t W_\tau}$ (which can be thought of as regularized Koopman operators) to a unitary evolution group on the Fock space $F(\mathcal H_\tau)$ generated by $\mathcal H_\tau$ that acts multiplicatively with respect to the tensor product.
	Our scheme also employs a representation of classical observables ($L^\infty$ functions of the state) by quantum observables (self-adjoint operators) acting on the Fock space, and a representation of probability densities in $L^1$ by quantum states.
	Combining these constructions leads to an approximation of the Koopman evolution of observables that is representable as evaluation of a tree tensor network built on a tensor product subspace $\mathcal H_\tau^{\otimes n} \subset F(\mathcal H_\tau)$ of arbitrarily high grading $n \in \mathbb N$.
	A key feature of this quantum-inspired approximation is that it captures information from a tensor product space of dimension $(2d+1)^n$, generated from a collection of $2d + 1$ eigenfunctions of $W_\tau$.
	Furthermore, the approximation is positivity preserving.
	The paper contains a theoretical convergence analysis of the method and numerical applications to two dynamical systems on the 2-torus: an ergodic torus rotation as an example with pure point Koopman spectrum and a Stepanoff flow as an example with topological weak mixing.
	The examples demonstrate improved consistency and prediction skill over conventional subspace projection methods, while also highlighting challenges stemming from numerical discretization of high-dimensional tensor product spaces.
\end{abstract}

\maketitle

\section{Introduction}
\label{sec:intro}

In recent years, there has been considerable interest in the development of operator-theoretic techniques for computational analysis and modeling of dynamical systems.
These methods leverage the linearity of the induced action of (nonlinear) state space dynamics on linear spaces of observables or measures, implemented through the Koopman and transfer operators, respectively, to carry out tasks such as mode decomposition, forecasting, uncertainty quantification, and control using linear-operator techniques.
From an analytical standpoint, the operator-theoretic approach to dynamics dates back to classical work of Koopman and von Neumann from the 1930s \cites{Koopman31,KoopmanVonNeumann32}, and has since become central to modern ergodic theory \cites{Baladi00,EisnerEtAl15}.
Starting from the late 1990s \cites{Froyland97,DellnitzJunge99,MezicBanaszuk99}, there has been a surge of research in this area from the perspective of data-driven computational techniques.
Popular examples include set-theoretic methods, \cite{DellnitzEtAl01}, Fourier analytical methods \cite{Mezic05}, and subspace projection methods; see, e.g., \cites{BruntonEtAl22,MauroyEtAl20,OttoRowley21,Colbrook24} for comprehensive surveys.

Among subspace projection methods, the dynamic mode decomposition (DMD) \cites{SchmidSesterhenn08,Schmid10,RowleyEtAl09} and the extended DMD (EDMD) \cite{WilliamsEtAl15} are some of the most popular approaches that build finite-rank approximations of the Koopman operator on spaces spanned by dictionaries of linear and possibly nonlinear observables, respectively.
These methods are related to linear inverse modeling  techniques \cite{Penland89} introduced in the 1980s, and have since been modified and generalized in various ways; e.g., by combining them with delay-embedding approaches \cites{ArbabiMezic17,BruntonEtAl17,DasGiannakis19}, methods for dictionary learning \cite{ConstanteAmoresEtAl24}, and Laplace transform techniques \cite{SusukiEtAl21}.
Extensions of (E)DMD that impose physics-informed constraints such as unitarity of the Koopman/transfer operator under measure-preserving dynamics have also been developed \cites{Colbrook23,BaddooEtAl23}.
Other approaches have focused on approximating the generator of Koopman/transfer semigroups in continuous time \cites{FroylandEtAl13b,BerryEtAl15,GiannakisEtAl15,Giannakis19,GiannakisValva24,GiannakisValva25}, also known as the Liouville operator \cites{RosenfeldEtAl19,RosenfeldEtAl22}.

While many data-driven methods perform approximation of Koopman operators in $L^p$ spaces, some approaches have explored approximations in reproducing kernel Hilbert spaces (RKHSs) \cites{Kawahara16,DasEtAl21,KlusEtAl20,KosticEtAl22} or reproducing kernel Banach spaces \cite{IkedaEtAl22}, making connections with aspects of statistical learning theory.
Yet another approach has been to build approximations of spectral measures \cites{KordaEtAl20,GovindarajanEtAl21,ColbrookTownsend24}, which allows approximation of functions of the operators via the functional calculus.
In the setting of non-autonomous dynamics, there is a rich literature on techniques based on the theory of operator cocycles \cites{Froyland13,FroylandPadberg09,GonzalezQuas14} and dynamic Laplace operators \cites{Froyland15,FroylandKoltai23} for detecting coherent sets under time-dependent dynamics.

\subsection{Connections with quantum information science}
\label{sec:intro_quantum}

Besides data-driven computational analysis and modeling, a more recent impetus for the development of operator methods for dynamical systems has been the prospect of practical quantum computing and the advances in information processing capabilities that it promises to deliver.
Since the early 2000s \cites{BenentiEtAl01,Kacewicz06,LeytonOsborne08}, efforts have been underway for designing quantum algorithms for simulation of classical systems.
This research has yielded a large body of methodological approaches \cite{BerryEtAl17,ElliottGu18,Joseph20,LloydEtAl20,KyriienkoEtAl21,KalevHen21,LiuEtAl21,GiannakisEtAl22,AndressEtAl24,JinEtAl24,StenglEtAl24,BharadwajSreenivasan25,CostaEtAl25,WuEtAl25}, along with applications in diverse domains dealing with complex systems \cites{MezzacapoEtAl15,CostaEtAl19,EngelEtAl19,BharadwajSreenivasan20,Gaitan20,DodinStartsev21,LinEtAl22,PfefferEtAl22,PfefferEtAl23,JosephEtAl23,BharadwajSreenivasan23,TenniePalmer23}.
Here, a primary challenge is that classical dynamical systems are described by generally nonlinear maps (flows) on state space, whereas quantum algorithms proceed by unitary linear operators acting on quantum states and observables.
Being firmly rooted in linear operator theory, the formulation of dynamics based on Koopman and transfer operators provides a natural bridge between classical and quantum systems that can serve as a foundation for building quantum algorithms.

In fact, the mathematical connections between ergodic and quantum theories were already explored in the foundational work of Koopman and von Neumann from the 1930s \cite{Koopman31,KoopmanVonNeumann32,VonNeumann32,VonNeumann32b}.
This work laid the foundations for a family of approaches now known as Koopman--von Neumann representations of classical and semiclassical dynamics that were developed over multiple decades \cites{DellaRiciaWiener66,WilkieBrumer97a,WilkieBrumer97b,Mauro02,BondarEtAl19,Joseph20,JosephEtAl23,StenglEtAl24,Barandes25}.
In its standard form, the Koopman--von Neumann formulation employs a mapping of the isometric evolution of the classical probability density, $p$, in $L^1$ under the transfer operator for measure-preserving dynamics to a unitary evolution of a wavefunction in $L^2$ given by the square root of $p$.
Methods in this direction include quantum algorithms that leverage Koopman--von Neumann embeddings \cites{Joseph20,GiannakisEtAl22,NovikauJoseph25} to simulate the evolution classical states and probability densities on finite-dimensional quantum systems (quantum computers).
Other works have developed ``quantum-inspired'' computational techniques that employ properties of spaces of operators in classically-implemented approximation schemes with improved structure preservation  properties (e.g., positivity-preservation) \cites{Giannakis19b,FreemanEtAl23,FreemanEtAl24} and/or computational efficiency \cite{YeLoureiro24}.

In a different context from dynamical systems, tensor networks \cites{Orus19,Banuls23} were originally introduced for calculations in statistical mechanics and many-body quantum physics \cites{Baxter68,AffleckEtAl87,KlumperEtAl93}, and have since found applications in function approximation \cites{Khoromskij11,GarciaRipoll21} and machine learning \cites{StoudenmireSchwab16,HanEtAl18} as a framework for expressing computations in high-dimensional vector spaces.
Tensor networks describe factorizations of such computations in terms of basic algebraic operations such as operator compositions and tensor products which can be carried out at a feasible computational cost when brute-force approaches would have been intractable.
Using tensor networks to express operator-theoretic computations for classical dynamical systems will be a primary focus of this work.

\subsection{Our contributions}

Leveraging mathematical connections between operator-theoretic formulations of classical dynamics and quantum theory, we introduce an approach for approximating the Koopman evolution of observables in continuous-time systems through tensor networks.
Our approach begins by a spectrally consistent approximation of the generator of the unitary Koopman group of a measure-preserving ergodic flow by skew-adjoint, diagonalizable operators $W_\tau$ with discrete spectra \cites{DasEtAl21,GiannakisValva24,GiannakisValva25}, acting on a one-parameter family of RKHSs $\mathcal H_\tau$, $\tau>0$.
In this work, we build the spaces $\mathcal H_\tau$ so as to have reproducing kernel Hilbert algebra (RKHA) structure \cites{DasGiannakis23,DasEtAl23,GiannakisMontgomery25}---this means that $\mathcal H_\tau$ admits a bounded operator $\Delta \colon \mathcal H_\tau \to \mathcal H_\tau \otimes \mathcal H_\tau$ mapping into its two-fold tensor product, whose adjoint  implements pointwise multiplication.
In other words, $\mathcal H_\tau$ is a commutative Banach algebra with respect to pointwise multiplication that has additional coalgebra structure with $\Delta$ as the comultiplication operator.
We lift the unitary group generated by $\mathcal H_\tau$ to a unitary group acting multiplicatively with respect to the tensor product on the Fock space, $F(\mathcal H_\tau)$, generated by $\mathcal H_\tau$.
The Fock space $F(\mathcal H_\tau)$ plays the role of the Hilbert space of a many-body quantum system \cite{Lehmann04} that allows us to represent (i) pointwise evaluation functionals of $\mathcal H_\tau$ as pure quantum states on $F(\mathcal H_\tau)$ with state vectors that project non-trivially to every tensor power $\mathcal H_\tau^{\otimes n} \subset F(\mathcal H_\tau)$, $n \in \mathbb N$;  (ii) classical observables in $\mathcal H_\tau$ by quantum observables (self-adjoint operators) acting on any given grading $\mathcal H_\tau^{\otimes n}$ as amplifications of multiplication operators.
See \cref{fig:dia1} for a schematic illustration of these constructions.

\begin{figure}
	\centering
	\includegraphics[draft=false, width=0.9\linewidth]{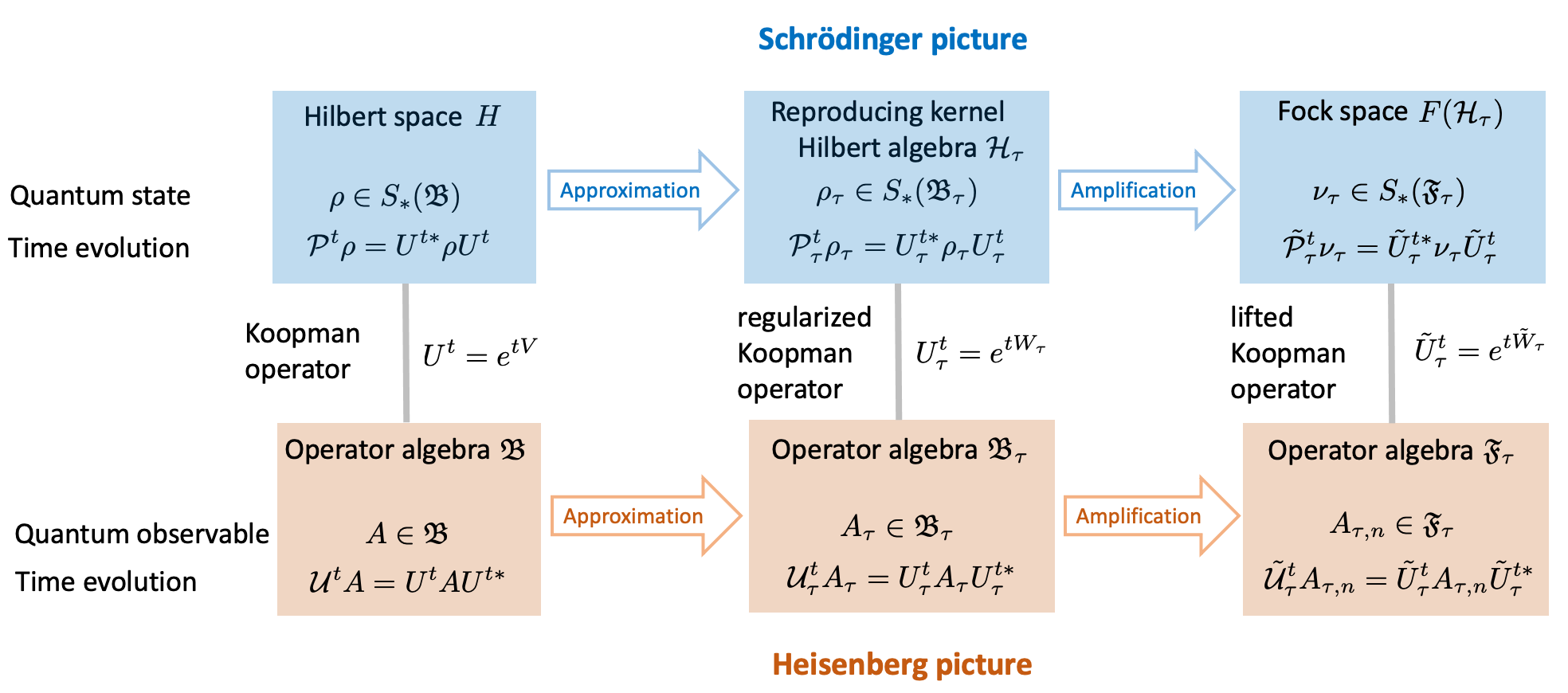}
	\caption{Evolution of quantum states (Schr\"odinger picture) and observables (Heisenberg picture) under the tensor network approximation framework described in this paper.
		Given a classical probability density $p \in S_*(\mathfrak A)$, our scheme approximates the evolution of a vector state $\rho = \langle p^{1/2}, \cdot \rangle_H p^{1/2} \in S_*(\mathfrak B)$ under the induced transfer operator $\mathcal P^t$ by a vector state $\rho_\tau = \langle \xi_\tau, \cdot \rangle_{\mathcal H_\tau} \xi_\tau$ on an RKHA $\mathcal H_\tau$ under an evolution operator $\mathcal P^t_\tau$ induced from a regularized Koopman operator $U^t_\tau$ on $\mathcal H_\tau$ with pure point spectrum.
		Similarly, the evolution of quantum observables $A \in \mathfrak B$ under the induced Koopman operator $\mathcal U^t$ is approximated by a regularized operator $\mathcal U^t_\tau$ induced from $U^t_\tau$ acting on smoothed quantum observables $A_\tau = K_\tau A K_\tau^*$, where $K_\tau\colon H \to \mathcal H_\tau$ is a kernel integral operator associated with $\mathcal H_\tau$.
		Our scheme then dilates the regularized Koopman dynamics on $\mathcal H_\tau$ to unitary dynamics $\tilde U^t \colon F(\mathcal H_\tau) \to F(\mathcal H_\tau)$ on the Fock space generated by $\mathcal H_\tau$, such that $\tilde U^t$ acts multiplicatively (tensorially) on the tensor algebra $T(\mathcal H_\tau) \subset F(\mathcal H_\tau)$.
		Furthermore, the quantum observables $A_\tau$ are amplified using the coalgebra structure of $\mathcal H_\tau$ to quantum observables $A_{\tau, n}$ that act on any grading $\mathcal H_\tau^{\otimes (n+1)} \subset F(\mathcal H_\tau)$.
		The state vector $\xi_\tau$ is also dilated to a vector $\eta_\tau \in F(\mathcal H_\tau)$ that projects non-trivially on $\mathcal H^{\otimes n}_\tau$ for every $n \in \mathbb N$, with an associated quantum state $\nu_\tau = \langle \eta_\tau, \cdot\rangle_{F(\mathcal H_\tau)} \eta_\tau$.
		The expectation of $A_{\tau, n}$ with respect to the time-evolved quantum state $\tilde\rho_\tau$ under the transfer operator $\tilde{\mathcal P}^t_\tau$ induced by $\tilde U^t_\tau$ is then used to approximate the corresponding evolution under the true dynamics on $\mathfrak B$ using a tree tensor network construction (see \cref{fig:network}).}
	\label{fig:dia1}
\end{figure}

Through this approach, we approximate the Koopman evolution of observables by a tree tensor network on $F(\mathcal H_\tau)$ that can be implemented efficiently using functional programming techniques.
Effectively, the tensor network approximates Koopman evolution on a $(2d+1)^n$-dimensional tensor product subspace of $\mathcal H_\tau^{\otimes n}$ generated by a collection of $d$ complex-conjugate pairs of eigenfunctions of $W_\tau$ together with a constant eigenfunction.
This allows devoting resources to compute a modest number ($2d+1$) of eigenfunctions of $W_\tau$, and algebraically amplifying these eigenfunctions to capture information from a high-dimensional tensor product space (vs.\ brute-force computation of large dictionaries of Koopman eigenfunctions).
Moreover, the tensor network scheme is positivity-preserving and is asymptotically consistent in a suitable joint limit of $n \to \infty$ and $\tau \to 0$.

The plan of the paper is as follows.
\Cref{sec:background} contains background on operator methods for classical dynamics, along with an illustrative example involving a circle rotation that we will reference in later sections of the paper.
This is followed by a high-level overview of Koopman von--Neumann representations and tensor networks in \cref{sec:quantum}, and a survey of relevant definitions and constructions of RKHAs in \cref{sec:rkha}.
In \cref{sec:tensornet,sec:num_impl}, we develop our tensor network approximation scheme and its numerical implementation.
In \cref{sec:pointwise} we specialize the schemes from \cref{sec:tensornet} to approximation of pointwise (as opposed to statistical) evolution of observables and study their associated convergence properties.
One of our main theoretical results, \cref{thm:torus_rotation_convergence}, is included in that section.
\Cref{sec:experiments} presents applications to two dynamical systems on the 2-torus: an ergodic rotation with pure point Koopman spectrum and a Stepanoff flow \cite{Oxtoby53} with weak topological mixing (absence of non-constant continuous Koopman eigenfunctions).
Concluding remarks and perspectives on future work are included in \cref{sec:conclusions}.
\Cref{app:symbols} contains a list of the main symbols used in the paper.
\Cref{app:spectral_approx} contains a description of the spectral approximation techniques for Koopman operators used in this paper.
Details on numerical implementation and pseudocode are included in \cref{app:numerical}.

\section{Background on operator methods for dynamical systems and a motivating example}
\label{sec:background}

We begin with a brief overview of operator-theoretic techniques for classical dynamical systems (\cref{sec:feature_extraction}) and a discussion of the analytical and numerical challenges faced by these methods (\cref{sec:challenges}).
In \cref{sec:circlerot}, we illustrate some of these challenges using a Galerkin-type discretization of the Koopman generator for a circle rotation as an example.
The circle rotation will serve as a running example when we discuss quantum mechanical approximation schemes in subsequent sections of the paper.

\subsection{Feature extraction and prediction by spectral decomposition of evolution operators}
\label{sec:feature_extraction}

A common aspect of many operator methods is computation of spectral data.
As a concrete example, consider a measure-preserving flow $\Phi^t\colon X \to X$ on a state space $X$ in continuous time $t \in \mathbb R$, with an invariant probability measure $\mu$.
The time-$t$ Koopman operator,
\begin{displaymath}
	U^t\colon L^p(\mu) \to L^p(\mu), \quad U^t f = f \circ \Phi^t,
\end{displaymath}
acts as an isometric isomorphism of the $L^p(\mu)$ spaces with $p \in [1, \infty]$ and its spectrum $\sigma(U^t)$ is a subset of the unit circle in the complex plane.
Moreover, for $p \in [1, \infty)$, $ \{ U^t \}_{t \in \mathbb R}$ is strongly continuous on $L^p(\mu)$ and, by the Hille-Yosida theorem, completely characterized by its generator.
The latter, is a closed (typically unbounded) operator $V\colon D(V) \to L^p(\mu)$ defined on a dense subspace $D(V)\subseteq L^p(\mu)$ via the norm limit
\begin{equation}
	\label{eq:generator}
	V f = \lim_{t\to 0} \frac{U^t f - f}{t}.
\end{equation}
The generator reconstructs the Koopman operator at any time $t \in \mathbb R$ through the functional calculus, $U^t = e^{t V}$.
As a result, every eigenfunction $h$ of $U^t$ has periodic evolution with period $2 \pi / \omega$, where $\omega \in \mathbb R$ is an eigenfrequency satisfying $U^t h = e^{i\omega t} h$.

A collection $ \{ h_1, \ldots, h_N \}$ of such eigenfunctions with corresponding eigenfrequencies $\omega_1, \ldots, \omega_N$ defines, for an appropriate normalization, a feature map $\vec h\colon X \to \mathbb C^N$ whose range is a subset of the $N$-dimensional torus $\mathbb T^N \subset \mathbb C^N$.
This map intertwines $\Phi^t$ with a rotation system $R^t_\alpha\colon \mathbb T^N \to \mathbb T^N$.
That is, we have
\begin{equation}
	\label{eq:rot_semiconj}
	\vec h \circ \Phi^t = R^t_\alpha \circ \vec h,
\end{equation}
$\mu$-a.e., where $R^t_\alpha(x) = (x + \alpha t) \mod 2 \pi$ and $\alpha = (\omega_1, \ldots, \omega_N).$ In other words, $\vec h$ has the property of rectifying the evolution of state space dynamics to a rotation with constant frequency parameters $\omega_j$ in feature space $\mathbb C^N$; see, e.g., \cite{FroylandEtAl21}*{Fig.~4} for an illustration.
This dynamical rectification property likely contributes to the physical interpretability of features based on eigenfunctions or approximate eigenfunctions of Koopman and transfer operators found in application domains spanning molecular dynamics \cite{SchutteEtAl01}, fluid dynamics \cite{Mezic13}, climate dynamics \cite{FroylandEtAl21}, neuroscience \cites{BruntonEtAl16b,MarrouchEtAl19}, and energy system science \cite{SusukiEtAl16}, among other fields.
In addition, observables $f \in L^p(\mu)$ that are well-represented by linear combination of Koopman eigenfunctions, $f \approx \sum_{j=0}^J c_j h_j$ for some $J \in \mathbb N$ and $c_j \in \mathbb C$, are amenable to prediction by leveraging the known evolution of the eigenfunctions,
\begin{equation}
	\label{eq:eig_approx}
	U^t f \approx \sum_{j=0}^J c_j e^{i\omega_j t} h_j.
\end{equation}
See, e.g., \cite{Giannakis19} for numerical examples and \cite{KoltaiKunde24} for a recent study that explores connections between the Wiener filter for least-squares prediction and Koopman eigenfunctions.
Eigenfunctions of the transfer operator $P^t \colon L^p(\mu) \to L^p(\mu)$, where $P^t f = f \circ \Phi^{-t}$ under invertible dynamics, are similarly useful for feature extraction and for approximating the evolution of measures $\nu$ with densities $\frac{d\nu}{d\mu} \in L^p(\mu)$.

\subsection{Practical challenges}
\label{sec:challenges}

The useful properties of Koopman and transfer operator eigenfunctions for feature extraction and prediction tasks outlined above provides ample motivation for developing algorithms and devoting computational resources for their computation.
Still, in practical applications there is a number of challenges that prevent a straightforward implementation of this program.

\subsubsection{Continuous spectrum}

At a fundamental level, Koopman operators associated with complex dynamics seldom possess sufficiently rich sets of eigenfunctions in function spaces such as $L^p(\mu)$ which are typically targeted by data-driven algorithms.
Indeed, by a fundamental result in ergodic theory (e.g., Mixing Theorem in \cite{Halmos56}*{p.~39}), the measure-preserving flow $\Phi^t$ is weak mixing if and only if the Koopman operator $U^t$ on $L^2(\mu)$ has a simple eigenvalue equal to 1, with $\mu$-a.e.\ constant corresponding eigenfunctions, and no other eigenvalues.
It follows that for weak-mixing systems there are no non-trivial analogs of the feature map from~\eqref{eq:rot_semiconj} and the prediction model~\eqref{eq:eig_approx}.
Thinking, intuitively, of weak mixing as a hallmark of dynamical complexity in a measure-theoretic sense, it consequently follows that in many real-world applications one has to seek generalizations of schemes based on pure Koopman/transfer operator eigenfunctions.

In response, a major focus of recent methodologies for Koopman operator approximation has been to consistently approximate the spectral measures of the operators, which typically have both atomic and continuous components, by discrete spectral measures that are amenable to numerical approximation, e.g., \cites{KordaEtAl20,DasEtAl21,ColbrookTownsend24,RosenfeldEtAl22,GiannakisValva24,GiannakisValva25}.
In broad terms, these methods yield versions of \eqref{eq:rot_semiconj} and~\eqref{eq:eig_approx} that are based on eigenfunctions of regularized approximations of $U^t$ (or the generator $V$) that possess non-trivial eigenfunctions and converge spectrally to $U^t$ in a suitable asymptotic limit.
In the transfer operator literature, a popular approach has been to perform approximations in Banach spaces where the operator is quasi-compact \cites{Froyland97,DellnitzEtAl00,BlankEtAl20}, allowing one to extract eigenvalues isolated from the essential spectrum.

\subsubsection{Bias--variance tradeoff}

Even if the system possesses a sufficiently rich set of eigenfunctions to well-represent the observables of interest, expansions such as~\eqref{eq:eig_approx} may require a large number, $J$, of eigenfunctions $h_j$ to achieve the desired approximation accuracy.
In practical applications, these eigenfunctions are seldom known analytically, and have to be numerically approximated, usually via data-driven algorithms.
Given fixed computational and data resources, the quality of these approximations invariably decreases with $J$, which means that  the predictive skill of a model utilizing~\eqref{eq:eig_approx} may actually decrease as $J$ increases.
This is a manifestation of the usual bias--variance tradeoff in statistical learning (e.g., \cite{CuckerSmale01}) that involves balancing bias errors due to using finitely many terms in the eigenfunction expansion for $U^t f$ (which become smaller as $J$ increases) against sampling errors associated with data-driven approximation of these eigenfunctions (which become larger as $J$ increases).
Methods for improving the quality of computed eigenfunctions include smoothing by integral operators \cites{FroylandEtAl21,JungeEtAl22}, learning of well-conditioned approximation dictionaries \cites{BerryEtAl15,Giannakis19,ConstanteAmoresEtAl24}, usage of delay-coordinate maps \cites{ArbabiMezic17,BruntonEtAl17,DasGiannakis19,Giannakis21a}, and residual computations to detect spurious eigenvalues \cites{ColbrookTownsend24}.

\subsubsection{Structure-preserving approximations}

Koopman and transfer operators exhibit important structural properties inherited from the underlying dynamical flow, which distinguish them from general operator (semi)groups.
In applications, it is desirable (and sometimes essential)  that the regularized operators preserve as many of these properties as possible.
Below, we list a few examples that will concern us in this paper:
\begin{enumerate}[wide]
	\item \emph{Unitarity under measure-preserving, invertible dynamics}.
	      In both discrete and continuous time, the Koopman operator $U\colon L^2(\mu) \to L^2(\mu)$, $U f = f \circ \Phi$ associated with a measure-preserving, invertible transformation $\Phi\colon X \to X$ is unitary.
	      In continuous time, the generator $V$ of the Koopman group $ \{ U^t \}_{t \in \mathbb R}$ on $L^2(\mu)$, is skew-adjoint, $V^* = - V$, by the Stone theorem on one-parameter unitary groups.
	\item \emph{Multiplicativity.} Being a composition operator, the Koopman operator is multiplicative on spaces of observables that are closed under pointwise multiplication.
	      As a concrete example, the Koopman operator $U$ on $L^\infty(\mu)$ satisfies
	      \begin{equation}
		      \label{eq:mult_koopman}
		      U(fg) = (U f)(U g), \quad \forall f, g \in L^\infty(\mu).
	      \end{equation}
	      In continuous time, a necessary and sufficient condition for a skew-adjoint operator $V$ on $L^2(\mu)$ to be the generator of a unitary Koopman group $ \{ U^t \}_{t\in \mathbb R}$ is that it satisfies the Leibniz rule on a suitable subspace of its domain \cite{TerElstLemanczyk17},
	      \begin{equation}
		      \label{eq:leibniz}
		      V(fg) = (V f) g + f (V g), \quad \forall f, g \in D(V) \cap L^\infty(\mu).
	      \end{equation}
	      An important consequence of~\eqref{eq:mult_koopman} and~\eqref{eq:leibniz} is that the point spectrum of $U^t$ is a multiplicative subgroup of the unit circle and the point spectrum of $V$ is an additive subgroup of the imaginary line.
	\item \emph{Positivity preservation.} A related property to multiplicativity is that the Koopman operator is positivity-preserving.
	      That is, for every $f \in L^p(\mu)$ such that $ f \geq 0$ $\mu$-a.e., we have $ Uf \geq 0$ $\mu$-a.e.
	      Analogous results hold for continuous-time systems as well as the transfer operator.
\end{enumerate}

In continuous-time systems, unitarity-preserving approximation techniques include smoothing the generator \cite{DasEtAl21} or its resolvent \cites{GiannakisValva24,GiannakisValva25} by Markovian kernel integral operators.
For example, given a family of symmetric kernel functions $k_\tau \colon X \times X \to \mathbb R_+$ of sufficient regularity, parameterized by $\tau>0$, and a corresponding family of self-adjoint kernel integral operators $\mathcal K_\tau \colon L^2(\mu) \to L^2(\mu)$, $\mathcal K_\tau f = \int_X k_\tau(\cdot, x) f(x) \, d\mu(x)$, we have that $V_\tau := \mathcal K_\tau V \mathcal K_\tau$ is a compact operator \cite{DasEtAl21}.
This operator is automatically skew-adjoint, and thus generates a strongly continuous unitary evolution group $ \{ U^t_\tau = e^{t V_\tau} \}_{t \in \mathbb R}$ on $L^2(\mu)$.
For suitable kernel constructions, the unitaries $U^t_\tau$ converge to $U^t$ strongly as $\tau \to 0$, and the corresponding spectral measures also converge in a suitable sense.
In discrete-time systems, recently developed variants of DMD \cites{BaddooEtAl23,Colbrook23} also preserve unitarity of the Koopman operator under measure-preserving invertible dynamics.

To our knowledge, there are no DMD-type Koopman operator approximation methods that are positivity preserving (but see below for non-commutative formulations) while offering convergence guarantees.
In essence, failure to preserve positivity stems from the fact that orthogonal projections of positive observables onto the DMD approximation subspace are not necessarily positive.
Similarly, no practical methods are known to us that preserve the Leibniz rule for the generator in continuous-time systems.
Very recently, multiplicative dynamic mode decomposition (MultDMD) was introduced as a method that preserves the multiplicative structure on the Koopman operator on spaces of finitely-sampled observables using constrained optimization methods \cite{BoulleColbrook24}.  In the setting of transfer operators, the Ulam method \cite{Ulam64} is a commonly employed scheme that is automatically positivity preserving by virtue of employing characteristic functions as basis elements.
However, convergence results for Ulam-type schemes (e.g., \cites{Li76,Froyland99,JungeKoltai09,BlankEtAl20}) are generally limited to specific classes of dynamical systems such as expanding interval maps, and involve spaces such as $L^1(\mu)$ or anisotropic Banach spaces that do not have Hilbert space structure (the latter being highly useful in numerical applications).

\subsection{Circle rotation}
\label{sec:circlerot}

As a basic example illustrating the behavior of different Koopman approximation schemes, we  consider a linear rotation $\Phi^t \colon X \to X$ on the circle $X=\mathbb T$,
\begin{equation*}
	\Phi^t(\theta) = \theta + \alpha t \mod 2 \pi,
\end{equation*}
where $\alpha>0$ is a frequency parameter.
This flow has the normalized Haar measure as its unique invariant ergodic Borel probability measure, $\mu$.
Moreover, the Koopman operator $U^t$ is diagonal in the Fourier basis $\{ \phi_j\colon \theta \mapsto e^{ij\theta} \}_{j \in \mathbb Z} $ of $L^2(\mu)$,
\begin{equation*}
	U^t \phi_j = e^{i j \alpha t} \phi_j.
\end{equation*}
The associated skew-adjoint generator $V$ has the space of absolutely continuous functions on $\mathbb T$ as its domain $D(V) \subset L^2(\mu)$, and acts as a weak derivative on this space, $V f(\theta) = \alpha f'(\theta)$ for $\mu$-a.e.\ $\theta \in \mathbb T$.
Moreover, we have the eigendecomposition
\begin{equation*}
	V \phi_j = i j \alpha \phi_j,
\end{equation*}
with eigenfrequencies $j \alpha$ given by integer multiples of the rotation frequency $\alpha$.
By virtue of these facts, the evolution of observables $f \in L^2(\mu)$ of this system can be approximated to any desired degree of accuracy via~\eqref{eq:eig_approx} for the eigenfunctions $\{h_0, h_1, \ldots\} = \{\phi_0, \phi_1, \phi_{-1}, \phi_2, \phi_{-2} \ldots\}$ and eigenfrequencies $\{\alpha_0, \alpha_1, \ldots\} = \{0, \alpha, -\alpha, 2 \alpha, -2 \alpha, \ldots \}$.

In this example, we set the frequency parameter $\alpha=1$ and consider the evolution of an observable $f$ in the family of von Mises probability density functions on the circle.
The von Mises probability density $p_{\mu, \kappa} \in C^\infty(\mathbb T)$ is parameterized by a sharpness parameter $\kappa>0$, a parameter $\mu \in [0, 2\pi)$ controlling the mode (location of maximum value) of the distribution, and is defined as
\begin{equation}
	\label{eq:von_mises_circle}
	p_{\mu, \kappa}(\theta) = \frac{e^{\kappa\cos(\theta-\mu)}}{I_0(\kappa)}.
\end{equation}
Here, $I_\nu \in C^\infty(\mathbb R)$, $\nu \geq 0$, is the modified Bessel function of the first kind of order $\nu$.
We choose $f = p_{\pi, 6}$ as an example of a strictly positive observable with an infinite Fourier series to illustrate positivity preservation, or lack thereof, by numerical techniques that utilize approximation spaces spanned by finitely many Fourier functions.

In \cref{fig:evo_circlerot}, we compare the Koopman evolution of $f$, which we denote by $f^{(t)}_\text{true} := U^t f$, with (i) a ``classical'' approximation, $f^{(t)}_\text{cl}$, obtained from a DMD-type method; (ii) a ``quantum mechanical'' approximation, $f^{(t)}_\text{qm}$, derived from a variant of the schemes from \cite{Giannakis19b,FreemanEtAl23,GiannakisEtAl22}; and (iii) an approximation $f^{(t)}_\text{Fock}$ obtained from the tensor network method described in this paper that employs a many-body quantum system defined on a Fock space.
These examples are essentially one-dimensional versions of the 2-torus experiments presented in \cref{sec:experiments} below; see, in particular, \cref{sec:evo} for detailed descriptions of the procedures employed in the construction of $f^{(t)}_\text{cl}$, $f^{(t)}_\text{qm}$, and $f^{(t)}_\text{Fock}$.

\begin{figure}
	\centering
	\includegraphics[draft=false, width=0.9\linewidth]{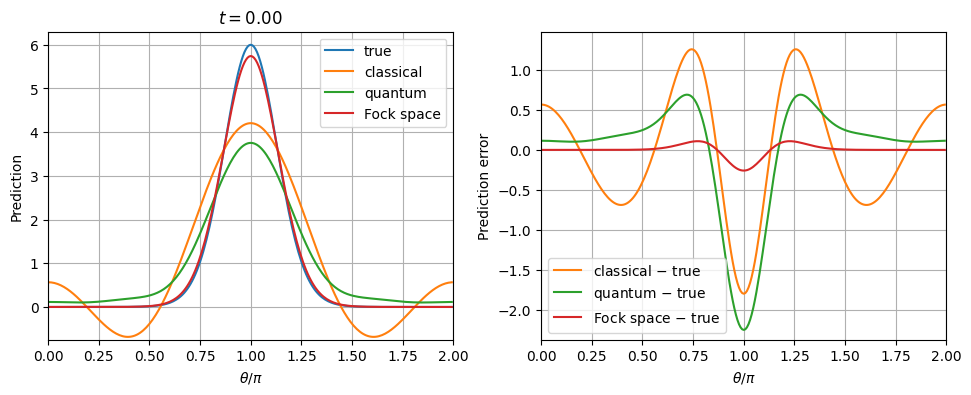}
	\includegraphics[draft=false, width=0.9\linewidth]{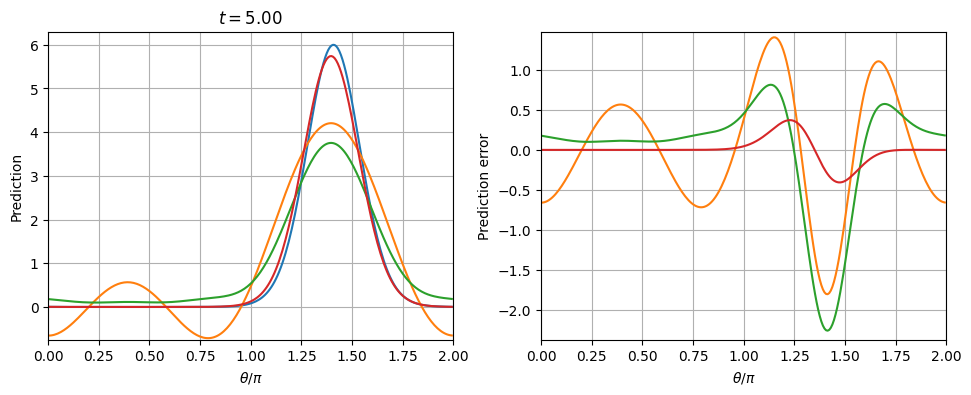}
	\caption{\label{fig:evo_circlerot}Time evolution of the von Mises density $p_{\mu, \kappa}$ with $\mu=\pi$ and $\kappa = 6$ under the circle rotation with frequency $\alpha=1$ ($f^{(t)}_\text{true}$; blue), the classical approximation ($f^{(t)}_\text{cl}$; orange), the quantum mechanical approximation ($f^{(t)}_\text{qm}$; green), and the tensor network/ Fock space approximation ($f^{(t)}_\text{Fock}$; red).
		The top and bottom rows show the evolution times $t= 0$ and $t=5$, respectively.
		The panels in the right-hand column show the pointwise approximation error $f^{(t)}_\text{approx} - f^{(t)}_\text{true}$ for each evolution time and approximation method.}
\end{figure}

Considering first the classical approximation, the function $f^{(t)}_\text{cl}$ is essentially equivalent to the Koopman eigenfunction expansion from~\eqref{eq:eig_approx}, modulo approximation errors in the eigenpairs $(\omega_j, h_j)$.
More specifically, for a regularization parameter $\tau >0$, we have
\begin{equation*}
	f^{(t)}_\text{cl} := \sum_{j=0}^{J} c_j e^{i \omega_{j,\tau} t} \zeta_{j,\tau},
\end{equation*}
where $\omega_{j,\tau} \in \mathbb R$ are approximations of the eigenfrequencies $\omega_j$, $\zeta_{j,\tau} \in \mathcal H_\tau$ are approximations of the eigenfunctions $h_j$ in an RKHA $\mathcal H_\tau \subset C(\mathbb T)$, and  $c_j \in \mathbb C$ are expansion coefficients determined from the Fourier series of $p_{\mu, \kappa}$ (see~\eqref{eq:von_mises_fourier}).
This approximation can be thought of in a similar vein as kernelized variants of DMD or generator DMD \cite{WilliamsEtAl15,KlusEtAl20}.
Here, the eigenpairs $(\omega_{j,\tau}, \zeta_{j,\tau})$ were computed using a variant of the techniques from \cite{GiannakisValva25}.
The specific algorithm is described in \cref{app:compact_res} and will be also employed in other numerical experiments presented in this paper.
In a nutshell, this method builds a family of skew-adjoint operators $V_\tau$, $\tau >0$, on $L^2(\mu)$ with compact resolvent (and thus an associated orthonormal eigenbasis), that converge spectrally in a suitable limit of $\tau\to 0^+$, even if $V$ has continuous or mixed spectrum.
The approximate Koopman eigenfunctions $\zeta_{j,\tau}$ are eigenfunctions of an operator $W_\tau$ acting on $\mathcal H_\tau$ that is equivalent to $V_\tau$ under a natural isometric embedding of $L^2(\mu)$ into $\mathcal H_\tau$.

\Cref{tab:circlerot_spec} lists the approximate eigenfrequencies $\omega_{j,\tau}$ with $\lvert j \rvert \leq 6$ computed from a Galerkin approximation of $V_\tau$ in the Fourier basis of $L^2(\mu)$ for the regularization parameter $\tau = 0.005$.
The table also lists additional spectral information such as the Dirichlet energies of the eigenfunctions $\zeta_{j,\tau}$ with respect to the norm of $\mathcal H_\tau$ (a measure of spatial variability akin to a Sobolev norm; see \cref{sec:tensornet}) and the eigenvalues of an auxiliary compact operator, $Q_{z,\tau}$, employed in the construction of $V_\tau$ (see \cref{app:compact_res}).
Notice that the relative frequency error $\lvert \omega_{j,\tau} - \omega_j \rvert / \lvert \omega_j \rvert$ becomes larger as the Dirichlet energy of $\zeta_{j,\tau}$ increases---this is a generic feature that can result to buildup of phase errors as the evolution time $t$ increases when using eigenfunctions with large energy.

Arguably, the need for regularization is obviated in the case of the circle rotation (and the errors it imparts on the spectrum can be avoided) since $V$ is already diagonal in the Fourier basis of $L^2(\mu)$ and has discrete spectrum.
In our example, the approximate eigenpairs $(\omega_{j,\tau}, \zeta_{j,\tau})$ can be thought of as stand-ins for solutions of potentially expensive eigenvalue problems with regularization necessitated by the presence of continuous spectrum and/or non-isolated eigenvalues of the generator.

\begin{table}
	\centering
	\caption{\label{tab:circlerot_spec}Representative spectral data for the regularized generator $V_\tau$ of the circle rotation computed via the method in \cref{app:compact_res}.
		The table lists the first seven eigenfrequencies, $\omega_{j,\tau}$, the Dirichlet energies of the corresponding generalized eigenfunctions, $\mathcal E(\tilde u_{j,z,\tau})$, and the eigenvalues, $\beta_{j,z,\tau}$, of the $Q_{z,\tau}$ operator used in the construction of $V_\tau$.
		See \eqref{eq:q_eig} and \eqref{eq:gev} for definitions of $\tilde u_{j,z,\tau}$ and $\beta_{j,z,\tau}$.
		The RKHA parameters are $(p,\tau) = (0.75, 0.005)$, the resolvent parameter is $z=0.01$, and the number of Fourier functions used for Galerkin approximation is $m=128$.
		In this table, and in \cref{sec:circlerot}, we suppress the resolvent parameter $z$ in our notation for $V_\tau$ and $\omega_{j,\tau}$.
	}
	\begin{tabular}{lSSS}
		\hline\hline
		                        & \multicolumn{1}{c}{Eigenfrequency}    & \multicolumn{1}{c}{Dirichlet} energy                  & \multicolumn{1}{c}{$Q_{z,\tau}$ eigenvalue} \\
		\multicolumn{1}{c}{$j$} & \multicolumn{1}{c}{$\omega_{j,\tau}$} & \multicolumn{1}{c}{$\mathcal E_\tau(\zeta_{j,\tau})$} & \multicolumn{1}{c}{$\beta_{j,\tau} / i$}    \\
		\hline
		0                       & 0.0000                                & 0.0000                                                & 0.0000                                      \\
		1                       & 1.0050                                & 0.0050                                                & 0.9949                                      \\
		2                       & -1.0050                               & 0.0050                                                & -0.9949                                     \\
		3                       & 2.0169                                & 0.0084                                                & 0.4958                                      \\
		4                       & -2.0169                               & 0.0084                                                & -0.4958                                     \\
		5                       & 3.0344                                & 0.0115                                                & 0.3296                                      \\
		6                       & -3.0344                               & 0.0115                                                & -0.3296                                     \\
		\hline\hline
	\end{tabular}
\end{table}

The classical approximation $f^{(t)}_\text{cl}$ in \cref{fig:evo_circlerot} is based on $J = 2d$ nonconstant eigenfunctions with $d=2$; this represents a severe amount of truncation whose purpose is to highlight approximation errors inherent in applications with higher state space dimension and/or dynamical complexity.
The effects of truncation are clearly evident in the discrepancy $f^{(t)}_\text{cl} - f^{(t)}_\text{true}$ (depicted in the right-hand column of \cref{fig:evo_circlerot}).
Besides pointwise reconstruction errors, an important deficiency in $f^{(t)}_\text{cl}$ is that it exhibits oscillations to negative values taking place at the tails of the von Mises density function.
Oscillations to values outside the range of $f$ are a generic feature of orthogonal $L^2$ projections, with the Gibbs phenomenon of the Fourier series approximation of the tophat function being one of the most characteristic examples.

Failure to preserve positivity, or, more generally, respect the admissible range of values of observables, can be especially detrimental when modeling physical systems with sign-definite quantities such as energy or temperature.
In \cref{fig:evo_circlerot}, the quantum mechanical approximation $f^{(t)}_\text{qm}$ is seen to preserve positivity of the von Mises density for all evolution times $t$.
We will discuss this method in more detail in \cref{sec:quantum} below, but for now we mention that the quantum mechanical approach preserves positivity by approximating the pointwise values $f^{(t)}_\text{true}(\theta)$ by quantum expectations of a quantum observable representing $f$.
This quantum observable is a projected multiplication operator (a Toeplitz operator) on $L^2(\mu)$ which is positive whenever the underlying classical observable is positive.
At the same time, $f^{(t)}_\text{qm}$ appears to be more diffusive than its classical counterpart $f^{(t)}_\text{cl}$ and thus less skillful in capturing the peak of the von Mises distribution.
In this example, both methods utilize the span of $\zeta_{0,\tau}, \ldots, \zeta_{4,\tau}$ as the approximation space so there appears to be a tradeoff between positivity preservation and reconstruction of extremal values.

The tensor network approximation $f^{(t)}_\text{Fock}$ put forward in this paper addresses this tradeoff by lifting the system to a Fock space generated by spans of finitely many eigenfunctions $\zeta_{j,\tau}$.
In the Fock space approach, $W_\tau$ is dilated to a free derivation satisfying the Leibniz rule with respect to the tensor product.
Moreover, $f$ is represented by an amplified multiplication operator that captures higher-order frequency content present in the pointwise products $\zeta_{j,\tau}$ and encoded as tensor products in the Fock space.
In \cref{fig:evo_circlerot}, $f^{(t)}_\text{Fock}$ is built using the Fock space generated by the span of $\zeta_{0,\tau}, \ldots, \zeta_{4,\tau}$.
It is seen to attain the highest accuracy compared to the classical and quantum mechanical approximations, while being positivity-preserving.

\section{Operator methods and quantum theory}
\label{sec:quantum}

As noted in \cref{sec:intro_quantum}, the operator-theoretic formulation of ergodic theory has close mathematical connections with quantum theory.
In this section, we give a brief survey of relevant quantum mechanical techniques that will be useful for the development of our tensor network scheme later on.
Specifically, in \cref{sec:koopman_von_neumann} we describe the Koopman-von Neumann representation of classical dynamics as it pertains to  measure-preserving ergodic flows studied in this paper, and in \cref{sec:kvn_finite_dim} we outline positivity-preserving finite-dimensional approximation schemes based on this framework.
In \cref{sec:tensornet_overview} we summarize a few basic definitions and constructions from the field of tensor networks.

\subsection{Koopman--von Neumann representation of classical dynamics}
\label{sec:koopman_von_neumann}

A useful starting point for building quantum mechanical representations of classical dynamics is to associate to the system under study abelian and non-abelian algebras of observables that we will consider as being ``classical'' and ``quantum'' respectively.

In the context of the measure-preserving flow $\Phi^t\colon X \to X$ from \cref{sec:feature_extraction}, a natural abelian algebra is the space $\mathfrak A := L^\infty(\mu)$ of essentially bounded observables with respect to the invariant measure $\mu$.
This space is a von-Neumann algebra \cite{Takesaki01} with respect to pointwise function multiplication and complex conjugation; that is, it is a unital $C^*$-algebra,
\begin{displaymath}
	\lVert fg\rVert_{\mathfrak A} \leq \lVert f\rVert_{\mathfrak A} \lVert g\rVert_{\mathfrak A}, \quad \lVert f^* f\rVert_{\mathfrak A} = \lVert f\rVert_{\mathfrak A}^2,
\end{displaymath}
with unit given by the constant function $\bm 1$ equal to 1 $\mu$-a.e., and has a Banach space predual, $\mathfrak A_* := L^1(\mu)$.
The state space of $\mathfrak A$, which we denote as $S(\mathfrak A)$, is the set of (automatically continuous) positive, unital functionals, i.e.,  functionals $\omega \colon \mathfrak A \to \mathbb C$ satisfying $\omega \bm 1 = 1$ and $\omega(f^* f) \geq 0 $ for every $f \in \mathfrak A$.
Every probability density $p \in \mathfrak A_*$ induces a state $\mathbb E_p \in S(\mathfrak A)$ that acts by expectation, $\mathbb E_p f = \int_X f p \, d\mu$.
Such states $\mathbb E_p$ are called normal states, and we denote the set of probability densities $p$ that induce them  by $S_*(\mathfrak A) \subset \mathfrak A_*$.
For any $t \in \mathbb R $, the Koopman operator acts as an isomorphism of $\mathfrak A$; in particular,
\begin{displaymath}
	U^t(fg) = (U^t f)(U^t g), \quad U^t(f^*) = (U^t f)^*, \quad \lVert U^t f\rVert_{\mathfrak A} = \lVert f \rVert_{\mathfrak A}.
\end{displaymath}
Moreover, by duality with the transfer operators $P^t \colon \mathfrak A_* \to \mathfrak A_*$, we have
\begin{equation}
	\label{eq:koop_transf_duality}
	\mathbb E_p(U^t f) = \mathbb E_{P^t p} f, \quad \forall f \in \mathfrak A, \quad \forall p \in S_*(\mathfrak A).
\end{equation}

Next, as a non-abelian algebra, we consider the space $\mathfrak B := B(H)$ of bounded linear operators on $H:=L^2(\mu)$, equipped with the operator norm.
This space is a von Neumann algebra with respect to operator composition and adjunction,
\begin{displaymath}
	\lVert AB\rVert_{\mathfrak B} \leq \lVert A\rVert_{\mathfrak B} \lVert B\rVert_{\mathfrak B}, \quad \lVert A^* A\rVert_{\mathfrak B} = \lVert A\rVert^2_{\mathfrak B},
\end{displaymath}
and has the space of trace class operators on $H$, $\mathfrak B_* := B_1(H)$, as its predual.
We define the state space $S(\mathfrak B)$ analogously to $S(\mathfrak A)$.
Every positive element $\rho \in \mathfrak B_*$ of unit trace induces a normal state $\mathbb E_\rho \in S_*(\mathfrak B)$ such that $\mathbb E_\rho A = \tr(\rho A)$.
Such elements $\rho$ are known as density operators, and can be viewed as analogs of the classical probability densities in $S_*(\mathfrak A)$.
We denote the set of density operators on $H$ as $S_*(\mathfrak B) \subset \mathfrak B_*$.
The group of unitary Koopman operators $U^t H \to H$ acts on $\mathfrak B$ through the adjoint representation, $\mathcal U^t\colon \mathfrak B \to \mathfrak B$ where $\mathcal U^t A = U^t A U^{t*}$.
This action is an isomorphism of the von Neumann algebra $\mathfrak B$,
\begin{displaymath}
	\mathcal U^t(A B) = (\mathcal U^t A)(\mathcal U^t B), \quad \mathcal U^t(A^*) = (\mathcal U^t A)^*, \quad \lVert \mathcal U^t A\rVert_{\mathfrak B} = \lVert A\rVert_{\mathfrak B}.
\end{displaymath}
Defining $\mathcal P^t\colon \mathfrak B_* \to \mathfrak B_*$, $\mathcal P^t \rho = U^{t*} \rho U^t$, as a non-abelian version of the transfer operator on $\mathfrak A_*$, we have $\mathbb E_\rho(\mathcal U^t A) = \mathbb E_{\mathcal P^t \rho} A$ for every density operator $\rho \in S_*(\mathfrak B)$.
This duality relation is a non-abelian analog of~\eqref{eq:koop_transf_duality}.
In quantum mechanics, the evolution of states under $\mathcal P^t$ and the evolution of observables under $\mathcal U^t$ are known as the Schr\"odinger and Heisenberg pictures, respectively.  From a more abstract point of view, we can view a Koopman operator as an element of the automorphism group of a von Neumann algebra, here either $\mathfrak A$ or $\mathfrak B$, with a natural embedding  $\operatorname{Aut}(\mathfrak A) \hookrightarrow \operatorname{Aut}(\mathfrak B)$ given by the adjoint representation.

To establish a correspondence between the evolution of classical observables in $\mathfrak A$ with the evolution of quantum observables in $\mathfrak B$, we introduce the multiplier representation $\pi\colon \mathfrak A \to \mathfrak B$ that maps $ f \in \mathfrak A$ to the multiplication operator in $\mathfrak B$ that multiplies by that element,
\begin{displaymath}
	(\pi f) g = f g, \quad \forall g \in H.
\end{displaymath}
We also introduce the map $\Gamma\colon S_*(\mathfrak A) \to S_*(\mathfrak B)$ such that
\begin{equation}
	\label{eq:kvn}
	\Gamma(p) =  \langle p^{1/2}, \cdot\rangle_H p^{1/2}.
\end{equation}
Note that $\Gamma(p)$ above is a rank-1 operator that projects along the unit vector $p^{1/2} \in H$; this unit vector plays an analogous role to a quantum mechanical wavefunction and is an example of a Koopman--von Neumann wave \cite{Mauro02}.
The associated state $\mathbb E_{\Gamma(p)} \in S(\mathfrak B)$ is an example of a vector state.

With these definitions, one can verify the following classical--quantum compatibility relations:
\begin{equation}
	\label{eq:classical_quantum_compat}
	\mathbb E_p(U^t f) = \mathbb E_{P^t p} f = \mathbb E_{\Gamma(p)}(\mathcal U^t (\pi f)) = \mathbb E_{\mathcal P^t(\Gamma(p))}( \pi f), \quad \forall f \in \mathfrak A, \quad \forall p \in S_*(\mathfrak B).
\end{equation}
The above, describes an embedding of the statistical evolution of classical observables in the algebra $\mathfrak A$ into an evolution of quantum observables in the operator algebra $\mathfrak B$ induced by the unitary Koopman operators on the Hilbert space $H$.

\subsection{Positivity-preserving projection}
\label{sec:kvn_finite_dim}

Recent work \cites{Giannakis19b,FreemanEtAl23,FreemanEtAl24} has put forward the perspective that quantum-inspired algorithms may provide opportunities for structure-preserving approximations that are not directly available when using classical numerical methods.
For example, \cites{Giannakis19b,FreemanEtAl23} developed a quantum mechanical formulation of Bayesian sequential data assimilation (filtering) that is automatically positivity preserving despite using orthogonal projections that do not preserve the positivity of functions in $H$.

In more detail, given a sequence $\Pi_L\colon H \to H$ of finite-rank orthogonal projections converging, as $L\to \infty$, strongly to the identity, they approximate the dynamical evolution~\eqref{eq:classical_quantum_compat} by approximating the quantum observable $\pi f$ and density operator $\Gamma(p)$ by finite-rank compressions,
\begin{displaymath}
	\pi_L f := \Pi_L (\pi f) \Pi_L, \quad \Gamma_L(p) := \frac{\Pi_L \Gamma(p) \Pi_L}{\tr(\Pi_L \Gamma(p) \Pi_L)},
\end{displaymath}
respectively.
Note that $\pi_L f$ is a positive operator whenever $f$ is a positive element of $\mathfrak A$, whereas a ``classical'' approximation $\pi_L f$ of $f$ by Hilbert subspace projection may not be positive.
The resulting evolution, $\mathbb E_{\Gamma_L(p)}(\mathcal U^t(\pi_L f))$, can be implemented using numerical linear algebra methods and converges to the true evolution $\mathbb E_p(U^tf)$ as $L\to\infty$.
In essence, embedding the classical evolution into a quantum one provides access to numerical approximation algorithms that are positivity preserving while enjoying the benefits of Hilbert space structure of the underlying space $H$.

\begin{example*}[Circle rotation]
	The quantum mechanical approximation $f^{(t)}_\text{qm}$ of the evolution of the von Mises density under the circle rotation in \cref{fig:evo_circlerot} was built using a family of quantum states that approximate pointwise function evaluation.
	For $\theta \in \mathbb T$, we set $p$ to a von Mises density $p_{\theta, \kappa_\text{eval}}$ with large concentration parameter $\kappa_\text{eval}$ so that $\int_{\mathbb T} f p_{\theta, \kappa_\text{eval}} \, d\mu \approx f(\theta)$ well-approximates pointwise evaluation at $\theta$ for continuous functions $f$ on the circle.
	In \cref{fig:evo_circlerot}, we use $\kappa_\text{eval} = 500$ which is the same value as we will use for our 2-torus experiments in \cref{sec:experiments}.
	Moreover, we choose $\Pi_L$ to be the projection operator that maps onto the leading $2d + 1 = 5$ eigenspaces of $V_\tau$ used in the classical approximation $f^{(t)}_\text{cl}$.
	The resulting approximation  $f^{(t)}_\text{qm}$ is of the form
	\begin{equation*}
		f^{(t)}_\text{qm}(\theta) = C \bm\xi_{t,\theta}^\dag \bm M \bm\xi_{t,\theta}.
	\end{equation*}
	Here, $\bm\xi_{t,\theta} \in \mathbb C^L$ is a quantum state vector with components given by the expansion of $p_{\theta, \kappa_\text{eval}}$ in the approximate Koopman eigenbasis $\zeta_{j,\tau}$ after time evolution by multiplication by the phase factors $e^{-i \omega_{j,\tau} t}$.
	Moreover, $\bm M \in \mathbb C^{L\times L}$ is the matrix representation of a multiplication operator associated with the prediction observable $f$, and $C$ is a positive normalization constant; see \cref{sec:qm_model} and \eqref{eq:ft_qm_matrix} for further details.
	Importantly, $\bm M$ is a positive-definite matrix whenever $f$ is a positive function, leading to positivity preservation of the von Mises prediction observable $f = p_{\pi,6}$.
	As noted in \cref{sec:circlerot}, a potential deficiency of $f^{(t)}_\text{qm}$ compared to a classical approximation $f^{(t)}_\text{cl}$ utilizing the same subspace of $H$ for approximation, is that it is more diffusive in nature (see again \cref{fig:evo_circlerot}).
	The tensor network approximation scheme developed in this paper overcomes this issue by lifting the problem to a higher-dimensional Fock space generated by the RKHA $\mathcal H_\tau$.
\end{example*}

\subsection{Tensor networks}
\label{sec:tensornet_overview}

Tensor networks provide a powerful diagrammatic framework for expressing computations involving multilinear maps on vector spaces, including operations on quantum mechanical states.
In this subsection, we give a brief outline of the basic constructions that are relevant for the Koopman operator approximation scheme described in this paper.
Further details can be found in one of the many surveys in the literature; e.g., \cites{Orus19,Banuls23} and references therein.

The fundamental objects in a tensor network are multidimensional arrays of finite rank, $\bm A = [A_{i_1 \cdots i_r}]$, where the indices $i_1, \ldots, i_r$ run over specified ranges, $1 \leq i_1 \leq D_i$.
Note that the term ``rank'' here corresponds to the number of indices $r$ of $\bm A$ and not the rank of $\bm A$ as a linear operator.
Diagrammatically, a rank-$r$ tensor is represented by a node in a graph with $r$ (optionally labeled) edges, sometimes referred to as ``legs''.
Scalars (rank-0 tensors) are depicted by nodes with no edges.
See \cref{fig:tensor_ex} for examples.

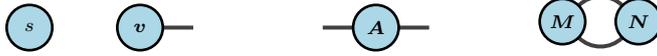
\begin{figure}
	\centering
	\begin{tikzpicture}
		\Vertex[label=$s$]{s}
	\end{tikzpicture}
	\qquad
	\begin{tikzpicture}
		\Vertex[label=$\bm v$]{v}
		\Vertex[x=1, Pseudo]{R}
		\Edge(v)(R)
	\end{tikzpicture}
	\quad
	\begin{tikzpicture}
		\Vertex[label=$\bm A$]{A}
		\Vertex[x=-1, Pseudo]{L}
		\Vertex[x=1, Pseudo]{R}
		\Edge(A)(L)
		\Edge(A)(R)
	\end{tikzpicture}
	\qquad
	\begin{tikzpicture}
		\Vertex[x=0,label=$\bm M$]{M}
		\Vertex[x=1,label=$\bm N$]{N}
		\Edge[bend=45](M)(N)
		\Edge[bend=-45](M)(N)
	\end{tikzpicture}
	\caption{Simple examples of tensor networks.
		From left to right: A scalar, $s$; a rank-1 tensor (vector), $\bm v$; a rank-2 tensor (matrix), $\bm A$; the contraction between two rank-2 tensors, $\bm M$ and $\bm N$.}
	\label{fig:tensor_ex}
\end{figure}

A tensor $\bm A$ represents a multilinear map $A\colon \mathbb C^{D_1} \times \cdots \times \mathbb C^{D_r} \to \mathbb C$ in the standard bases of $\mathbb C^{D_1}, \ldots, \mathbb C^{D_r}$.
Given a collection of $k \leq r$ vectors $\bm v^{(1)}, \ldots, \bm v^{(k)}$ with $\bm v^{(j)} = [v^{(j)}_m] \in \mathbb C^{D_i} $, we can form the multilinear map $B = A(\bm v_1, \ldots, \bm v_k, \cdot, \ldots, \cdot) \colon \mathbb C^{D_{k+1}} \times \cdots \times \mathbb C^{D_r} \to \mathbb C$.
This map is represented by a rank $r-k$ tensor $\bm B = [B_{j_1 \cdots j_{r-k}}]$ obtained by contraction (pairing) of the first $k$ indices of $\bm A$ with the indices of $\bm v^{(1)}, \ldots, \bm v^{(k)}$, i.e., $B_{j_1 \cdots j_{r-k}} = \sum_{i_1=1}^{D_1} \cdots \sum_{i_k=1}^{D_k} A_{i_1 \cdots i_k j_1 \cdots j_{r-k}} v_{i_1}^{(1)} \cdots v_{i_k}^{(k)}$.
Diagrammatically, $\bm B$ is represented by joining the first $k$ edges of $\bm A$ with the edges of $\bm v^{(1)}, \ldots, \bm v^{(k)}$.
More general contractions involving different combinations of indices and/or contractions with rank $>1$ tensors are represented similarly.
For example, joining the legs of two rank-2 tensors $\bm M$ and $\bm N$ yields the scalar $\tr(\bm M^\top \bm N)$.
See again \cref{fig:tensor_ex}.

As noted in \cref{sec:intro}, a major application of tensor networks is representation of states of many-body quantum systems and operations on such states.
Given $n$ distinguishable quantum systems with associated Hilbert spaces $\mathbb H_1, \ldots, \mathbb H_n$ the composite ($n$-body) quantum system is defined on the tensor product Hilbert space $\mathbb H:= \mathbb H_1 \otimes \cdots \otimes \mathbb H_n$.
When the Hilbert spaces $\mathbb H_i$ have finite dimension, the state vector $\xi$ associated with a density operator $\rho = \langle \xi, \cdot \rangle_{\mathbb H} \xi$ on $\mathbb H$ can be represented by a rank-$n$ tensor $\bm \xi$ storing the components of $\xi $ in a tensor product basis of $\mathbb H$.
Similarly, unitary evolution operators, $U^t$, and quantum observables, $A$, on $\mathbb H$ can be represented by rank-$2n$ tensors, $\bm U^t$ and $\bm A$, respectively.
Letting $\bm \xi^\dag$ denote the elementwise complex conjugate of $\bm \xi$ and setting $n=3$ for illustration, the time-evolved quantum expectation $\mathbb E_{\mathcal P^t \rho} A = \langle U^{t*} \xi, A U^{t*} \xi\rangle_{\mathbb H}$ is implemented by the tensor network:
\begin{displaymath}
	\begin{tikzpicture}
		\Vertex[label=$\bm A$]{A}
		\Vertex[x=1, label=$\bm U^{t*}$]{Ustar}
		\Vertex[x=2, label=$\bm \xi$]{xi}
		\Vertex[x=-1, label=$\bm U^{t}$]{U}
		\Vertex[x=-2, label=$\bm \xi^\dag$]{xibar}
		\Edge(A)(Ustar)
		\Edge[bend=45](A)(Ustar)
		\Edge[bend=-45](A)(Ustar)
		\Edge(Ustar)(xi)
		\Edge[bend=45](Ustar)(xi)
		\Edge[bend=-45](Ustar)(xi)
		\Edge(xibar)(U)
		\Edge[bend=45](xibar)(U)
		\Edge[bend=-45](xibar)(U)
		\Edge(U)(A)
		\Edge[bend=45](U)(A)
		\Edge[bend=-45](U)(A)
	\end{tikzpicture}
\end{displaymath}
Note that this network utilizes rank-$n$ and rank-$(2n)$ \emph{tensors} even though the density operator $\rho$ has rank 1 as a \emph{linear operator}.

While tensor networks can express arbitrary multilinear maps on tensor product Hilbert spaces, applications are typically based on architectures that exhibit different types of network locality, motivated by the local nature of interactions in physical systems of interest.
Examples include matrix product states (MPSs), where the tensors are arranged in one-dimensional lattices, projected entangled pairs states (PEPs), which are higher-dimensional generalizations of MPSs, and tree tensor networks (TTN) which exhibit tree network topologies.
See, e.g., \cite{Banuls23}*{Figure~1} for illustrations.
The class of tensor networks that we will build for Koopman operator approximation lie in the TTN family (see \cref{fig:network}).

Based on notions from perturbative calculations in interacting quantum field theories, the dangling (open) legs in a tensor network can be thought of as representing physical degrees of freedom (particles), whereas the internal (contracted) legs are interpreted as virtual particles.
Since one can insert a product $\bm A \bm A^{-1}$ of an arbitrary invertible matrix $\bm A$ and its inverse at any internal edge without changing the state represented by the open edges, it follows that a given tensor network has infinitely many equivalent representations corresponding to different configurations of virtual particles.
This form of ``gauge freedom'' has enabled the development of approaches for efficient evaluation and/or approximation of tensor networks.
Examples include renormalization group techniques for coarsening MPSs and TTNs \cites{McCulloch07,NakataniChen13}.

\section{Reproducing kernel Hilbert algebras}
\label{sec:rkha}

We briefly review the definition, properties, and constructions of RKHAs employed in this paper.
A detailed analysis of these objects can be found in \cite{GiannakisMontgomery25}.
For expositions of RKHS theory, see \cites{PaulsenRaghupathi16,SteinwartChristmann08}.

\subsection{Definitions and basic constructions}
\label{sec:rkha_defs}

Let $\mathcal H$ be an RKHS of complex-valued functions on a set $X$ with reproducing kernel $k\colon X \times X \to \mathbb C$ and inner product $\langle \cdot, \cdot \rangle_{\mathcal H}$.
We use $k_x := k(x, \cdot)$ to denote the kernel section at $x \in X$ and $\delta_x \colon \mathcal H \to \mathbb C$ to denote the corresponding pointwise evaluation functional, $\delta_x f = f(x) = \langle k_x, f\rangle_{\mathcal H}$.
We also let $\varphi \colon X \to \mathcal H$ denote the canonical feature map, $\varphi(x) = k_x$.

\begin{defn}[RKHA]
	An RKHS $\mathcal H$ on a set $X$ is a \emph{reproducing kernel Hilbert algebra (RKHA)} if $k_x \mapsto k_x \otimes k_x$, $x \in X$, extends to a bounded linear map $\Delta\colon \mathcal H \to \mathcal H \otimes \mathcal H$.
\end{defn}

Since
\begin{displaymath}
	\langle k_x, \Delta^*(f \otimes g)\rangle_{\mathcal H} = \langle \Delta k_x, f \otimes g \rangle_{\mathcal H} = \langle k_x \otimes k_x, f \otimes g\rangle_{\mathcal H \otimes \mathcal H} = \langle k_x, f\rangle_{\mathcal H} \langle k_x, g\rangle_{\mathcal H} = f(x)g(x),
\end{displaymath}
it follows that $\Delta^*$ is a bounded linear map that implements pointwise multiplication.
As a result, $\mathcal H$ is simultaneously a Hilbert function space and commutative algebra with respect to pointwise function multiplication.
Letting $\pi \colon \mathcal H \to  B(\mathcal H)$ be the multiplier representation that maps $f \in \mathcal H$ to the multiplication operator $(\pi f) \colon g \mapsto f g$, one readily verifies that the operator norm, $\lVert f\rVert_\text{op} := \lVert \pi f \rVert_{B(\mathcal H)}$, generates a coarser topology than the Hilbert space norm on $\mathcal H$, and satisfies
\begin{displaymath}
	\lVert fg\rVert_\text{op} \leq \lVert f\rVert_\text{op} \lVert g\rVert_\text{op}.
\end{displaymath}
Thus, $\mathcal H$ has the structure of a Banach algebra with respect to pointwise function multiplication.
When $k$ is a real-valued kernel we may further equip $\mathcal H$ with pointwise complex conjugation $^*\colon f \mapsto \bar f$ as an isometric involution, making it a Banach $^*$-algebra.
When the constant function $\bm 1\colon X \to \mathbb C$ with $\bm 1(x) = 1$ lies in $\mathcal H$, then $\mathcal H$ is a unital algebra with unit $\bm 1$.
Furthermore, the Hilbert space norm and operator norms become equivalent since $\lVert f \rVert_\mathcal H \leq \lVert \bm 1 \rVert_{\mathcal H} \cdot \lVert f \rVert_\text{op}$, i.e.,
\begin{displaymath}
	\frac{1}{\lVert \bm 1 \rVert_\mathcal H} \lVert f \rVert_\mathcal H \leq \lVert f \rVert_\text{op} \leq \lVert \Delta \rVert_{B(\mathcal H)} \lVert f \rVert_{\mathcal H}.
\end{displaymath}

In addition to having Banach algebra structure, an RKHA $\mathcal H$ is a cocommutative, coassociative coalgebra with $\Delta \colon \mathcal H \to \mathcal H \otimes \mathcal H$ as its comultiplication operator.
In particular, coassociativity of $\Delta$, i.e., $(\Delta \otimes \Id) \circ \Delta = (\Id \otimes \Delta) \circ \Delta$, is an important property that follows from associativity of multiplication (i.e., $\Delta^*$).
This property allows us to amplify $\Delta$ from $\mathcal H$ to the tensor product spaces $\mathcal H^{\otimes(n+1)}$, $n \in \mathbb N$, by defining $\Delta_n \colon \mathcal H \to \mathcal H^{\otimes (n+1)}$ as
\begin{displaymath}
	\Delta_1 = \Delta, \quad \Delta_{n} = (\Delta \otimes \Id^{\otimes (n-1)}) \Delta_{n-1} \quad \text{for } n>1.
\end{displaymath}

Next, let $\sigma(\mathcal H)$ be the spectrum of an RKHA $\mathcal H$ as Banach algebra, i.e., the set of (automatically continuous) multiplicative linear functionals $\chi \colon \mathcal H \to \mathbb C$,
\begin{displaymath}
	\chi(fg) = (\chi f)(\chi g), \quad \forall f, g \in \mathcal H,
\end{displaymath}
equipped with the weak-$^*$ topology of $\mathcal H^* \supset \sigma(\mathcal H)$.
The pointwise evaluation functionals $\delta_x = \langle \varphi(x), \cdot \rangle_{\mathcal H}$ are elements of the spectrum for every $x \in X$, inducing a map $\hat\varphi\colon X \to \sigma(\mathcal H)$ by $\hat \varphi(x) = \delta_x$.
If the feature map $\varphi$ is injective, then so is $\hat\varphi$ and the spectrum $\sigma(\mathcal H)$ contains a copy of $X$ as a subset.

In this work, our focus will be on unital RKHAs whose corresponding map $\hat\varphi$ is a bijection.
This property implies (e.g., \cite{DasGiannakis23}*{Corollary~7}) that the spectrum $\sigma_{\mathcal H}(f)$ of each element $f \in \mathcal H$ (i.e., the set of complex numbers $z$ such that $f-z$ does not have a multiplicative inverse in $\mathcal H$) is equal to its range, $\sigma_{\mathcal H}(f) = f(X)$.
This enables in turn the use of holomorphic functional calculus of functions in $\mathcal H$,
\begin{displaymath}
	a(f) = \frac{1}{2\pi i} \int_\Gamma \frac{a(z)}{z-f} \, dz,
\end{displaymath}
where $a\colon D \to \mathbb C$ is any holomorphic function on a domain $D \subseteq \mathbb C$ that contains $\sigma_{\mathcal H}(f)$ and $\Gamma$ is an appropriate union of Jordan curves that encircle $\sigma_{\mathcal H}(f)$.
In particular, for every $\xi \in \mathcal H$ satisfying $\xi(x) \geq \varepsilon > 0$ for all $x \in X$, the $n$-th root $\xi^{1/n}$ lies in $\mathcal H$ for every $n \in \mathbb N$ by the holomorphic functional calculus.
As we will see in \cref{sec:tensornet}, the well-definition of such roots as elements of $\mathcal H$ plays a key role in our tensor network approximation scheme.

\subsection{Examples of RKHAs}
\label{sec:rkha_examples}

The main class of examples considered here are built on compact abelian groups, $G$.
Let $\hat G$ be the Pontryagin dual of $\hat G$, and $\mu$, $\hat{\mu}$ the Haar measures on $G$, $\hat G$, respectively.
We use $\mathcal F \colon L^1(G) \to C_0(\hat{G})$ and $\hat{\mathcal F} \colon L^1(\hat{G}) \to C_0(G)$ to denote the Fourier transforms, and $*\colon L^1(\hat G) \times L^1(\hat G) \to L^1(\hat G)$ to denote convolution on the dual group:
$$(\mathcal F f)(\gamma) = \int_G f(x)\gamma(-x)\, d\mu(x), \quad (\hat{\mathcal F}\hat{f})(x) = \int_{\hat{G}} \hat{f}(\gamma)\gamma(x)\,d\hat{\mu}(x), \quad (\hat f * \hat g)(\gamma) = \int_{\hat G} \hat f(\gamma') g(\gamma-\gamma')\, d\hat\mu(\gamma'). $$
Note that, by compactness of $G$, the dual group $\hat G$ has a discrete topology and $\hat \mu$ is a counting measure.
See, e.g., \cite{Rudin17} for further details on Fourier analysis on groups.

By Bochner's theorem, for every positive function $\lambda \in L^1(\hat G)$ (referred to here as inverse weight) the translation-invariant function $k \colon G \times G \to \mathbb C$ defined as $k(x, x') = (\hat{\mathcal F} \lambda)(x - x')$ is a continuous, positive-definite kernel with an associated RKHS $\mathcal H$ of continuous functions.
The space $\mathcal H$ can be characterized in terms of a decay condition on Fourier coefficients,
\begin{displaymath}
	\mathcal H = \left\{ f \in C(G):\; \sum_{\gamma\in\hat G} \frac{\lvert \mathcal F f(\gamma)\rvert^2}{\lambda(\gamma)} < \infty \right\}.
\end{displaymath}
Moreover, by Mercer's theorem, the reproducing kernel $k$ can be expressed in terms of the uniformly convergent series
\begin{displaymath}
	k(x, x') = \sum_{\gamma \in \hat G} \lambda(\gamma)\overline{\gamma(x)}\gamma(x'), \quad \forall x, x' \in G,
\end{displaymath}
and $ \{ \psi_\gamma := \sqrt{\lambda(\gamma)}\gamma \}_{\lambda(\gamma)\neq 0}$ is an orthonormal basis of $\mathcal H$.
If $\lambda$ is a strictly positive function, then $\mathcal H$ is a dense subspace of $C(G)$ (equivalently, $k$ is a so-called universal kernel \cite{SriperumbudurEtAl11}) and the corresponding feature map $\varphi\colon G \to \mathcal H$ is injective.

\begin{defn}
	A strictly positive, summable function $\lambda \colon \hat G \to \mathbb R_{>0}$ is said to be \emph{subconvolutive} if there exists $C>0$ such that $\lambda * \lambda \leq C \lambda$ pointwise on $\hat G$.
\end{defn}

Subconvolutive functions have a rich history of study in the context of Beurling convolution algebras, e.g., \cites{Feichtinger79,Grochenig07,Kaniuth09}.
In \cite{GiannakisMontgomery25}, the RKHS $\mathcal H$ induced by a subconvolutive function $\lambda$ is shown to be an RKHA and $\Delta$ is easily diagonalizable by
$$\Delta\psi_\gamma = \sum_{\alpha+\beta=\gamma} \sqrt{\frac{\lambda(\alpha)\lambda(\beta)}{\lambda(\gamma)}} \psi_\alpha \otimes \psi_\beta, \quad \gamma \in \hat{G}.$$
Standard examples for $G = \mathbb T^N$, $\hat G \cong Z^N$, include the families of functions $\{ \lambda_\tau \colon \mathbb Z^N \to \mathbb R_{>0} \}_{\tau>0}$ with subexponential decay, where
\begin{equation}
	\label{eq:lambda_subexp}
	\lambda_\tau(j) = \prod_{i=1}^N e^{-\tau \lvert j_i \rvert^p}
\end{equation}
for fixed $p \in (0,1)$.
For each such $p$, we have a one-parameter family of unital RKHAs $\mathcal H_\tau$ generated by a Markovian family of kernels $k_\tau$, i.e., $k_\tau \geq 0$ and
\begin{equation}
	\label{eq:markov}
	k_\tau(x, x') \geq 0, \quad \int_G k_\tau(x, \cdot) \, d\mu = 1, \quad \forall x, x' \in X;
\end{equation}
see \cite{DasGiannakis23}.
Moreover, in this setting the elements $\gamma$ of the dual group $\hat G$ are standard Fourier functions, $\gamma_j(x)=e^{i j \cdot x}$ for some $j \in \mathbb Z^N$ using the parameterization $x \in [0, 2\pi)^N$ for the $N$-torus.

			\begin{defn}
				A function $\lambda\colon \hat G \to \mathbb R_{>0}$ is said to satisfy the \emph{Gelfand--Raikov--Shilov (GRS)} condition if
				\begin{displaymath}
					\lim_{n \to \infty} \lambda(n\gamma)^{1/n} = 1, \quad \forall \gamma \in \hat{G},
				\end{displaymath}
				and the \emph{Beurling--Domar (BD)} condition if
				\begin{displaymath}
					\sum_{n=1}^\infty \dfrac{\ln(\lambda^{-1}(n\gamma))}{n^2} < \infty, \quad \forall \gamma \in \hat{G}.
				\end{displaymath}
			\end{defn}

			Functions in the subexponential family from~\eqref{eq:lambda_subexp} satisfy both the GRS and BD conditions, which has two implications for $\mathcal H_\tau$.
			First, the GRS condition implies that the associated map $\hat \varphi_\tau\colon G \to \sigma(\mathcal H_\tau)$ is a homeomorphism \cites{DasEtAl23,GiannakisMontgomery25}.
			Second, as shown in \cite{GiannakisMontgomery25}, the BD condition implies that $\mathcal H_\tau$ contains functions with arbitrary compact support.
			These two conditions play a vital role for the construction of RKHAs on generic compact Hausdorff spaces $X \subseteq \mathbb T^N$ (e.g., attractors of dynamical systems).
			We may restrict to $\mathcal H_\tau(X) := \overline{\spn \left\lbrace k_{\tau,x} \mid x \in X \right \rbrace}$ which is also an RKHA with comultiplication $\Delta|_{\mathcal H_\tau(X)}$.
			More importantly, the spectrum of $\mathcal H_\tau(X)$ as a Banach algebra is $X$, hence, $\sigma_{\mathcal H_\tau(X)}(f) = f(X)$.
			This enables the use of the holomorphic functional calculus of functions in $\mathcal H_\tau(X)$ as outlined in \cref{sec:rkha_defs}.

			\section{Tensor network approximation method}
			\label{sec:tensornet}

			Having surveyed the relevant notions from quantum theory and RKHS theory, in this section we describe our tensor network approximation framework for the Koopman evolution of observables.
			The primary elements of the scheme are (i) spectral approximation of the generator to yield a diagonalizable operator acting on an RKHA; (ii) lifting to a Fock space generated by the RKHA where the approximate generator acts as a free derivation; and (iii) finite-rank compression performed at the level of the Fock space, resulting in an approximation implementable as a TTN.
			These topics are discussed in \cref{sec:spectral_approx,sec:fock,sec:finite_rank}, respectively, following basic definitions and constructions related to the dynamical system in \cref{sec:dyn_syst}.
			The main approximation result in this section is \cref{thm:tau_conv} that establishes asymptotic convergence of the quantum evolution of the Fock space to the true statistical evolution under the transfer operator.
			\Cref{fig:network} displays the tensor network architecture stemming from our approach.

			\begin{figure}
				\centering
				\begin{tikzpicture}[align]
					\clip[rounded corners] (-1.5,0) rectangle (19.5,6.5);

					\draw[very thick] (-0.5,1)--(10,1);
					\draw[very thick] (-0.5,3)--(5.5,3);
					\draw[very thick] (-0.5,4)--(5.5,4);
					\draw[very thick] (-0.5,5)--(4,5);
					\draw[very thick] (4.5,3)--(6.5,3);
					\draw[very thick] (6.5,2)--(7.5,2);

					\node at (-0.5,1)[Tbox, minimum height=.5cm, minimum width=.5cm]{\scriptsize $\xi^\dag_{\tau, n,d}$};
					\node at (-0.5,5)[Tbox, minimum height=.5cm, minimum width=.5cm]{\scriptsize $\xi^\dag_{\tau, n,d}$};
					\node at (-0.5,4)[Tbox, minimum height=.5cm, minimum width=.5cm]{\scriptsize $\xi^\dag_{\tau, n,d}$};
					\node at (-0.5,3)[Tbox, minimum height=.5cm, minimum width=.5cm]{\scriptsize $\xi^\dag_{\tau, n,d}$};

					\node at (2.5,3)[Tbox, minimum height=.5cm, minimum width=.5cm]{\scriptsize $\tilde G_{\tau, \sigma}$};
					\node at (2.5,4)[Tbox, minimum height=.5cm, minimum width=.5cm]{\scriptsize $\tilde G_{\tau, \sigma}$};
					\node at (2.5,5)[Tbox, minimum height=.5cm, minimum width=.5cm]{\scriptsize $\tilde G_{\tau, \sigma}$};
					\node at (2.5,1)[Tbox, minimum height=.5cm, minimum width=.5cm]{\scriptsize $\tilde G_{\tau, \sigma}$};

					\node at (1,3)[Tbox, minimum height=.5cm, minimum width=.5cm]{\scriptsize $U^t_\tau$};
					\node at (1,4)[Tbox, minimum height=.5cm, minimum width=.5cm]{\scriptsize $U^t_\tau$};
					\node at (1,5)[Tbox, minimum height=.5cm, minimum width=.5cm]{\scriptsize $U^t_\tau$};
					\node at (1,1)[Tbox, minimum height=.5cm, minimum width=.5cm]{\scriptsize $U^t_\tau$};

					\node at (4,4.5)[Tbox, minimum height=1cm, minimum width=.5cm]{\scriptsize $\Delta^*$};
					\node at (5.5,3.5)[Tbox, minimum height=1cm, minimum width=.5cm]{\scriptsize $\Delta^*$};
					\node at (7.5,1.5)[Tbox, minimum height=1cm, minimum width=.5cm]{\scriptsize $\Delta^*$};

					\node at (9,1)[Tbox, minimum height=.5cm, minimum width=.5cm]{\scriptsize $M_{f,\tau}$};

					\node at (-0.5,2){$\cdot$};
					\node at (-0.5,2.25){$\cdot$};
					\node at (-0.5,1.75){$\cdot$};
					% \node at (-0.5,3){$\cdot$};

					\node at (1,2){$\cdot$};
					\node at (1,2.25){$\cdot$};
					\node at (1,1.75){$\cdot$};
					% \node at (1,3){$\cdot$};

					\node at (2.5,2){$\cdot$};
					\node at (2.5,2.25){$\cdot$};
					\node at (2.5,1.75){$\cdot$};
					% \node at (2.5,3){$\cdot$};

					\node at (6.5,2.75){$\cdot$};
					\node at (6.5,2.5){$\cdot$};
					\node at (6.5,2.25){$\cdot$};

					%This is the reflection of the above diagram
					\draw[very thick] (18.5,1)--(10,1);
					\draw[very thick] (18.5,4)--(12.5,4);
					\draw[very thick] (18.5,5)--(14,5);
					\draw[very thick] (18.5,3)--(13,3);
					\draw[very thick] (13.5,3)--(11.5,3);
					\draw[very thick] (11.5,2)--(10.5,2);

					\node at (18.5,1)[Tbox, minimum height=.5cm, minimum width=.5cm]{\scriptsize $\xi_{\tau, n,d}$};
					\node at (18.5,5)[Tbox, minimum height=.5cm, minimum width=.5cm]{\scriptsize $\xi_{\tau, n,d}$};
					\node at (18.5,4)[Tbox, minimum height=.5cm, minimum width=.5cm]{\scriptsize $\xi_{\tau, n,d}$};
					\node at (18.5,3)[Tbox, minimum height=.5cm, minimum width=.5cm]{\scriptsize $\xi_{\tau, n,d}$};

					\node at (17,3)[Tbox, minimum height=.5cm, minimum width=.5cm]{\scriptsize $U^{t*}_\tau$};
					\node at (17,4)[Tbox, minimum height=.5cm, minimum width=.5cm]{\scriptsize $U^{t*}_\tau$};
					\node at (17,5)[Tbox, minimum height=.5cm, minimum width=.5cm]{\scriptsize $U^{t*}_\tau$};
					\node at (17,1)[Tbox, minimum height=.5cm, minimum width=.5cm]{\scriptsize $U^{t*}_\tau$};

					\node at (15.5,3)[Tbox, minimum height=.5cm, minimum width=.5cm]{\scriptsize $\tilde G_{\tau, \sigma}$};
					\node at (15.5,4)[Tbox, minimum height=.5cm, minimum width=.5cm]{\scriptsize $\tilde G_{\tau, \sigma}$};
					\node at (15.5,5)[Tbox, minimum height=.5cm, minimum width=.5cm]{\scriptsize $\tilde G_{\tau, \sigma}$};
					\node at (15.5,1)[Tbox, minimum height=.5cm, minimum width=.5cm]{\scriptsize $\tilde G_{\tau, \sigma}$};

					\node at (14,4.5)[Tbox, minimum height=1cm, minimum width=.5cm]{\scriptsize $\Delta$};
					\node at (12.5,3.5)[Tbox, minimum height=1cm, minimum width=.5cm]{\scriptsize $\Delta$};
					\node at (10.5,1.5)[Tbox, minimum height=1cm, minimum width=.5cm]{\scriptsize $\Delta$};

					\node at (18.5,2){$\cdot$};
					\node at (18.5,2.25){$\cdot$};
					\node at (18.5,1.75){$\cdot$};

					\node at (17,2){$\cdot$};
					\node at (17,2.25){$\cdot$};
					\node at (17,1.75){$\cdot$};

					\node at (15.5,2){$\cdot$};
					\node at (15.5,2.25){$\cdot$};
					\node at (15.5,1.75){$\cdot$};

					\node at (11.5,2.75){$\cdot$};
					\node at (11.5,2.5){$\cdot$};
					\node at (11.5,2.25){$\cdot$};

					\draw[rounded corners, very thick] (-1.5,0) rectangle (19.5,6.5);
				\end{tikzpicture}
				\caption{Tree tensor network used to compute the Fock space evolution $f^{(t)}_{\sigma, \tau, n, d}$ in~\eqref{eq:ft_fock_d}.
				Structurally, the comultiplication operators form a tree surrounded by smoothing operators $\tilde G_{\tau, \sigma}$.
				}
				\label{fig:network}
			\end{figure}

			\subsection{Dynamical system and function spaces}
			\label{sec:dyn_syst}

			We consider a dynamical flow $\Phi^t\colon G \to G$, $t \in \mathbb R$, on a compact Lie group $G$ generated by a continuous vector field $\vec V\colon G \to TG$, and possessing a Borel, invariant, ergodic probability measure $\mu$ with (compact) support $X \subseteq G$.
			Following the notation of \cref{sec:background}, we let $U^t \colon L^p(\mu) \to L^p(\mu)$ be the associated Koopman operators, $U^t f = f \circ \Phi^t$, defined for $p \in [1, \infty]$ and $t \in \mathbb R$.
			We also let $V\colon D(V) \to H$ be the skew-adjoint generator on $H=L^2(\mu)$ defined as in~\eqref{eq:generator}.
			Note that $V$ reduces to a directional derivative along $\vec V$ when acting on continuously differentiable functions.
			That is, we have
			\begin{equation}
				\label{eq:v_dot_grad}
				V \iota f = \iota \vec V \cdot \nabla f, \quad \forall f \in C^1(G),
			\end{equation}
			where $\iota \colon C(G) \to L^p(\mu)$ denotes the restriction map from continuous functions to $L^p$ spaces.

			We assume that $G$ is equipped with a one-parameter family of kernels $k_\tau \colon G \times G \to \mathbb R_+$, $\tau > 0$, such that the corresponding RKHSs $\mathcal H_\tau$ are unital RKHAs.
			In addition, we require that the kernels $k_\tau$ have the following properties:
			\begin{enumerate}[label=(K\arabic*)]
				\item \label[prty]{prty:K1} $k_\tau$ is continuous, and $\mathcal H_\tau$ is a dense subspace of $C(G)$.
				\item \label[prty]{prty:K2} For every invariant vector field $\partial$ on $G$, the derivatives $(\partial \times \partial) k_\tau$ exist and are continuous.
				\item \label[prty]{prty:K3} $k_\tau$ is Markovian with respect to $\mu$; i.e., \eqref{eq:markov} holds.
				\item \label[prty]{prty:K4} The integral operators $\mathcal K_\tau \colon L^p(\mu) \to L^p(\mu)$, where $p \in [1, \infty)$ and $\mathcal K_\tau f = \int_X k_\tau(\cdot, x) f(x) \, d\mu(x)$, are bounded and have the semigroup property, $\mathcal K_\tau \mathcal K_{\tau'} = \mathcal K_{\tau+\tau'}$.
				      Moreover, $\mathcal K_\tau$ converges strongly to the identity on $L^p(\mu)$ as $\tau \to 0^+$.
				\item \label[prty]{prty:K7} The RKHAs $\mathcal H_\tau(X)$ have homeomorphic spectra $\sigma(\mathcal H_\tau)$ to $X$.
			\end{enumerate}

			By standard RKHS results (e.g., \cite{SteinwartChristmann08}*{Chapter~4}), $\mathcal K_\tau \colon H \to H$ is a strictly positive, self-adjoint, trace class operator.
			This operator admits the factorization $\mathcal K_\tau = K_\tau^* K_\tau$, where $K_\tau \colon H \to \mathcal H_\tau$, $K_\tau f = \int_X k_\tau(\cdot, x) f(x) \, d\mu(x)$, is Hilbert--Schmidt, and $K_\tau^*$ implements the restriction map from $\mathcal H_\tau$ to $H$, $K_\tau^* = \iota \rvert_{\mathcal H_\tau}$.
			This latter map restricts to a compact embedding $K_\tau^* \colon \mathcal H_\tau(X) \hookrightarrow H$ with dense image $H_\tau:= K_\tau^*(\mathcal H_\tau) \subseteq H$.
			The pseudoinverse $K_\tau^{*+}\colon H_\tau \to \mathcal H_\tau(X)$ is a (possibly unbounded) operator, known as the Nystr\"om operator, that recovers the representative in $\mathcal H_\tau(X)$ of an element $f \in H_\tau$; i.e.\ $K_\tau^* K_\tau^{*+} f = f$.
			The Nystr\"om operator induces a Dirichlet energy functional $\mathcal E_\tau \colon H_\tau \to \mathbb R_+$,
			\begin{displaymath}
				\mathcal E_\tau(f) = \frac{\lVert K_\tau^{*+}f \rVert_{\mathcal H_\tau}^2}{\lVert f \rVert_H^2} - 1.
			\end{displaymath}
			Intuitively, $\mathcal E_\tau(f)$ measures an average notion of spatial variability of elements of $H_\tau$; in particular, $\mathcal E_\tau(f) = 0$ iff $f$ is $\mu$-.a.e.\ constant.
			We also have that $\ran K_\tau$ is a dense subspace of $\mathcal H_\tau(X)$ and, by \Cref{prty:K2}, $\ran K_\tau^*$ lies in the domain $D(V)$ of the generator.
			Since $\mathcal H_\tau$ is a subspace of $C(G)$, we have $\ran K_\tau^* \subset \mathfrak A \equiv L^\infty(\mu) $.
			Thus, viewed as a map from $\mathcal H_\tau$ into $\mathfrak A$, $K_\tau^*$ is multiplicative when $\mathcal H_\tau$ is an RKHA,
			\begin{displaymath}
				K_\tau^*(fg) = (K_\tau^* f)(K_\tau^* g), \quad \forall f, g \in \mathcal H_\tau.
			\end{displaymath}

			The following are consequences of the semigroup property in~\ref{prty:K4} \cite{DasEtAl21}:
			\begin{enumerate}[label=(K\arabic*), resume]
				\item \label[prty]{prty:K5} The spaces $\mathcal H_\tau$ form a nested family, $\mathcal H_\tau \subseteq \mathcal H_{\tau'}$ for $0 < \tau' < \tau$.
				\item \label[prty]{prty:K6} $K_\tau$ admits the polar decomposition $K_\tau = T_\tau \mathcal K_{\tau/2}$, where $T_\tau\colon H \to \mathcal H_\tau$ is an isometry with $\ran T_\tau = \mathcal H_\tau(X)$ and $\ran T_\tau^* = H$.
			\end{enumerate}

			Define now the space $\mathcal H_\infty = \bigcap_{\tau>0} \mathcal H_\tau(X)$.
			Another consequence of \cref{prty:K4} is that the operators $\mathcal K_\tau$ are simultaneously diagonalized in an orthonormal basis $ \{ \phi_j \}_{j \in \mathbb N_0}$ of $H$ with representatives $\varphi_j \in \mathcal H_\infty$, viz.
			\begin{equation}
				\label{eq:k_eig}
				\mathcal K_\tau \phi_j = \Lambda_{j,\tau} \phi_j, \quad \phi_j = \iota \varphi_j, \quad \varphi_j = \frac{1}{\Lambda_{j,\tau}} K_\tau \phi_j,
			\end{equation}
			where the eigenvalues $\Lambda_{j,\tau}$ are strictly positive and satisfy $\Lambda_{j,\tau} \Lambda_{j,\tau'} = \Lambda_{j,\tau + \tau'}$.
			The representatives $\varphi_j \in \mathcal H_\infty$ can be employed in a Mercer sum of the restriction of the kernel $k_\tau$ on $X \times X$,
			\begin{displaymath}
				k_\tau(x, x') = \sum_{j=0}^\infty \Lambda_{j,\tau} \overline{\varphi_j(x)} \varphi_j(y), \quad \forall x, x' \in X.
			\end{displaymath}
			Moreover, the set $ \{ \psi_{j,\tau} \}_{j \in \mathbb N_0}$ with
			\begin{equation}
				\label{eq:psi}
				\psi_{j,\tau} = \frac{1}{\Lambda_{j,\tau/2}} K_\tau \phi_j = \Lambda_{j,\tau/2} \varphi_j
			\end{equation}
			is an orthonormal basis of $\mathcal H_\tau(X)$.

			For later convenience, we note that the Nystr\"om representative $K_\tau^{*+}$ of $f=\sum_{j\in \mathbb N_0} c_j \phi_j \in H_\tau$ is given by
			\begin{displaymath}
				K_\tau^{*+} f = \sum_{j=0}^\infty \frac{c_j}{\Lambda_{j,\tau/2}} \psi_{j,\tau}.
			\end{displaymath}
			This leads to the following formula for the Dirichlet energy,
			\begin{equation}
				\label{eq:dirichlet}
				\mathcal E_\tau(f) = \frac{\sum_{j=1}^\infty \lvert c_j\rvert^2\Lambda_{j,\tau}^{-1}}{\sum_{j=0}^\infty \lvert c_j\rvert^2}.
			\end{equation}
			Similarly, for $f = \sum_{j\in \mathbb N_0} c_j \varphi_j \in \ran K_\tau$, the pseudoinverse $K_\tau^+ f \in H$ can be computed as
			\begin{equation*}
				K_\tau^+ f = \sum_{j=0}^\infty \frac{c_j}{\Lambda_{j,\tau}} \phi_j.
			\end{equation*}

			\begin{example*}[Circle rotation]
				In the circle rotation example of \cref{sec:circlerot} we have $G = X = \mathbb T$, and we build $k_\tau$ from the subexponential families of weights in~\eqref{eq:lambda_subexp} with $p=0.75$ and $\tau=0.005$.
				The resulting spaces $\mathcal H_\tau$ then satisfy all of \crefrange{prty:K1}{prty:K7}.
				In addition, the functions $\varphi_j \in \mathcal H_\infty$ were chosen as characters $\gamma$ of the group (Fourier functions) as described in \cref{sec:rkha_examples}.
				We will employ analogous RKHA constructions in the numerical experiments of \cref{sec:experiments} performed on the 2-torus.
			\end{example*}

			In addition to these examples, our approach can be generalized to cases where the support $X$ of the invariant measure is a strict subset of $G$, and/or cases where $G$ is a forward-invariant subset of an ambient metric space $\mathcal M$ in which the flow $\Phi^t \colon \mathcal M \to \mathcal M$ takes place.

			Next, we consider the parameter $\tau$ in the interval $(0,1]$ and note that $\mathcal H_1 = \bigcap_{\tau \in (0,1]} \mathcal H_\tau$ since the spaces $\mathcal H_\tau$ form a nested family.
Letting $\iota\colon \mathcal H_1 \to L^r(\mu)$, $r \in [1, \infty]$, be the corresponding restriction map into any of the $L^r$ spaces associated with the invariant measure we see from~\eqref{eq:dirichlet} that $\iota(\mathcal H_1)$ contains elements of $H$ with uniformly bounded Dirichlet energies over $\tau \in (0,1]$.
Moreover, by our choice of reproducing kernels for $\mathcal H_\tau$, $\iota (\mathcal H_1)$ lies dense in $L^r(\mu)$ for all $r \in [1,\infty)$.
			In particular, with $\mathfrak A = L^\infty(\mu)$ as in \cref{sec:quantum}, we show in the following lemma that the evolution $\mathbb E_q (U^t f)$ of an observable $f \in \mathfrak A$ with respect to an probability density $q \in S_*(\mathfrak A)$ can be approximated at arbitrary accuracy by the evolution $\mathbb E_p (U^t f)$ with respect to a density $ p \in S_*(\mathfrak A)$ with a strictly positive representative $\varrho \in \mathcal H_1$.
			Strict positivity of $\varrho$ is important when $\mathcal H_1$ satisfies \cref{prty:K7} for it implies that the $n$-th root $\varrho^{1/n}$ is a well-defined element of $\mathcal H_1$ for every $n \in \mathbb N$ by the holomorphic functional calculus.

			\begin{lem}
				\label{lem:prob}
				For every $q \in S_*(\mathfrak A)$ and $\epsilon>0$, there exists a probability density $p = \iota\varrho \in \mathfrak A_*$ with a representative $\varrho \in \mathcal H_1$ such that (i) $\varrho(x) > 0$ for all $x \in G$; and (ii) for every $f \in \mathfrak A$ and $t \in \mathbb R$, $ \lvert \mathbb E_p(U^t f) - E_q (U^t f) \rvert < \epsilon \lVert f \rVert_{L^\infty(\mu)}$.
			\end{lem}

			\begin{proof}
				Since $U^t$ acts as an isometry of $\mathfrak A$, it suffices to prove the lemma for $t=0$.
				To that end, for $\varepsilon >0$ choose $g \in C(G)$ such that $g \geq 0$, $\lVert \iota g \rVert_{L^1(\mu)} = 1$, and  $\lVert q - \iota g\rVert_{L^1(\mu)} < \varepsilon$.
					Note that the existence of such a $g$ follows from density of $\iota(C(G))$ in $L^1(\mu)$.
				For $\delta > 0$, define $g_\delta \in C(G)$ such that $g_\delta(x) = g(x) + \delta$.
				By density of $\mathcal H_1$ in $C(G)$, there exists $h_\delta \in \mathcal H_1$ such that $\lVert h_\delta - g_\delta\rVert_{C(G)} \leq \delta/2$.
				In particular, $h_\delta$ is a strictly positive function with $h_\delta(x) \geq \delta / 2$.
				We also have $\lVert \iota g_\delta \rVert_{L^1(\mu)} = 1 + \delta$, and since
				\begin{displaymath}
					\left\lvert \lVert \iota h_\delta\rVert_{L^1(G)} - \lVert \iota g_\delta\rVert_{L^1(G)} \right\rvert \leq \lVert \iota(h_\delta - g_\delta)\rVert_{L^1(\mu)} \leq \lVert h_\delta - g_\delta\rVert_{C(G)} \leq \delta/2,
				\end{displaymath}
				it follows that $c_\delta := \lVert \iota h_\delta\rVert_{L^1(\mu)}$ satisfies $1+\delta/2 \leq c_\delta \leq 1 + 3 \delta/2$.
					The above imply that
				\begin{equation*}
					\frac{\lVert h_\delta \rVert_{C(G)}}{c_\delta} \leq \frac{\lVert g_\delta - h_\delta \rVert_{C(G)}+ \lVert g_\delta \rVert_{C(G)}}{c_\delta} \leq \frac{3\delta/2 + \lVert g \rVert_{C(G)}}{1 + \delta/2} \leq M
				\end{equation*}
				for a uniform upper bound $M$ over $\delta \geq 0$.

				Next, define $\varrho_\delta = h_\delta / c_\delta$.
				Then, $p_\delta := \iota \varrho_\delta$ is a probability density in $L^1(\mu)$ with a strictly positive representative $\varrho_\delta \in \mathcal H_1$.
				In addition, we have
				\begin{align*}
					\lVert q - p_\delta \rVert_{L^1(\mu)}         & = \lVert q - \iota \varrho_\delta + \iota g - \iota g \rVert_{L^1(\mu)} \leq \lVert q - \iota g \rVert_{L^1(\mu)} + \lVert \varrho_\delta - g \rVert_{C(G)} < \varepsilon + \lVert \varrho_\delta - g \rVert_{C(G)}, \\
					\lVert \varrho_\delta - g \rVert_{C(G)}       & \leq \lVert \varrho_\delta - g_\delta \rVert_{C(G)} + \lVert g_\delta - g\rVert_{C(G)} =  \lVert \varrho_\delta - g_\delta \rVert_{C(G)} + \delta,                                                             \\
					\lVert \varrho_\delta - g_\delta\rVert_{C(G)} & \leq \lVert \varrho_\delta - h_\delta \rVert_{C(G)} + \delta / 2,                                                                                                                                                    \\
					\lVert \varrho_\delta - h_\delta\rVert_{C(G)} & =  \lVert h_\delta\rVert_{C(G)} \frac{\lvert 1 - c_\delta\rvert}{c_\delta} \leq \frac{3M\delta}{2}.
				\end{align*}
				Combining the above estimates, we obtain
				\begin{displaymath}
					\lVert q - p_\delta\rVert_{L^1(\mu)} \leq \varepsilon + \delta + \frac{\delta}{2} + \frac{3M\delta}{2} = \varepsilon + \frac{3(M+1)\delta}{2},
				\end{displaymath}
				and thus $\lVert q - p\rVert_{L^1(\mu)} < \epsilon$ for $p = p_\delta$ and sufficiently small $\varepsilon,\delta$.
				The claim of the lemma follows from $\lvert \mathbb E_p f - E_q f \rvert \leq \lVert p-q\rVert_{L^1(\mu)} \lVert f \rVert_{L^\infty(\mu)}$.
			\end{proof}

			\subsection{Spectral approximation}
			\label{sec:spectral_approx}

			We approximate the skew-adjoint generator $V\colon D(V) \to H$ on $H=L^2(\mu)$ by a family of skew-adjoint, diagonalizable operators $W_\tau \colon D(W_\tau) \to \mathcal H_\tau$, $\tau > 0$, acting on RKHSs $\mathcal H_\tau$ of complex-valued, continuous functions on the state space $X$.
			The operators $W_\tau$ are built so as to converge spectrally to $V$ as $\tau \to 0^+$ in a sense that we elaborate on in this subsection.

			We will make use of the following notions of convergence of skew-adjoint operators; e.g., \cites{Chatelin11,Oliveira09}.

			\begin{defn} Let $A\colon D(A) \to \mathbb H$ be a skew-adjoint operator on a Hilbert space $\mathbb H$ and $A_\tau \colon D(A_\tau) \to \mathbb H$ a family of skew-adjoint operators indexed by $\tau$, with resolvents $R_z(A) = (zI - A)^{-1}$ and $R_z(A_\tau) = (zI - A_\tau)^{-1}$ respectively, for a complex number $z$ in the resolvent sets of $A$ and $A_\tau$.
				\begin{enumerate}
					\item The family $A_\tau$ is said to converge in \emph{strong resolvent sense} to $A$ as $\tau\to 0^+$ if for some (and thus, every) $z \in \mathbb C \setminus i \mathbb R$ the resolvents $R_z(A_\tau)$ converge strongly to $R_z(A)$; that is, $\lim_{\tau\to 0^+}R_z(A_\tau) f = R_z(A) f$ for every $f\in \mathbb H$.
					\item The family $A_\tau$ is said to converge in \emph{strong dynamical sense} to $A$ as $\tau\to 0^+$ if for every $t \in \mathbb R$, the unitary operators $e^{tA_\tau}$ converge strongly to $e^{tA}$; that is, $\lim_{\tau\to0^+} e^{tA_\tau} f = e^{tA} f$ for every $f\in \mathbb H$.
				\end{enumerate}
			\end{defn}

			It can be shown, e.g., \cite{Oliveira09}*{Proposition~10.1.8}, that strong resolvent convergence and strong dynamical convergence are equivalent notions.
			For our purposes, this implies that if a family of skew-adjoint operators $V_\tau$ on $H$ converges to the Koopman generator $V$ in strong resolvent sense, the unitary evolution groups $e^{t V_\tau}$ generated by these operators consistently approximate the Koopman group $U^t = e^{t V}$ generated by $V$.
			Strong resolvent convergence and strong dynamical convergence also imply a form of strong convergence of spectral measures, e.g., \cite{DasEtAl21}*{Proposition~13}.

			With these definitions, and assuming that $k_\tau \colon G \times G \to \mathbb R_+$, $\tau >0$, is a family of kernels satisfying \crefrange{prty:K1}{prty:K6}. the first step in our tensor network approximation scheme is to build a family of approximating operators $V_\tau\colon D(V_\tau) \to H$ with the following properties:

			\begin{enumerate}[label=(V\arabic*)]
				\item \label[prty]{prty:V1} $V_\tau$ is skew-adjoint.
				\item \label[prty]{prty:V2} $V_\tau$ is real, i.e., $\overline{V_\tau f} = V_\tau \bar f$ for all $f \in D(V_\tau)$.
				\item \label[prty]{prty:V3} $V_\tau$ is diagonalizable.
				\item \label[prty]{prty:V4} $V_\tau$ has a simple eigenvalue at 0 with $\bm 1$ as a corresponding eigenfunction.
				\item \label[prty]{prty:V5} $D(V_\tau)$ includes $\ran \mathcal K_{\tau/2}$ as a subspace.
				\item \label[prty]{prty:V6} As $\tau \to 0^+$, $V_\tau$ converges in strong resolvent sense, and thus in strong dynamical sense, to $V$.
			\end{enumerate}

			For any family of operators $V_\tau$ satisfying \crefrange{prty:V1}{prty:V6}, we define the skew-adjoint operators $W_\tau \colon D(W_\tau) \to \mathcal H_\tau$ on the dense domains $D(W_\tau) = T_\tau(D(V_\tau)) \oplus \mathcal H_\tau(X)^\perp $  as $W_\tau = T_\tau V_\tau T_\tau^*$.
			Note that the well-definition of $W_\tau$ depends on \cref{prty:V5}.

			By \cref{prty:V3}, for every $\tau>0$, $W_\tau$ is a diagonalizable operator.
			Moreover, by \cref{prty:V2,prty:V4}, the restriction of $W_\tau$ on $\mathcal H_\tau(X)$ admits an eigendecomposition
			\begin{equation}
				\label{eq:w_eig}
				W_\tau \zeta_{j,\tau} = i \omega_{j,\tau} \zeta_{j,\tau}, \quad j \in \mathbb N_0,
			\end{equation}
			where $\omega_{j,\tau}$ are real eigenfrequencies and the corresponding eigenfunctions $\zeta_{j,\tau}$ form an orthonormal basis of $\mathcal H_\tau$ restricted on the support of $\mu$.
			The eigenfrequencies satisfy $\omega_{0,\tau} = 0$ and $\omega_{2j,\tau} = - \omega_{2j-1,\tau}$ for $ j \in \mathbb N$, and are ordered in increasing order of a Dirichlet energy associated with $\mathcal H_\tau$.
			Moreover, we have $\zeta_{0,\tau} = \bm 1$, and the eigenvectors corresponding to nonzero eigenvalues form complex-conjugate pairs, $\zeta_{2j,\tau}^* = - \zeta_{2j-1,\tau}$.
			We denote the unitary evolution operators on $\mathcal H_\tau$ generated by $W_\tau$ as $U^t_\tau := e^{t W_\tau}$, $t \in \mathbb R$.

			The following lemma establishes how the unitaries $e^{t V_\tau}$  and $e^{t W_\tau}$ can be used interchangeably to approximate the Koopman evolution of observables in $H$.

			\begin{lem}
				\label{lem:koopman_rkhs_approx}
				For every $f \in H$, $t \in \mathbb R$, and $\tau>0$, we have $K_\tau^* e^{t W_\tau} K_\tau f= \mathcal K_{\tau/2}e^{tV_\tau} \mathcal K_{\tau/2}f$.
				Moreover, the following hold.
				\begin{enumerate}
					\item For every $f \in H$, $\lim_{\tau\to 0^+} K_\tau^* e^{t W_\tau} K_\tau f = U^t f$.
					\item For every $f \in \mathcal H_1$, $\lim_{\tau\to 0^+} K_\tau^* e^{t W_\tau} f = U^t \iota f$.
				\end{enumerate}
			\end{lem}

			\begin{proof}
				For $h_0 \in D(W_\tau)$, define $h \colon \mathbb R \to \mathcal H_\tau$ by $h(t) = T_\tau e^{tV_\tau} T_\tau^*h_0$.
				Then, $T_\tau^* h_0$ lies in $D(V_\tau)$, and by the generator equation for the one-parameter unitary group $ \{ e^{t V_\tau} \}_{t \in \mathbb R}$ we have
				\begin{displaymath}
					\frac{dh(t)}{dt} = T_\tau V_\tau e^{t V_\tau} T_\tau^* h_0 = W_\tau T_\tau e^{t V_\tau} T_\tau^* h_0 = W_\tau h(t).
				\end{displaymath}
				Thus, $h(t)$ satisfies the generator equation for $ \{ e^{t W_\tau} \}_{t \in \mathbb R}$ with initial condition $h(0) = h_0$, which implies that $h(t) = e^{t W_\tau} h_0$.
				Letting $h_0 = K_\tau f$ (which is an element of $D(W_\tau)$ by \cref{prty:V5}), it follows that $T_\tau e^{t V_\tau} T_\tau^* K_\tau f = e^{t W_\tau} K_\tau f$ and thus  $T_\tau e^{t V_\tau} \mathcal K_{\tau/2} f = e^{t W_\tau} K_\tau f$ by the polar decomposition from \cref{prty:K6}.
				Pre-multiplying both sides of the last equation by $K_\tau^*$ and using again \cref{prty:K6}, we get $\mathcal K_{\tau/2}e^{tV_\tau} \mathcal K_{\tau/2}f = K_\tau^* e^{t W_\tau} K_\tau f$ as claimed.

				For Claim~1 of the lemma, we use the result just proved in conjunction with strong convergence of $e^{t V_\tau}$ to $U^t$ (\cref{prty:V6}) and strong convergence of $\mathcal K_{\tau/2}$ to $\Id$ (\cref{prty:K4}) to deduce
				\begin{displaymath}
					\lim_{\tau \to 0^+} K_\tau^* e^{t W_\tau} K_\tau f = \lim_{\tau\to 0^+} \mathcal K_{\tau/2} e^{t V_\tau} \mathcal K_{\tau/2} f = U^t f.
				\end{displaymath}

				For Claim~2, note that for $\tau \in [0,1)$, $K_\tau^+\rvert_{\mathcal H_1}$ is a right inverse for $K_\tau$, i.e., $K_\tau K_\tau^+ f = f$ for $f \in \mathcal H_1$.
				Moreover, we have  $M = \sup_{\tau\in (0,1]} \sum_{j=0}^\infty \lvert \langle \phi_j, \iota f \rangle_H\rvert^2 \Lambda^{-1}_{j,\tau} < \infty$, so for every $\epsilon>0$ there exists $J \in \mathbb N$ such that $\sum_{j=J+1}^\infty \lvert \langle \phi_j, \iota f\rangle_H\rvert^2 < \epsilon$ and, for every $\tau>0$, $\sum_{j=J+1}^\infty \lvert \langle \phi_j, \iota f \rangle_H\rvert^2 \Lambda^{-1}_{j,\tau} < \epsilon M $.
				Thus, since $\lim_{\tau \to 0^+} \Lambda_{j,\tau} = 1$ (by \cref{prty:K4}), we have
				\begin{align*}
					\limsup_{\tau\to 0} \lVert K_\tau^+ f - \iota f \rVert_H^2
					 & = \limsup_{\tau \to 0} \left( \sum_{j=0}^J \lvert \langle \phi_j, \iota f\rangle_H\rvert^2 (\Lambda_{j,\tau/2}^{-1} - 1)^2 + \sum_{j=J+1}^\infty \lvert \langle \phi_j, \iota f\rangle_H\rvert^2 (\Lambda_{j,\tau/2}^{-1} - 1)^2\right) \\
					 & \leq \epsilon (3M+1),
				\end{align*}
				and thus $\lim_{\tau \to 0^+} K_\tau^+ f = \iota f$ since $\epsilon$ was arbitrary.
				Therefore, using Claim~1, we obtain
				\begin{displaymath}
					\lim_{\tau\to 0^+} K_\tau^* e^{t W_\tau} f = \lim_{\tau\to 0^+} K_\tau^* e^{t W_\tau} K_\tau K_\tau^+ f = U^t \iota f.
				\end{displaymath}
				This proves Claim~2 and completes the proof of the lemma.
			\end{proof}

			Examples of approximation techniques leading to operators satisfying \crefrange{prty:V1}{prty:V6} are described in \cref{app:spectral_approx}.
			These methods are based on the papers \cites{DasEtAl21,GiannakisValva24,GiannakisValva25}.
			Besides these schemes, any approximation technique for Koopman operators satisfying the requisite properties may be employed in the subsequent steps of our tensor network framework.

			\begin{example*}[Circle rotation]
				In \cref{sec:circlerot} we used the method in \cref{app:compact_res} that yields approximations $W_\tau$ with compact resolvent.
				The same method will be used in the 2-torus experiments in \cref{sec:experiments}.
				Referring again to the spectrum of $W_\tau$ listed in \cref{tab:circlerot_spec} for the circle rotation, we see that even though the eigenfrequencies $\omega_{j,\tau}$ with $j >2$ are approximately equal to integer multiples of the basic frequency $\omega_{1,\tau}$ there are clear violations of the Leibniz rule.
				In \cref{sec:fock} we will dilate $W_\tau$ to an operator acting on the Fock space generated by $\mathcal H_\tau$ that satisfies a version of the Leibniz rule with respect to the tensor product.
			\end{example*}

			\subsection{Lifting to Fock space and quantum evolution}
			\label{sec:fock}

			Our approach builds an approximation of the statistical evolution $f^{(t)} := \mathbb E_p(U^t f)$ as a TTN on a Fock space generated by $\mathcal H_\tau$ with a dynamical evolution induced from $U^t_\tau$ from \cref{sec:spectral_approx}.

			Using standard constructions from many-body quantum theory \cite{Lehmann04}, we define the Fock space $F(\mathcal H_\tau)$ as the Hilbert space closure of the tensor algebra $T(\mathcal H_\tau) := \mathbb C \oplus \mathcal H_\tau \oplus \mathcal H_\tau^{\otimes 2} \oplus \ldots$ with respect to the inner product defined by linear extension of
			\begin{displaymath}
				\langle f_1 \otimes \cdots \otimes f_n, g_1 \otimes \cdots \otimes g_n \rangle_{T(\mathcal H_\tau)} = \langle f_1, g_1 \rangle_{\mathcal H_\tau} \cdots \langle f_n, g_n \rangle_{\mathcal H_\tau},
			\end{displaymath}
			where $f_1, \ldots, f_n, g_1, \ldots, g_n \in \mathcal H_\tau$.
			Intuitively, elements of an orthonormal basis of $\mathcal H_\tau \subset F(\mathcal H_\tau)$ represent quantum states associated with individual ``particles'' of different types, and the elements of the corresponding tensor product basis of $F(\mathcal H_\tau)$ represent states of potentially multiple such particles (with the particle number corresponding to the grading $n \in \mathbb N$ of $\mathcal H_\tau^{\otimes n} \subset F(\mathcal H_\tau)$).
			Meanwhile, the unit complex number $1_{\mathbb C} \subset F(\mathcal H_\tau)$ represents the ``vacuum'' state with no particles.

			\subsubsection{Lifting of observables}

			The $n$-fold comultiplication and multiplication operators $\Delta_{n-1}$ and $\Delta_{n-1}^*$, respectively, associated with the tensor product space $\mathcal H_\tau^{\otimes n}$ (referred to in this context as the $n$-th grading of the Fock space) canonically extend to operators on $F(\mathcal H_\tau)$.
			Reusing notation, we will denote these operators as $\Delta_n\colon \mathcal H_\tau \to F(\mathcal H_\tau)$ and $\Delta_n^* \colon F(\mathcal H_\tau) \to \mathcal H_\tau$, where $\ran \Delta_n \subset \mathcal H_\tau^{\otimes n}$ and $\ker \Delta_n^* = (\mathcal H_\tau^{\otimes n})^\perp$.
			Recall from \cref{sec:rkha_defs} that well-definition of $\Delta_n$ depends on coassociativity of $\Delta$.

			Next, let $\mathfrak B = B(H)$ as in \cref{sec:quantum}, and define the operator algebras $\mathfrak B_\tau := B(\mathcal H_\tau)$ and $\mathfrak F_\tau := B(F(\mathcal H_\tau))$ for $\tau>0$.
			Let also $\tilde G_{\tau, \sigma} : \mathcal H_\tau \to \mathcal H_\tau$, $\tau,\sigma>0$, be the compact, self-adjoint operators acting on the orthonormal basis $\{\psi_{j,\tau}\}_{j \in \mathbb N_0}$ of $\mathcal H_\tau$ as $\tilde G_{\tau, \sigma} \psi_{j,\tau} = \lambda_{j,\sigma} \psi_{j,\tau}$.
			Note that
			\begin{equation}
				\label{eq:g_intertwine}
				K^*_\tau \circ \tilde G_{\tau, \sigma} = G_\sigma \circ K^*_\tau
			\end{equation}
			which can be verified by application of the left- and right-hand sides to $\psi_{j,\tau}$.
			With $f \in \mathfrak A$ a real-valued classical observable and $ \pi f \in \mathfrak B$ the corresponding (self-adjoint) multiplication operator, we define, for each $n \in \{2, 3, \ldots\}$, a self-adjoint quantum observable $A_{f, \sigma, \tau, n} \in \mathfrak F_\tau$ acting on the Fock space, where
			\begin{equation}
				\label{eq:a_fock}
				A_{f, \sigma, \tau, n} = \tilde G_{\tau, \sigma}^{\otimes n}\Delta_{n-1} M_{f,\tau} \Delta_{n-1}^*\tilde G_{\tau, \sigma}^{\otimes n}, \quad M_{f,\tau} = K_\tau M_f K_\tau^* \in \mathfrak B_\tau, \quad M = \pi f \in \mathfrak B.
			\end{equation}

			Intuitively, we think of $M_{f,\tau}$ in~\eqref{eq:a_fock} as a smoothed multiplication operator obtained by conjugating $M_f$ by the integral operator $K_\tau \colon H \to \mathcal H_\tau$ associated with the RKHA $\mathcal H_\tau$.
			The quantum observables $A_{f, \sigma, \tau, n}$ are $n$-fold amplifications of $M_{f,\tau}$ obtained by conjugation by $\Delta_n$ followed by smoothing by $\tilde G_{\tau, \sigma}^{\otimes n}$.

			\subsubsection{Lifting of states}

			In light of \cref{lem:prob}, we consider approximating the evolution of expectations with respect to a probability density $p \in \mathfrak A_*$ with a strictly positive representative $\varrho \in \mathcal H_1$.
			For such a density the square root $\xi := \varrho^{1/2}$ and the $n$-th roots $\xi^{1/n}$, $n \in \mathbb N$, are all well-defined elements of $\mathcal H_1$ obtained from the holomorphic functional calculus (see \cref{sec:rkha_defs}).
			We lift the density $p$ to a rank-1 density operator $\nu_\tau \in \mathfrak F_{\tau*}$, $\tau \in (0, 1]$, with associated state vector $\eta_\tau \in F(\mathcal H_\tau)$, i.e., $\nu_\tau = \langle \eta_\tau, \cdot \rangle_{F(\mathcal H_\tau)} \eta_\tau$.
We define this vector as
\begin{equation}
	\label{eq:statevector}
	\eta_\tau = w_1 \frac{\xi}{\lVert \xi\rVert_{\mathcal H_\tau}} + w_2 \frac{\xi^{1/2} \otimes \xi^{1/2}}{\lVert \xi^{1/2}\rVert_{\mathcal H_\tau}^2} + w_3 \frac{\xi^{1/3} \otimes \xi^{1/3} \otimes \xi^{1/3}}{\lVert \xi^{1/3}\rVert_{\mathcal H_\tau}^3} + \ldots,
\end{equation}
where $w = (w_1, w_2, \ldots) \in \ell^2(\mathbb N)$ is any unit vector with strictly positive elements (i.e., $w_n >0$ and $\sum_{n=1}^\infty w_n^2 = 1$).
Intuitively, the unit vectors
\begin{equation}
	\label{eq:eta_root}
	\eta_{\tau, n} = \frac{\xi^{1/n} \otimes \cdots \otimes \xi^{1/n}}{\lVert \xi^{1/n}\rVert_{\mathcal H_\tau}^n}
\end{equation}
in~\eqref{eq:statevector} distribute a localized density function, $\varrho$, into an $n$-fold tensor product of coarser functions, $\xi^{1/n}$.
Note that $\eta_\tau$ has nonzero projections onto all gradings $\mathcal H^{\otimes n}_\tau \subset F(\mathcal H_\tau)$ of the Fock space.
As a result, the associated quantum state $\mathbb E_{\nu_\tau} \in S_*(\mathfrak F_\tau)$ is faithful on the quantum observables $A_{f, \sigma, \tau, n}$ from~\eqref{eq:a_fock}, meaning that $\mathbb E_{\nu_\tau} A_{f, \sigma, \tau, n}$ vanishes for $f \geq 0$ iff $f=0$.

\begin{rk}
	\label{rk:roots}
	Throughout this paper, our approach for dilating the vector $\xi \in \mathcal H_\tau$ to the Fock space is based on the roots $\xi^{1/n}$ in accordance with~\eqref{eq:eta_root}.
	It should be noted that many other choices are available besides this scheme.
	In particular, \eqref{eq:eta_root} can be replaced by
	\begin{displaymath}
		\eta_{\tau, n} = \frac{\xi^{(1)} \otimes \cdots \otimes \xi^{(n)}}{\lVert \xi^{(1)}\rVert_{\mathcal H_\tau} \cdots \lVert \xi^{(n)}\rVert_{\mathcal H_\tau}}
	\end{displaymath}
	for any collection of vectors $\xi^{(i)} \in \mathcal H_\tau$ whose pointwise product $\prod_{i=1}^n \xi^{(i)}$ is equal to $\xi$.
	We comment on possibilities of using such generalized dilation schemes in \cref{sec:conclusions}.
\end{rk}

\subsubsection{Lifting of dynamics}

We lift the unitary evolution operators $U^t_\tau\colon \mathcal H_\tau \to \mathcal H_\tau$ to unitaries $\tilde U^t_\tau\colon F(\mathcal H_\tau) \to F(\mathcal H_\tau)$ that act multiplicatively on the tensor algebra $T(\mathcal H_\tau)$,
\begin{equation*}
	\tilde U^t_\tau 1_{\mathbb C} = 1_{\mathbb C}, \quad \tilde U^t_\tau(g_1 \otimes \cdots \otimes g_n) = (U^t_\tau g_1) \otimes \cdots \otimes (U^t_\tau g_n),
\end{equation*}
where $1_{\mathbb C}$ is the unit vacuum vector in $F(\mathcal H_\tau)$ and $g_1, \ldots, g_n$ are vectors in $\mathcal H_\tau$.
Note that by~\eqref{eq:w_eig} the point spectrum of $U^t_\tau$, denoted by $\sigma_p(U^t_\tau)$, is a countable subset of the unit circle, but since, in general, the regularized generator $W_\tau$ does not satisfy the Leibniz rule (cf.\ \eqref{eq:leibniz}), $\sigma_p(U^t_\tau)$ does not have group structure.
On the other hand, by multiplicativity of $\tilde U^t_\tau$, the point spectrum $\sigma_p(\tilde U^t_\tau)$ \emph{is} a union of multiplicative subgroups of $S^1$ (generated by $\sigma_p(U^t_\tau)$).
Specifically, $\tilde U^t_\tau$ admits the eigendecomposition
\begin{equation*}
	\tilde U^t_\tau 1_{\mathbb C} = 1_{\mathbb C}, \quad \tilde U^t_\tau \zeta_{\vec\jmath,\tau} = e^{i\omega_{\vec\jmath,\tau} t} \zeta_{\vec\jmath,\tau},
\end{equation*}
where $\vec\jmath = (j_1, \ldots, j_N) \in \mathbb N_0^N$ is a multi-index of arbitrary length $N \in \mathbb N$, $\omega_{\vec\jmath,\tau} = \omega_{j_1,\tau} + \ldots + \omega_{j_N,\tau}$ are eigenfrequencies, and the corresponding eigenvectors $\zeta_{\vec\jmath,\tau} = \zeta_{j_1,\tau} \otimes \cdots \otimes \zeta_{j_N,\tau}$ form an orthonormal basis of $F(\mathcal H_\tau)$.

From the point of view of generators, multiplicativity on $T(\mathcal H_\tau)$ implies that the strongly continuous group of unitary operators $\tilde U^t_\tau$ is generated by a skew-adjoint operator $\tilde W_\tau \colon D(\tilde W_\tau) \to F(\mathcal H_\tau)$ that satisfies
\begin{equation*}
	\tilde W_\tau(g_1 \otimes \cdots \otimes g_n) = (W_\tau g_1) \otimes g_2 \otimes \cdots \otimes g_n + \ldots + g_1 \otimes \cdots \otimes g_{n-1} \otimes (W_\tau g_n)
\end{equation*}
for all $g_1, \ldots, g_n \in D(W_\tau)$.
In other words, $\tilde W_\tau$ is a dilation of $W_\tau$ to a free derivation with respect to the tensor product on $T(\mathcal H_\tau)$.
The point spectrum of $\tilde W_\tau$ consists of all integer linear combinations $i q_1 \omega_{1,\tau} + \ldots + i q_m \omega_{m,\tau}$, $q_1, \ldots, q_m \in \mathbb Z$, of elements of $\sigma_p(W_\tau)$.
As a result, the unitary evolution $\tilde U^t_\tau$ captures frequency content from the entire lattice of such linear combinations.

Intuitively, if $W_\tau$ well-approximates $V$ in the sense of the Leibniz rule~\eqref{eq:leibniz}, then $\omega_{\vec\jmath, \tau}$ can be viewed as an approximate eigenfrequency of $V$, and the pointwise product $\prod_{i=1}^N \zeta_{j_i,\tau} \in \mathcal H_\tau$ associated with eigenvector $\zeta_{\vec\jmath,\tau}$ behaves as an approximate corresponding eigenfunction.
Importantly, a finite collection of incommensurate eigenfrequencies $\omega_{i,\tau}$ and their corresponding eigenfunctions $\zeta_{i,\tau}$ can generate arbitrarily large families of such approximate Koopman eigenfrequencies and eigenfunctions.

\subsubsection{Quantum evolution}

The unitary operators $U^t_\tau$ induce evolution operators on quantum observables and states on the RKHA $\mathcal H_\tau$, $\mathcal U^t_\tau\colon \mathfrak B_\tau \to \mathfrak B_\tau$ and $\mathcal P^t_\tau\colon \mathfrak B_{\tau*} \to \mathfrak B_{\tau*}$, respectively, where $\mathcal U^t_\tau A = U^t_\tau A U^{t*}_\tau$ and $\mathcal P^t_\tau \rho = U^{t*}_\tau \rho U^t_\tau$.
Similarly, $\tilde U^t_\tau \colon F(\mathcal H_\tau) \to F(\mathcal H_\tau)$ induce quantum evolution operators on the Fock space, $\tilde{\mathcal U}^t_\tau\colon \mathfrak F_\tau \to \mathfrak F_\tau$ and $\tilde{\mathcal P}^t_\tau\colon \mathfrak F_{\tau*} \to \mathfrak F_{\tau*}$, where
\begin{equation*}
	\tilde{\mathcal U}^t_\tau A = \tilde U^t_\tau A \tilde U^{t*}_\tau, \quad \tilde{\mathcal P}^t_\tau \rho = \tilde U^{t*}_\tau \rho \tilde U^t_\tau.
\end{equation*}
Note that all of the quantum evolutions $\mathcal U^t_\tau$, $\mathcal P^t_\tau$, $\tilde{\mathcal U}^t_\tau$, and $\tilde{\mathcal P}^t_\tau$ are automatically multiplicative.
In particular, $\mathcal U^t_\tau$ and $\mathcal P^t_\tau$ are multiplicative even though $U^t_\tau$ is not.

For any $\tau,\sigma >0$ and $n \in \mathbb N$, the quantum observables $A_{f, \sigma, \tau, n}$ from~\eqref{eq:a_fock} evolve under $\tilde{\mathcal U}^t_\tau$ defined above.
In the dual (``Schr\"odinger") picture, the vector states $\mathbb E_{\nu_\tau} \in S_*(\mathfrak F_\tau)$ associated with $\eta_\tau$ from~\eqref{eq:statevector}evolve under $\tilde{\mathcal P}^t_\tau$.
Using any of these evolutions, we define
\begin{align*}
	g^{(t)}_{\sigma, \tau, n} & := \mathbb E_{\nu_\tau} (\tilde{\mathcal U}^t_\tau A_{f, \sigma, \tau, n}) \equiv \mathbb E_{\tilde{\mathcal P}^t_\tau \nu_\tau} A_{f, \sigma, \tau, n}                                                                                                                 \\
	                          & = \frac{w_n^2}{\lVert \xi^{1/n}\rVert^n_{\mathcal H_\tau}} \langle (\tilde G_{\tau, \sigma} U^{t*}_\tau \xi^{1/n})^{\otimes n}, \Delta_{n-1} M_{f,\tau} \Delta_{n-1}^*  (\tilde G_{\tau, \sigma} U^{t*}_\tau \xi^{1/n})^{\otimes n} \rangle_{F(\mathcal H_\tau)}.
\end{align*}
Defining, in addition, the normalization function
\begin{displaymath}
	C^{(t)}_{\sigma, \tau, n} :=\mathbb E_{\nu_\tau} (\tilde{\mathcal U}^t_\tau A_{\bm 1, \sigma, \tau, n}) \equiv \mathbb E_{\tilde{\mathcal P}^t_\tau \nu_\tau} A_{\bm 1, \sigma, \tau, n} = \frac{w_n^2}{\lVert \xi^{1/n}\rVert^n_{\mathcal H_\tau}} \lVert K_\tau^* (\tilde G_{\tau, \sigma}U^{t*}_\tau \xi^{1/n})^n\rVert^2_H,
\end{displaymath}
our Fock space approximation of the statistical evolution $\mathbb E_{\mathcal P^t p} f$ is given by the normalized quantum mechanical expectation
\begin{equation}
	\label{eq:ft_fock}
	f^{(t)}_{\sigma, \tau, n} := \frac{g^{(t)}_{\sigma, \tau, n}}{C^{(t)}_{\sigma, \tau, n}}.
\end{equation}
Equivalently, we have
\begin{align*}
	f^{(t)}_{\sigma, \tau, n} & = \frac{\langle\Delta_{n-1}^*(\tilde G_{\tau, \sigma}U^{t*}_\tau \xi^{1/n})^{\otimes n}, M_{f,\tau} \Delta_{n-1}^*(\tilde G_{\tau, \sigma}U^{t*}_\tau \xi^{1/n})^{\otimes n} \rangle_{\mathcal H_\tau}}{\lVert K_\tau^*( U^{t*}_\tau \xi^{1/n})^n\rVert^2_H} \\
	                          & = \frac{\langle K_\tau^* (\tilde G_{\tau, \sigma} U^{t*}_\tau \xi^{1/n})^n, (\pi f) K_\tau^*(\tilde G_{\tau, \sigma} U^{t*}_\tau \xi^{1/n})^n \rangle_H}{\lVert K_\tau^* (\tilde G_{\tau, \sigma} U^{t*}_\tau \xi^{1/n})^n\rVert^2_H}.
\end{align*}
Note that $f^{(t)}_{\sigma, \tau, n}$ is independent of the weights $w$ employed in the definition of the state vector $\eta_\tau$ (see~\eqref{eq:statevector}).
Moreover, one can verify (using a similar approach as in~\eqref{eq:g_sigma_lim} below to take the limit) that as $\sigma \to 0^+$, $f^{(0)}_{\sigma, \tau, n}$ converges to the expectation $\mathbb E_p f$ for any (fixed) $n \in \{ 2, 3, \ldots \}$ and $\tau >0$,
\begin{equation}
	\label{eq:sigma_lim_t0}
	\lim_{\sigma \to 0^+} f^{(0)}_{\sigma, \tau, n} = \lim_{\sigma \to 0^+} \frac{\langle K_\tau^* (\tilde G_{\tau, \sigma} \xi^{1/n})^n, (\pi f) K_\tau^*(\tilde G_{\tau, \sigma} \xi^{1/n})^n \rangle_H}{\lVert K_\tau^* (\tilde G_{\tau, \sigma} \xi^{1/n})^n\rVert^2_H}
	= \frac{\langle K_\tau^* \xi, (\pi f) K_\tau^*\xi \rangle_H}{\lVert K_\tau^* \xi\rVert^2_H} = \frac{\langle p^{1/2}, f p^{1/2} \rangle_H}{\lVert p^{1/2}\rVert^2_H} = \mathbb E_p f.
\end{equation}
This means that in the limit of $\sigma \to 0^+$ the trajectories $t \mapsto f^{(t)}_{\sigma, \tau, n}$ start from the same initial value at $t=0$ but follow different subsequent evolutions for different $n$ and $\tau$.

The behavior of the quantum expectations $\mathbb E_{\tilde{\mathcal P}^t_\tau \rho_\tau} A_{f, \sigma, \tau, n}$ and $\mathbb E_{\tilde{\mathcal P}^t_\tau \rho_\tau} A_{\bm 1, \sigma, \tau, n}$ in~\eqref{eq:ft_fock} giving $f^{(t)}_{\sigma, \tau, n}$ through $g^{(t)}_{\tau, n,\epsilon}$ and $C_{\sigma, \tau, n}^{(t)}$, respectively, can be further understood as follows.
\begin{itemize}[wide]
	\item As mentioned above, the component of the vector $\eta_\tau$ in the subspace $\mathcal H_\tau^{\otimes n} \subset F(\mathcal H_\tau)$ is a scalar multiple of $(\xi^{1/n})^{\otimes n}$.
	      Evolving the density operator $\rho_\tau$ by $\tilde{\mathcal P^t_\tau}$ evolves this vector to $(U^{t*}_\tau\xi^{1/n})^{\otimes n}$.
	      This evolution captures frequency content from all $n$-fold integer linear combinations of the eigenfrequencies $\omega_{i,\tau}$ (i.e., all frequencies $\omega_{\vec\jmath,\tau}$ with $\vec\jmath \in \mathbb N_0^n$).
	\item With the above time evolution of the state, multiplicativity of $K_\tau^*$, the intertwining relation~\eqref{eq:g_intertwine}, and up to a proportionality constant that is eliminated upon division by $C_{\sigma, \tau, n}^{(t)}$, the quantum expectation $\mathbb E_{\tilde{\mathcal P}^t_\tau \rho_\tau} A_{f, \sigma, \tau, n}$ reduces to
	      \begin{multline}
		      \label{eq:quantum_approx1}
		      \langle \Delta_{n-1}^*(\tilde G_{\tau, \sigma} U^{t*}_\tau \xi^{1/n})^{\otimes n}, M_{f,\tau} \Delta_{n-1}^* (\tilde G_{\tau, \sigma} U^{t*}_\tau \xi^{1/n})^{\otimes n} \rangle_{\mathcal H_\tau} \\
		      = \langle (G_\sigma K_\tau^* U^{t*}_\tau\xi^{1/n})^n, (\pi f) (G_\sigma K_\tau^*U^{t*}_\tau\xi^{1/n})^n \rangle_H.
	      \end{multline}
	      In essence, this corresponds (again up to proportionality) to expectation of the multiplication operator $\pi f \in \mathfrak B$ with respect to a vector state induced by the vector $(K_\tau^* G_\sigma U^{t*}_\tau\xi^{1/n})^n \in \mathfrak A \subset H$.
	      This latter vector reassembles the dynamically evolved factors $\tilde G_{\tau, \sigma} U^{t*}_\tau\xi^{1/n} \in \mathcal H_\tau$ from the tensor product $(\tilde G_{\tau, \sigma} U^{t*}_\tau\xi^{1/n})^{\otimes n} \in \mathcal H_\tau$ into a pointwise product, while capturing higher-order frequency content from $\omega_{\vec\jmath,\tau}$.
\end{itemize}

As $\tau \to 0^+$, the restriction of $K_\tau^* U^{t*}_\tau$ on $\mathcal H_1$ converges strongly to $U^t \iota$; see \cref{lem:koopman_rkhs_approx}.
In addition, we have:
\begin{lem}
	\label{lem:koopman_rkhs_approx_pwr}
	For every $h \in \mathcal H_1$, $n \in \mathbb N$, and $\sigma >0$, $(G_\sigma K_\tau^* U^t_\tau h)^n$ converges, as $\tau \to 0^+$, to $(G_\sigma U^t \iota h)^n$ in the norm of $H$.
\end{lem}

\begin{proof}
	First, note that the vectors $a_\tau := G_\sigma K_\tau^* U^t_\tau h$ are uniformly bounded in $L^\infty(\mu)$ norm over $\tau \in (0,1)$ since
	\begin{equation*}
		\lVert G_\sigma K_\tau^* U^t_\tau h \rVert_{L^\infty(\mu)} \leq \lVert k_\tau \rVert_{C(X\times X)} \lVert K_\tau^* U^t_\tau h \rVert_{L^1(\mu)}  \leq \lVert k_\tau \rVert_{C(X\times X)} \lVert K_\tau^* U^t_\tau h \rVert_H
	\end{equation*}
	and $K_\tau^* U^t_\tau h$ norm-converges in $H$ to $a := U^t \iota h$ by \cref{lem:koopman_rkhs_approx}(2).
	As, a result, with $C =\sup_{\tau \in (0,1)} \lVert a_\tau \rVert_{L^\infty(\mu)}$ and $n \in \mathbb N$ we have
	\begin{align*}
		\lVert a_\tau^{n+1} - a^{n+1} \rVert_{H} & = \lVert a_\tau^{n+1} - a a_\tau^n + a a_\tau^n - a^{n+1} \rVert_H                                                                      \\
		                                         & \leq \lVert a_\tau^n \rVert_{L^\infty(\mu)} \lVert a_\tau - a \rVert_H+ \lVert a \rVert_{L^\infty(\mu)} \lVert a_\tau^n - a^n \rVert_H, \\
		                                         & \leq C^n \lVert a_\tau - a \rVert_H+ \lVert a \rVert_{L^\infty(\mu)} \lVert a_\tau^n - a^n \rVert_H,
	\end{align*}
	and the claim of the lemma follows by induction from the base case $n=0$.

\end{proof}

The following theorem uses this fact in conjunction with multiplicativity of $U^t$ to establish that the Fock-space-based approximation $f^{(t)}_{\sigma, \tau, n}$ from~\eqref{eq:ft_fock} recovers the true statistical evolution of $f$ for any $n \in \mathbb N$.

\begin{thm}
	\label{thm:tau_conv}
	Assume that the kernels $k_\tau$ and regularized generators $V_\tau$ have \crefrange{prty:K1}{prty:K7} and \crefrange{prty:V1}{prty:V6}, respectively.
	Suppose also that $p \in \mathfrak A_*$ is a probability density with a strictly positive representative $\varrho \in \mathcal H_1$.
	Then, for every classical observable $f \in \mathfrak A$, evolution time $t \in \mathbb R$, and $n \in \mathbb N$, the Fock space evolution $f^{(t)}_{\sigma, \tau, n}$ converges to the classical statistical evolution $f^{(t)} = \mathbb E_{P^tp} f$ in the iterated limit of $\sigma \to  0^+$ after $\tau \to 0^+$, i.e., $\lim_{\sigma \to 0^+}\lim_{\tau\to 0^+} f^{(t)}_{\sigma, \tau, n} = f^{(t)}$.
\end{thm}

\begin{proof}
	Since $\xi^{1/n} \in \mathcal H_1$, \cref{lem:koopman_rkhs_approx_pwr} implies that $\lim_{\tau\to 0^+} (G_\sigma K_\tau^* U^{t*}_\tau \xi^{1/n}) = (G_\sigma U^{t*} \iota \xi^{1/n})^n$, where the limit was taken in $H$ norm.
	Moreover, since $G_\sigma$ is an averaging operator and $U^{t*}$ acts as an isometry on $L^\infty(\mu)$, we have
	\begin{equation*}
		\lVert G_\sigma U^{t*} \iota \xi^{1/n} \rVert_{L^\infty(\mu)} \leq \lVert \xi^{1/n} \rVert_{C(X)} \leq \max \{1, \lVert \xi \rVert_{C(X)} \}.
	\end{equation*}
	As a result, $\{ G_\sigma U^{t*} \iota \xi^{1/n}\}_{\sigma > 0}$ is an $L^\infty(\mu)$-norm-bounded family that converges as $\sigma \to 0^+$ to $U^{t*} \iota \xi^{1/n}$ in $H$ norm.
	Thus, by a similar argument as in the proof of \cref{lem:koopman_rkhs_approx_pwr} and multiplicativity of $U^{t*} = U^{-t}$ it follows that
	\begin{equation}
		\label{eq:g_sigma_lim}
		\lim_{\sigma \to 0^+} (G_\sigma U^{t*} \iota \xi^{1/n})^n = (U^{t*} \iota \xi^{1/n})^n = U^t \iota \xi = U^t p^{1/2}.
	\end{equation}
	Using~\eqref{eq:quantum_approx1} and the classical--quantum correspondence~\eqref{eq:classical_quantum_compat}, we deduce
	\begin{align*}
		\lim_{\sigma \to 0^+}\lim_{\tau\to 0^+} f^{(t)}_{\sigma, \tau, n} & = \lim_{\sigma \to 0^+} \lim_{\tau\to 0^+}\frac{\langle (G_\sigma K_\tau^* U^{t*}_\tau \xi^{1/n})^n, (\pi f) (G_\sigma K_\tau^* U^{t*}_\tau \xi^{1/n})^n \rangle_H}{\lVert G_\sigma K_\tau^* (U^{t*}_\tau \xi^{1/n})^n\rVert^2_H} \\
		                                                                  & = \frac{\langle U^{t*} p^{1/2}, (\pi f) U^{t*} p^{1/2} \rangle_H}{\lVert U^{t*} p^{1/2} \rVert^2_H} = \frac{\mathbb E_{\mathcal P^t(\Gamma(p))}(\pi f)}{\lVert p^{1/2} \rVert_H} = f^{(t)}. \qedhere
	\end{align*}
\end{proof}

\subsection{Finite-rank compression}
\label{sec:finite_rank}

For $d \in \mathbb N$, let $\mathcal Z_{\tau,d}$ be the $(2d+1)$-dimensional subspace of $\mathcal H_\tau$ spanned by $\zeta_{0,\tau}, \ldots, \zeta_{2d,\tau}$, and $F(\mathcal Z_{\tau,d}) \subset F(\mathcal H_\tau)$ the Fock space generated by $\mathcal Z_{\tau,d}$.
Note that $\mathcal Z_{\tau,d}$ is spanned by the constant function $\zeta_{0,\tau} = \bm 1$ (the unit of $\mathcal H_\tau$) and, for $1\leq j \leq d$, complex-conjugate pairs of eigenfunctions $\{\zeta_{2j,\tau}, \zeta_{2j-1,\tau}\}$ with corresponding eigenfrequencies $\omega_{2j,\tau} = - \omega_{2j-1,\tau}$.
Moreover, the image of $\mathcal Z_{\tau,d}^{\otimes n} \subset F(\mathcal Z_{\tau,d})$ is spanned by polynomials in $\zeta_{0,\tau}, \ldots, \zeta_{2d,\tau}$ of degree up to $n$.

We approximate the observable evolution $f^{(t)}_{\sigma, \tau, n}$ from~\eqref{eq:ft_fock} by projecting the state vector $\eta_\tau$ onto the finitely generated Fock space $F(\mathcal Z_{\tau,d})$.
Our approach results in approximations of $g^{(t)}_{\tau, n}$ and $C_{\sigma, \tau, n}^{(t)}$ that can be efficiently evaluated using TTNs, as follows.

Let $Z_{\tau,d}\colon \mathcal H_\tau \to \mathcal H_\tau$ and $\tilde Z_{\tau,d} \colon F(\mathcal H_\tau) \to F(\mathcal H_\tau)$ be the orthogonal projections with $\ran Z_{\tau,d} = \mathcal Z_{\tau,d}$ and $\ran \tilde Z_{\tau,d} = F(\mathcal Z_{\tau,d})$, respectively.
Defining the projected state vector $\eta_{\tau,d} = \tilde Z_{\tau, d} \eta_\tau / \lVert \tilde Z_{\tau,d} \eta_\tau \rVert_{F(\mathcal H_\tau)}$, we obtain a compressed density operator $\nu_{\tau,d} := \langle \eta_{\tau,d}, \cdot \rangle_{F(\mathcal H_\tau)} \eta_{\tau,d}$ that will serve as our approximation of $\nu_\tau$.
As $d \to \infty$, the density operators $\nu_{\tau,d}$ converge to $\nu_\tau$ in the trace norm of $\mathfrak F_{\tau*}$.

The density operators $\nu_{\tau,d}$ lead to the following approximation of $f^{(t)}_{\sigma, \tau, n}$:
\begin{equation}
	\label{eq:ft_fock_d}
	f^{(t)}_{\sigma, \tau, n, d} = \frac{g^{(t)}_{\sigma, \tau, n, d}}{C_{\sigma, \tau, n, d}^{(t)}}, \quad g^{(t)}_{\sigma, \tau, n, d} = \mathbb E_{\tilde{\mathcal P}^t_\tau \nu_{\tau,d}} A_{f, \sigma, \tau, n}, \quad C_{\sigma, \tau, n, d}^{(t)} = \mathbb E_{\tilde{\mathcal P}^t_\tau \nu_{\tau,d}} A_{\bm 1, \sigma, \tau, n}.
\end{equation}
A key property of this approximation is that it reduces to computing expectations of finite-rank quantum observables.
Indeed, noticing that $\tilde U^t_\tau$ and $\tilde Z_{\tau,d}$ commute, we have that $g^{(t)}_{\sigma, \tau, n, d}$ equals, up to an unimportant proportionality constant,
\begin{multline*}
	\langle  \tilde G_{\tau, \sigma}^{\otimes n} \tilde U^{t*}_\tau \tilde Z_{\tau,d} \eta_\tau, \Delta_{n-1} M_{f,\tau}\Delta_{n-1}^* \tilde G_{\tau, \sigma}^{\otimes n} \tilde U^{t*}_\tau \tilde Z_{\tau,d} \eta_\tau\rangle_{F(\mathcal H_\tau)} \\
	\begin{aligned}
		 & = \langle \tilde G_{\tau, \sigma}^{\otimes n} \tilde Z_{\tau,d} \tilde U^{t*}_\tau \eta_\tau, \Delta_{n-1} M_{f,\tau}\Delta_{n-1}^* \tilde G_{\tau, \sigma}^{\otimes n} \tilde Z_{\tau,d} U^{t*}_\tau \eta_\tau\rangle_{F(\mathcal H_\tau)} \\
		 & = \langle \tilde U^{t*}_\tau \eta_\tau, A_{f, \sigma, \tau, n, d} \tilde U^{t*}_\tau \eta_\tau\rangle_{F(\mathcal H_\tau)},
	\end{aligned}
\end{multline*}
where
\begin{equation}
	\label{eq:quantum_obs_d}
	A_{f, \sigma, \tau, n, d} := \tilde Z_{\tau,d} A_{f, \sigma, \tau, n} \tilde Z_{\tau,d} \in \mathfrak F_\tau
\end{equation}
is a quantum observable on the Fock space that has finite rank.
The latter, stems from the fact that the restriction of $\tilde Z_{\tau,d}$ on $\mathcal H_\tau^{\otimes n}$ is equal to $Z_{\tau,d}^{\otimes n}$; in particular, $\ran\tilde Z_{\tau,d}\rvert_{\mathcal H_\tau^{\otimes n}} = \mathcal Z_{\tau,d}^{\otimes n}$, and thus $\tilde Z_{\tau,d}\rvert_{\mathcal H_\tau^{\otimes n}}$ has rank $(2d+1)^n$.
Since $\ran \Delta_{n-1} \subset \mathcal H_\tau^{\otimes n}$, it follows that $\Delta_{n-1}^* \tilde Z_{\tau,d} = (\tilde Z_{\tau,d} \Delta_{n-1})^*$ is a finite-rank operator, and thus so is $A_{f, \sigma, \tau, n, d}$.

By multiplicativity of $\tilde U^t_\tau$, we also have
\begin{multline}
	\label{eq:quantum_approx2}
	\langle \tilde G_{\tau, \sigma}^{\otimes n} U^{t*}_\tau \tilde Z_{\tau,d} \eta_\tau, \Delta_{n-1} M_{f,\tau}\Delta_{n-1}^* \tilde G_{\tau, \sigma}^{\otimes n} \tilde U^{t*}_\tau \tilde Z_{\tau,d} \eta_\tau\rangle_{F(\mathcal H_\tau)}\\
	\begin{aligned}
		 & \propto \langle \Delta_{n-1}^*(\tilde G_{\tau, \sigma} U^{t*}_\tau\xi_{\tau, n,d})^{\otimes n}, M_{f,\tau}\Delta_{n-1}^*(\tilde G_{\tau, \sigma} U^{t*}_\tau\xi_{\tau, n,d})^{\otimes n}\rangle_{\mathcal H_\tau} \\
		 & = \langle K^*_\tau (\tilde G_{\tau, \sigma} U^{t*}_\tau \xi_{\tau, n,d})^n, (\pi f)K^*_\tau (\tilde G_{\tau, \sigma} U^{t*}_\tau \xi_{\tau, n,d})^n\rangle_H,
	\end{aligned}
\end{multline}
where $\xi_{\tau, n,d} = Z_{\tau,d} \xi_\tau^{1/n} / \lVert Z_{\tau,d} \xi_\tau^{1/n}\rVert_{\mathcal H_\tau} $ and the proportionality constant in the first line is again unimportant.
Comparing~\eqref{eq:quantum_approx1} and~\eqref{eq:quantum_approx2}, we see that the approximation $f^{(t)}_{\sigma, \tau, n, d}$ essentially involves approximating the vector state on $\mathfrak B$ induced by the pointwise product $K_\tau^*(G_\sigma U^{t*}_\tau \xi_\tau^{1/n})^n$ by the pointwise product $K_\tau^*(G_\sigma U^{t*}_\tau \xi_{\tau, n,d})^n$ involving the projected $n$-th root $\xi_{\tau, n,d}$.
Note that
\begin{equation*}
	K_\tau^*(\tilde G_{\tau, \sigma} U^{t*}_\tau \xi_{\tau, n,d})^n = (K_\tau^* \tilde G_{\tau, \sigma} U^{t*}_\tau \xi_{\tau, n,d})^n = (G_\sigma K_\tau^* U^{t*}_\tau \xi_{\tau, n,d})^n
\end{equation*}
by multiplicativity of $K_\tau^*$ and~\eqref{eq:g_intertwine}.

Using the tensor product terms $(G_\sigma U^{t*}_\tau\xi_{\tau, n,d})^{\otimes n}$ appearing in the right-hand side of~\eqref{eq:quantum_approx2} and associativity of $\Delta_{n-1}^*$, we can express the computation of $f^{(t)}_{\sigma, \tau, n, d}$ through a tree tensor network where the operator $\Delta^*$ is applied recursively to assemble the pointwise product $(G_\sigma U^{t*}_\tau\xi_{\tau, n,d})^n$ without requiring explicit formation of the $(2d+1)^n$-dimensional tensor product $(G_\sigma U^{t*}_\tau\xi_{\tau, n,d})^{\otimes n}$.
See \cref{fig:network} for a diagrammatic illustration.
This tensor network architecture forms the centerpiece of our Koopman operator approximation framework.
The tensor network in \cref{fig:network} represents a combination of infinite-dimensional operators that must be approximated by a finite cutoff. This cutoff and network resembles a tensor train decomposition (see \cite{Oseledets11}). While tensor train methods may be beneficial here, we will develop an approximation specific to the tensor network Koopman approach.

\subsubsection{Two forms of approximation}

For any fixed $\sigma,\tau>0$ and $n \in \mathbb N$, the Fock space approximation $f^{(t)}_{\sigma, \tau, n, d}$ from~\eqref{eq:ft_fock_d} converges to $f^{(t)}_{\sigma, \tau, n}$ as $d\to\infty$.
This convergence is an elementary consequence of the fact that the eigenfunctions $\zeta_{j,\tau}$ form an orthonormal basis of $\mathcal H_\tau$---this means that the subspaces $\mathcal Z_{\tau,d}$ increase to $\mathcal H_\tau$, and thus that the associated Fock spaces $F(\mathcal Z_{\tau,d})$ increase to $F(\mathcal H_\tau)$, recovering the formula for $f^{(t)}_{\sigma, \tau, n}$ in~\eqref{eq:ft_fock} which holds for any $n \mathbb N$.

Besides approximations of this type, one of our primary interests in this work is on approximations performed at \emph{fixed} $d$ with \emph{increasing} $n \to \infty$.
In this scenario, the approximation space is a fixed, infinite-dimensional but strict subspace $F(\mathcal Z_{\tau,d})$ of the full Fock space $F(\mathcal H_\tau)$, and within this subspace we compute expectations of the quantum observables $A_{f, \sigma, \tau, n, d}$ from~\eqref{eq:quantum_obs_d}.
As $n$ increases, these observables have increasingly high rank and access increasingly high gradings $\mathcal Z_{\tau,d}^{\otimes n}$ of $F(\mathcal Z_{\tau,d})$.
To the extent that $W_\tau$ sufficiently well-approximates $V$ for $n$-fold products of eigenfunctions $\zeta_{0,\tau}, \ldots, \zeta_{2d,\tau}$ to behave approximately as Koopman eigenfunctions (as would have been the case if $W_\tau$ perfectly recovered $V$ as a derivation obeying the Leibniz rule), increasing $n$ at sufficiently large, but fixed, $d$ can potentially reconstruct the full evolution $f^{(t)}$.
For instance, if $V$ is the generator of a circle rotation, $\Phi^t(\theta) = \theta + \alpha t \mod 2\pi$, the eigenfunctions of $V$ are Fourier functions, $\phi_j(\theta) = e^{ij\theta}$ with $j \in \mathbb Z$, so if $\zeta_{0,\tau}$, $\zeta_{1,\tau}$, and $\zeta_{2,\tau}$ well-approximate $\phi_0$, $\phi_1$, and $\phi_{-1}$, respectively, fixing the Fock space $F(\mathcal Z_{\tau,d})$ with $d=1$ and increasing $n$ while decreasing $\tau$ and $\sigma$ could potentially approximate $f^{(t)}$ to arbitrarily high precision.
In \cref{thm:torus_rotation_convergence} below, we show that such a result indeed holds for $d=2$ in a joint limit of $n \to \infty$ and $\tau\to 0^+$.

\begin{example*}[Circle rotation]
	In \cref{sec:circlerot}, the Fock space approximation $f^{(t)}_\text{Fock}$ of the evolution of the von Mises density under the circle rotation was based on~\eqref{eq:ft_fock_d} with dimension parameter $d=2$ and a grading of $n=10$.
	The improved approximation accuracy over the classical and quantum approximation methods operating at the same $d$ (see \cref{fig:evo_circlerot}) provides a numerical illustration of \cref{thm:tau_conv}, as well as the related \cref{thm:torus_rotation_convergence} in \cref{sec:convergence_torus} below.
\end{example*}

We have two major motivations to explore approximations of this latter type:
\begin{enumerate}[wide]
	\item In practical applications, the eigenfrequencies $\omega_{j,\tau}$ and corresponding eigenfunctions $\zeta_{j,\tau}$ can seldom be computed exactly.
	      Instead, one has to resort to numerical approximations obtained from some type of discretization of the eigenvalue problem for $W_\tau$.
	      Obtaining high-fidelity approximations of large numbers $d$ of eigenfunctions can be computationally costly, and in certain important cases (e.g., data-driven approximations using a fixed amount of available training data), altogether infeasible.
	      An algebraic approach for building high-dimensional approximation spaces from a modest number, $2d +1$, of basic functions (here, via tensor products) is a promising strategy for amplifying the utility of numerical Koopman spectral decompositions.
	\item Working at fixed $d$ as $n$ increases allows one to leverage the tensor network construction in \cref{fig:network}, and access approximation spaces $\mathcal Z_{\tau,d}^{\otimes n}$ of exponentially high dimension, $(2d+1)^n$, by composing operations on $(2d+1)$-dimensional vectors.
\end{enumerate}

\section{Numerical implementation}
\label{sec:num_impl}

We now consider the numerical implementation of the tensor network scheme described in \cref{sec:tensornet}.

\subsection{Generator approximation}
\label{sec:num_impl_generator}

Our implementation begins by discretizing the eigenvalue problem for $W_\tau$ in~\eqref{eq:w_eig} in finite-dimensional subspaces $\mathcal H_{\tau, m}$ of the RKHA $\mathcal H_\tau$,
\begin{equation}
	\label{eq:w_tau_eig}
	W_{\tau,m} \zeta_{j,\tau,m} = i \omega_{j,\tau, m} \zeta_{j,\tau, m}.
\end{equation}
Here, $W_{\tau,m}\colon \mathcal H_{\tau, m} \to \mathcal H_{\tau,m} $ is a skew-adjoint operator that approximates $W_\tau$, $\omega_{j,\tau,m} \in \mathbb R$ are its eigenfrequencies, and $ \zeta_{j,\tau,m}$ are corresponding eigenfunctions that form an orthonormal basis of $\mathcal H_{\tau,m}$.
The specific choice of approximation spaces $\mathcal H_{\tau,m}$ will depend on the application at hand, but in general we require that $W_{\tau,m}$ converges spectrally to $W_\tau$ as $m \to \infty$.
Note that since $\mathcal H_\tau$ is a space of functions (as opposed to equivalence classes of functions in $L^p$ spaces), the approximate eigenfunctions $\zeta_{j,\tau,m}$ are everywhere-defined functions that can be represented computationally as function objects (as opposed to arrays containing values of functions).
Similarly, $W_{\tau,m}$ can be implemented natively as an operator on functions using automatic differentiation of the reproducing kernel of $\mathcal H_\tau$; see \cref{app:numerical}.

In what follows, we let $W_{\tau,m} \colon \mathcal H_\tau \to \mathcal H_\tau$ represent finite-rank approximations of either $W_\tau$ or $\tilde W_{z,\tau}$ from \cref{app:compact_approx,app:compact_res}.
Moreover, we will continue to use $\omega_{j,\tau,m} \in \mathbb R$ and $\zeta_{j,\tau,m} \in \mathcal H_\tau$ to denote corresponding eigenfrequencies and eigenfunctions, respectively.
We also let $U^t_{\tau,m} = e^{t W_{\tau,m}}$, $t \in \mathbb R$, denote the unitary evolution operators generated by $W_{\tau,m}$.

\subsection{Tensor network scheme}

With the approximations for $W_\tau$ and its eigenvalues/eigenfunctions from \cref{sec:num_impl_generator}, the numerical implementation of the tensor network scheme from \cref{sec:finite_rank} based on subspaces $\mathcal Z_{\tau,d} \subset \mathcal H_\tau$ proceeds in a structurally similar manner, replacing $\mathcal Z_{\tau,d}$ by $\mathcal Z_{\tau,d,m} = \spn \{ \zeta_{0,\tau,m}, \ldots, \zeta_{2d,\tau,m} \} \subseteq \mathcal H_{\tau,m}$ and the unitaries $U^t_\tau$ by $U^t_{\tau,m} := e^{t W_{\tau,m}}$.

For $d \in \mathbb N$ chosen such that $2d + 1 \leq \dim \mathcal H_{\tau,m}$, consider the subspace $\mathcal Z_{\tau,d,m} = \spn \{ \zeta_{0,\tau,m}, \ldots, \zeta_{2d,\tau,m} \} \subseteq \mathcal H_{\tau,m} $, and let $Z_{\tau,d,m}\colon \mathcal H_\tau \to \mathcal H_\tau$ be the orthogonal projection with $\ran Z_{\tau,d,m} = \mathcal Z_{\tau,d,m}$; that is, $Z_{\tau,d,m} = \sum_{j=0}^{2d} \langle \zeta_{j,\tau,m}, f\rangle_{\mathcal H_\tau} \zeta_{j,\tau,m}$.
Defining the projected state vector
\begin{equation*}
	\xi_{\tau, n,d,m} = \frac{1}{\lVert Z_{\tau,d,m} \xi^{1/n}\rVert_{\mathcal H_\tau}} Z_{\tau,d,m} \xi^{1/n},
\end{equation*}
we approximate the functions $U^{t*}_\tau\xi_{\tau, n,d} \in \mathcal H_\tau$ appearing in~\eqref{eq:quantum_approx2} by $U^{t*}_{\tau,m}\xi_{\tau, n,d,m} \in \mathcal H_{\tau,m}$.
We similarly approximate $(U^{t*}_\tau\xi_{\tau, n,d})^n$ by the $n$-fold pointwise function product $(U^{t*}_{\tau,m}\xi_{\tau, n,d,m})^n$, but note that since, in general, the subspace $\mathcal H_{\tau,m}$ is not an algebra, $(U^{t*}_{\tau,m}\xi_{\tau, n,d,m})^n$ must be viewed as an element of the RKHA $\mathcal H_\tau \supseteq \mathcal H_{\tau,m}$.
Similarly to the approximations $W_{\tau,m}$ of $W_\tau$, these approximations are again implemented using function objects, and converge as $m\to \infty$ to $(U^{t*}_\tau\xi_{\tau, n,d})^n$ in the norm of $\mathcal H_\tau$.

The remaining ingredients needed to approximate the right-hand side of \eqref{eq:quantum_approx2} are approximations of (i) the multiplication operator $\pi f \in \mathfrak B \equiv B(H)$; and (ii) $L^2$ inner products $\langle \cdot, \cdot \rangle_H$ with respect to the invariant measure $\mu$.
To that end, for each $l \in \mathbb N$ we introduce a sequence of points $x_0^{(l)}, \ldots, x_{N_l-1}^{(l)} \in G$ whose associated sampling measures
\begin{equation}
	\label{eq:sampling_meas}
	\mu_l := \frac{1}{N_l} \sum_{i=0}^{N_l-1} \delta_{x_i}
\end{equation}
converge as $l\to \infty$ to the invariant measure $\mu$ in weak-$^*$ sense.
This means $\lim_{l \to \infty} \int_G \tilde f \, d\mu_l = \int_G \tilde f \, d\mu$ for every continuous function $\tilde f \colon G \to \mathbb C$.
Possible choices for the $x_i^{(l)}$ are grid nodes on $G$ of increasingly small spacing as $l$ increases, or trajectories of points under discrete-time dynamics having $\mu$ as an ergodic invariant measure (obtained, e.g., by sampling the continuous-time dynamical flow $\Phi^t$ at an interval $\Delta t > 0$).
Assuming that $f$ has a representative $\tilde f \in C(G)$ (we will relax this assumption below), we further require that the values $y_i^{(l)} = \tilde f(x_i^{(l)}) \in \mathbb R$ are known on the $x_i^{(l)}$.

For each $l \in \mathbb N$ the samples $y_i^{{(l)}}$ induce an element $f_l$ in the finite-dimensional von Neumann algebra $\mathfrak A_l = L^\infty(\mu_l)$.
Letting $H_l = L^2(\mu_l)$ and $\mathfrak B_l = B(H_l)$, $f_l$ induces in turn a multiplication operator $(\pi_l f_l) \in \mathfrak B_l$, where $ (\pi_l g) h = g h$ for every $g \in \mathfrak A_l$ and $h \in H_l$.
Moreover, analogously to $K_\tau \colon H \to \mathcal H_\tau$, for every $l \in \mathbb N$ there is a finite-rank integral operator $K_{\tau,l}\colon H_l \to \mathcal H_\tau$, where
\begin{equation}
	\label{eq:k_op_discrete}
	K_{\tau,l} f = \int_G k_\tau(\cdot, x) f(x) \, d\mu_l(x) = \frac{1}{l}\sum_{i=0}^{N_l-1} k_\tau(\cdot, x^{(l)}_i) f(x^{(l)}_i).
\end{equation}
The adjoint $K_{\tau,l}^*$ is a multiplicative map that implements restriction (sampling) of functions in $\mathcal H_\tau$ to equivalence classes of functions in $H_l$.

With these definitions, we approximate $\langle K^*_\tau (\tilde G_{\tau, \sigma} U^{t*}_\tau \xi_{\tau, n,d})^n, (\pi f)K^*_\tau (\tilde G_{\tau, \sigma} U^{t*}_\tau \xi_{\tau, n,d})^n\rangle_H$
in the last line of~\eqref{eq:quantum_approx2} by
\begin{displaymath}
	\langle (\tilde G_{\tau, \sigma} U^{t*}_\tau \xi_{\tau, n,d})^n, M_{f,\tau,l} (\tilde G_{\tau, \sigma} U^{t*}_\tau \xi_{\tau, n,d})^n\rangle \equiv \langle K^*_{\tau,l} (\tilde G_{\tau, \sigma} U^{t*}_{\tau, m} \xi_{\tau, n,d,m})^n, (\pi_l f_l)K^*_{\tau,l} (\tilde G_{\tau, \sigma} U^{t*}_{\tau, m} \xi_{\tau, n,d,m})^n\rangle_{H_l},
\end{displaymath}
where
\begin{equation}
	\label{eq:discrete_mult_op}
	M_{f,\tau,l} = K_{\tau,l} \pi_l f_l, K_{\tau,l}^*.
\end{equation}
This approximation converges as $l,m\to \infty$, where the limits can be taken in either order.
Using this approximation of~\eqref{eq:quantum_approx2}, we define $g^{(t)}_{\sigma, \tau, n, d, m, l}\colon X \to \mathbb R$ and $C_{\sigma, \tau, n, d, m, l}^{(t)} >0$ analogously to $g^{(t)}_{\sigma, \tau, n, d}$ and $C_{\sigma, \tau, n, d}^{(t)}$ from~\eqref{eq:ft_fock_d}, respectively, using the evolution operators $\tilde{\mathcal U}^t_{\tau,m}: \mathfrak F_\tau \to \mathfrak F_\tau$ and $\tilde{\mathcal P}^t_{\tau,m}: \mathfrak F_{\tau*} \to \mathfrak F_{\tau*}$ defined as $\tilde{\mathcal U}^t_{\tau,m} = U^t_{\tau,m} A U^{t*}_{\tau,m}$ and  $\tilde{\mathcal P}^t_{\tau,m} = U^{t*}_{\tau,m} \rho U^t_{\tau,m}$, together with the quantum observable
\begin{equation}
	\label{eq:discrete_a_fock}
	A_{f, \sigma, \tau, n, l} = \tilde G_{\tau, \sigma}^{\otimes n} \Delta_n M_{f,\tau,l} \Delta_n^* \tilde G_{\tau, \sigma}^{\otimes n} \in \mathfrak F_\tau
\end{equation}
acting on the Fock space (cf.\ \eqref{eq:a_fock}).
Our numerical approximation of the statistical evolution $f^{(t)}$ is given by
\begin{equation}
	\label{eq:ft_fock_d_num}
	f^{(t)}_{\sigma, \tau, n, d, m, l} = \frac{g^{(t)}_{\sigma, \tau, n, d, m, l}}{C_{\sigma, \tau, n, d, m, l}^{(t)}}, \quad g^{(t)}_{\sigma, \tau, n, d, m, l} = \mathbb E_{\tilde{\mathcal P}^t_{\tau,m}\nu_{\tau,d,m}}A_{f, \sigma, \tau, n, l}, \quad C^{(t)}_{\sigma, \tau, n, d, m, l} = \mathbb E_{\tilde{\mathcal P}^t_{\tau,m}\nu_{\tau,d,m}}A_{\bm 1, \sigma, \tau, n, l},
\end{equation}
where $\nu_{\tau,d,m} = \langle \eta_{\tau,d,m}, \cdot \rangle_{\mathcal H_\tau} \eta_{\tau,d,m} \in \mathfrak F_{\tau*}$ is the density operator associated with the state vector $\eta_{\tau,d,m} = \sum_{n=1}^\infty w_n \xi_{\tau, n,d,m}^{\otimes n} \in F(\mathcal Z_{\tau,d,m})$.
This approximation converges to $f^{(t)}_{\sigma, \tau, n, d}$ as $l,m \to \infty$.

Finally, if $f$ does not have a continuous representative, for a tolerance $\epsilon>0$, pick $\tilde f \in C(G)$ such that $\lVert f - \iota \tilde f \rVert_{L^1(\mu)} < \epsilon$.
Since $H_\tau$ is a subspace of $L^\infty(\mu)$, for every $g \in \mathcal H_\tau$ we have
\begin{align*}
	\lvert \langle K_\tau^* g, (f - \iota \tilde f) K_\tau^* g \rangle_H\rvert & \leq \lVert f - \iota \tilde f\rVert_{L^1(\mu)} \lVert K_\tau^* g \rVert_{L^\infty(\mu)}^2                                \\
	                                                                           & \leq \lVert f - \iota \tilde f\rVert_{L^1(\mu)} \lVert k_\tau\rVert_{C(G \times G)}^2 \lVert g \rVert_{\mathcal H_\tau}^2 \\
	                                                                           & < \epsilon \lVert k_\tau\rVert_{C(G \times G)}^2 \lVert g \rVert_{\mathcal H_\tau}^2.
\end{align*}
This means that for sufficiently small $\epsilon$ and large $l$, $\langle K^*_\tau (\tilde G_{\tau, \sigma} U^{t*}_\tau \xi_{\tau, n,d})^n, (\pi f)K^*_\tau (\tilde G_{\tau, \sigma} U^{t*}_\tau \xi_{\tau, n,d})^n\rangle_H$ can be approximated to any desired accuracy by $\langle K^*_{\tau,l} (\tilde G_{\tau, \sigma} U^{t*}_{\tau, m} \xi_{\tau, n,d,m})^n, (\pi_l f_l)K^*_{\tau,l} (\tilde G_{\tau, \sigma} U^{t*}_{\tau, m} \xi_{\tau, n,d,m})^n\rangle_{H_l}$, where $f_l$ is the element of $\mathfrak A_l$ represented by $\tilde f$.
Correspondingly, $f^{(t)}_{\sigma, \tau, n, d}$ can also be represented to any accuracy by $f^{(t)}_{\sigma, \tau, n, d, m, l}$ from~\eqref{eq:ft_fock_d_num} computed using samples of $\tilde f$.

\section{Approximation of pointwise evaluation}
\label{sec:pointwise}

Thus far, we have considered approximation of the evolution of expectation values of observables with respect to general probability densities $p \in \mathfrak A_*$ using quantum states on the Fock space $F(\mathcal H_\tau)$.
In this subsection, we specialize these constructions to quantum states that approximate pointwise evaluation of continuous functions on the $N$-torus, $G = \mathbb T^N$.
Following \cite{GiannakisEtAl22}, we build these states via a ``quantum feature map'', i.e., a map $\Xi_{\kappa, \tau}\colon G \to S_*(\mathfrak F_\tau)$, which we parameterize here by a sharpness parameter $\kappa > 0$.
For $\tilde f \in C(G)$, our goal is to show that the observable $\tilde f^{(t)}_{\kappa, \sigma, \tau, n} \in C(G)$, where $\tilde f^{(t)}_{\kappa, \sigma, \tau, n}(x)$ is obtained by applying~\eqref{eq:ft_fock} to the observable $f = \iota \tilde f \in \mathfrak A$ and the quantum state $\Xi_{\kappa, \tau}(x)$, well-approximates $(U^t \tilde f)(x) = \tilde f(\Phi^t(x))$, where $U^t\colon C(G) \to C(G)$ denotes the time-$t$ Koopman operator on continuous functions.

\subsection{Feature maps based on von Mises densities}
\label{sec:von_mises}

Our construction employs isotropic von Mises density functions on $\mathbb T^N$, which we build starting from the circle, $N=1$, using $p_{\mu, \kappa}$ from~\eqref{eq:von_mises_circle}.
In higher dimensions, we consider the product density $p_{\bm \mu, \bm \kappa} \in C^\infty(\mathbb T^N)$ with $\bm \kappa = (\kappa_1, \ldots, \kappa_N) \in \mathbb R^N_+$ and $\bm \mu = (\mu_1, \ldots, \mu_N) \in [0, 2\pi)^N$, given by
\begin{equation}
	\label{eq:von_mises}
	p_{\bm \mu, \bm \kappa}(x)= \prod_{i=1}^N p_{\mu_i, \kappa_i}(\theta_i), \quad x = (\theta_1, \ldots, \theta_N).
\end{equation}
By properties of Bessel functions, $p_{\bm \mu, \bm \kappa}$ integrates to 1 with respect to the Haar probability measure on $\mathbb T^N$.

The following are useful properties of von Mises density functions in the context of the approximation schemes studied in this paper.
\begin{itemize}[wide]
	\item $p_{\mu, \kappa}$ has known expansion coefficients in the character (Fourier) basis of $ \mathbb T$.
	      Specifically, for $j \in \mathbb Z \cong \widehat{\mathbb T}$, we have
	      \begin{equation*}
		      \hat p_{\mu,\bm \kappa}(j) := \langle \iota \gamma_j, \iota p_{\mu, \kappa}\rangle_{L^2(\mathbb T)} = \frac{I_{\lvert j\rvert}(\kappa)e^{ij\mu}}{I_0(\kappa)}.
	      \end{equation*}
	      Similarly, the Fourier coefficients of $p_{\bm \mu, \bm \kappa}$ on $\mathbb T^N$ are given by
	      \begin{equation}
		      \label{eq:von_mises_fourier}
		      \hat p_{\bm \mu, \bm \kappa}(j) := \langle \iota \gamma_{j}, \iota p_{\bm \kappa, \bm \mu}\rangle_{L^2(\mathbb T^N)} = \prod_{i=1}^N \hat\sigma_{\kappa_i,\mu_i}(j_i), \quad j = (j_1, \ldots, j_N).
	      \end{equation}
	\item The $n$-th root of a von Mises density function is a von Mises density function up to a proportionality constant,
	      \begin{equation*}
		      p_{\bm \mu, \bm \kappa}^{1/n}(x) = \prod_{i=1}^N\frac{I_0(\kappa_i/n)}{I_0^{1/n}(\kappa_i)} p_{\mu_i, \kappa_i/n}(\theta_i), \quad x = (\theta_1, \ldots, \theta_N).
	      \end{equation*}
	\item Recall the series representation \cite{AbramowitzStegun64}*{Chapter~9}
	      \begin{equation}
		      \label{eq:bessel_series}
		      I_j(\kappa) = (\kappa/2)^j \sum_{r=0}^\infty \frac{(\kappa/2)^{2r}}{r! (r+j)!}, \quad j \in \mathbb N_0.
	      \end{equation}
	      The above implies the upper bound
	      \begin{equation}
		      \label{eq:bessel_uppper}
		      I_j(\kappa/n) = \left(\frac{\kappa}{2n}\right)^j \sum_{r=0}^\infty \frac{(\frac{\kappa}{2n})^{2r}}{r! (r+j)!} \leq \frac{ \kappa^j}{(2n)^j j!} e^{\kappa^2/4n^2}.
	      \end{equation}
	      For the family of inverse weights~\eqref{eq:lambda_subexp} on $\mathbb T^N$, the rapid decay of $I_j(\kappa)$ as $\lvert j\rvert$ grows ensures that $p_{\bm\mu,\bm \kappa}$ lies in the RKHA $\mathcal H_\tau$ for every $\tau>0$.
\end{itemize}

Since $p_{\bm\mu, \bm \kappa} \in \mathcal H_\tau$, it follows that $F_{\kappa, \tau} \colon \mathbb T^N \to \mathcal H_\tau$ with $\kappa>0$,
\begin{equation*}
	F_{\kappa, \tau}(x) = p_{x,\bm \kappa} / \lVert p_{x,\bm \kappa}\rVert_{\mathcal H_\tau}, \quad \bm \kappa = (\kappa, \ldots, \kappa) \in \mathbb R^N_+,
\end{equation*}
is a well-defined feature map for every $\tau>0$.
We lift $F_{\kappa, \tau}$ to a quantum feature map $\Xi_{\kappa, \tau}\colon \mathbb T^N \to S_*(\mathfrak B_\tau)$, where $\mathfrak B_\tau = B(\mathcal H_\tau)$ and $\Xi_{\kappa, \tau}(x)$ is the rank-1 density operator that projects along $F_{\kappa, \tau}(x)$,
\begin{equation*}
	\Xi_{\kappa, \tau}(x) = \langle F_{\kappa, \tau}(x), \cdot \rangle_{\mathcal H_\tau} F_{\kappa, \tau}(x).
\end{equation*}
Given a probability vector $w \in \ell^1(\mathbb N)$ with strictly positive entries (as in \cref{sec:fock}), we also define variants $\tilde F_{\kappa, \tau}\colon \mathbb T^N \to F(\mathcal H_\tau)$ and $\tilde \Xi_{\kappa, \tau} \to S_*(\mathfrak F_\tau)$ of $F_{\kappa, \tau}$ and $\Xi_{\kappa, \tau}$, respectively, mapping into the Fock space,
\begin{displaymath}
	\tilde F_{\kappa, \tau}(x) = \eta_\tau, \quad \tilde\Xi_{\kappa, \tau}(x) = \langle \eta_\tau, \cdot \rangle_{F(\mathcal H_\tau)} \eta_\tau,
\end{displaymath}
where $\eta_\tau$ is obtained by applying~\eqref{eq:statevector} to the vector $\xi = F_{\kappa, \tau}(x) \in \mathcal H_\tau$.

The following lemma establishes that these feature maps can be used to approximate pointwise evaluation for continuous functions.

\begin{lem}
	\label{lem:pointwise}
	Assume the choice of inverse weights $\lambda_\tau$ in \eqref{eq:lambda_subexp}.
	Then for every $n \in \mathbb N$, $\tilde f \in C(\mathbb T^N)$, and $x \in X \equiv \supp(\mu)$, we have
	\begin{displaymath}
		\lim_{\kappa \to \infty} \frac{\mathbb E_{\Xi_{\kappa, \tau}(x)} M_{f, \tau}}{\mathbb E_{\Xi_{\kappa, \tau}(x)} M_{\bm 1, \tau}}
		= \lim_{\kappa \to \infty} \lim_{\sigma \to 0^+} \frac{\mathbb E_{\tilde \Xi_{\kappa, \tau}(x)} A_{f, \sigma, \tau, n}}{\mathbb E_{\tilde \Xi_{\kappa, \tau}(x)} A_{\bm 1, \sigma, \tau, n}} = \tilde f(x),
	\end{displaymath}
	where the quantum observables $M_{f, \tau}, M_{\bm 1, \tau} \in \mathfrak B_\tau$ and $A_{f, \sigma, \tau, n}, A_{\bm 1, \sigma, \tau, n} \in \mathfrak F_\tau$ are determined from~\eqref{eq:a_fock} for $f = \iota \tilde f \in \mathfrak A$.
\end{lem}

\begin{proof}
	By translation invariance of the kernel we may assume, without loss of generality, that the origin lies in $X$ and the evaluation point is $x=0$.

	Let $h_\kappa(\theta) = e^{\kappa \cos\theta}$ and observe that the feature vector $F_{\kappa, \tau}(0)$ is equal to $h_\kappa / \lVert h_\kappa\rVert_{\mathcal H_\tau}$.
	Moreover, we have $h_\kappa = \sqrt{h_{2\kappa}}$, so
	\begin{displaymath}
		F_{\kappa, \tau}(0) = \frac{\sqrt{\lVert \iota h_{2\kappa}\rVert_{L^1(\mu)}}}{\lVert h_\kappa\rVert_{\mathcal H_\tau}} \sqrt{\frac{h_{2\kappa}}{\lVert \iota h_{2\kappa}\rVert_{L^1(\mu)}}} = \frac{\lVert \iota h_\kappa\rVert_H}{\lVert h_\kappa\rVert_{\mathcal H_\tau}} \varrho_\kappa^{1/2},
	\end{displaymath}
	where $\varrho_\kappa = h_{2\kappa} / \lVert \iota h_{2\kappa}\rVert_{L^1(\mu)} \in \mathcal H_\tau$ is the representative of a probability density $p_\kappa = \iota \varrho_\kappa \in L^1(\mu)$.
	Thus, using~\eqref{eq:kvn}, we get
	\begin{align*}
		\mathbb E_{\Xi_{\kappa, \tau}(0)} M_{f,\tau}
		 & = \langle F_{\kappa, \tau}(0), K_\tau (\pi f) K_\tau^* F_{\kappa, \tau}(0)\rangle_{\mathcal H_\tau} = \langle K_\tau^* F_{\kappa, \tau}(0), (\pi f) K_\tau^* F_{\kappa, \tau}(0)\rangle_H                                         \\
		 & = \frac{\lVert\iota h_\kappa\rVert_H^2}{\lVert h_\kappa\rVert_{\mathcal H_\tau}^2} \mathbb E_{\Gamma(p_\kappa)}(\pi f) = \frac{\lVert\iota h_\kappa \rVert_H^2}{\lVert h_\kappa\rVert_{\mathcal H_\tau}^2} \mathbb E_{p_\kappa}f,
	\end{align*}
	and thus
	\begin{displaymath}
		\frac{\mathbb E_{\Xi_{\kappa, \tau}(0)} M_{f,\tau}}{\mathbb E_{\Xi_{\kappa, \tau}(0)} M_{\bm 1,\tau}} = \mathbb E_{p_\kappa}f.
	\end{displaymath}

	Next, since $0$ lies in the support of $\mu$, for every neighborhood $O$ of 0 and $\epsilon>0$ there exists $\kappa_* >0$ such that for all $\kappa \geq \kappa_*$, $\int_O \varrho_\kappa \, d\mu \geq 1 - \epsilon$.
	As a result, for every $g \in C(\mathbb T^N)$, we have $\lvert\int_{\mathbb T^N} g \varrho_\kappa \, d\mu - g(0) \rvert \leq \epsilon \lVert g\rVert_{C(G)} + (\int_O \varrho_\kappa \, d\mu) \sup_{x\in O} \lvert g(0) - g(x)\rvert$.
	Since $O$ and $\epsilon$ are arbitrary, it follows that $\varrho_\kappa \to \delta_0$ as $\kappa \to \infty$ in a distributional sense, so
	\begin{displaymath}
		\lim_{\kappa\to\infty} \frac{\mathbb E_{\Xi_{\kappa, \tau}(0)} M_{f,\tau}}{\mathbb E_{\Xi_{\kappa, \tau}(0)} M_{\bm 1,\tau}} = \lim_{\kappa\to\infty} \mathbb E_{p_\kappa} f = \tilde f(0),
	\end{displaymath}
	as claimed.
	We deduce the second claim of the lemma using~\eqref{eq:sigma_lim_t0},
	\begin{displaymath}
		\lim_{\kappa \to \infty} \lim_{\sigma \to 0^+} \frac{\mathbb E_{\tilde\Xi_{\kappa, \tau}(x)} A_{f, \sigma, \tau, n}}{\mathbb E_{\tilde\Xi_{\kappa, \tau}(x)} A_{\bm 1, \sigma, \tau, n}}
		= \lim_{\kappa \to \infty} \frac{\mathbb E_{\Xi_{\kappa, \tau}(x)} M_{f,\tau}}{\mathbb E_{\Xi_{\kappa, \tau}(x)} M_{\bm 1,\tau}} = \tilde f(0). \qedhere
	\end{displaymath}
\end{proof}

Combining \cref{thm:tau_conv} with \cref{lem:pointwise}, the same feature maps can also be used to approximate pointwise evaluation of time-evolved continuous observables by taking an additional $\tau \to 0^+$ limit.

\begin{cor}
	\label{cor:pointwise}
	With the assumptions and notation of \cref{thm:tau_conv,lem:pointwise}, define $\tilde f^{(t)}_{\kappa,\tau}, \tilde f^{(t)}_{\kappa, \sigma, \tau, n} \in C(\mathbb T^N)$, where
	\begin{displaymath}
		\tilde f^{(t)}_{\kappa, \tau} = \frac{\mathbb E_{\mathcal P^t_\tau(\Xi_{\kappa, \tau}(x))} M_{f, \tau}}{\mathbb E_{\mathcal P^t_\tau(\Xi_{\kappa, \tau}(x))} M_{\bm 1, \tau}}, \quad
		\tilde f^{(t)}_{\kappa, \sigma, \tau, n} = \frac{\mathbb E_{\tilde{\mathcal P}^t_\tau(\tilde\Xi_{\kappa, \tau}(x))} A_{f, \sigma, \tau, n}}{\mathbb E_{\tilde{\mathcal P}^t_\tau(\tilde\Xi_{\kappa, \tau}(x))} A_{\bm 1, \sigma, \tau, n}}.
	\end{displaymath}
	Then, the following hold for every $x \in X$:
	\begin{enumerate}
		\item $\lim_{\kappa\to\infty} \lim_{\tau\to 0^+} \tilde f^{(t)}_{\kappa, \tau}(x) = U^t \tilde f(x)$.
		\item For every $n \in \{2, 3, \ldots \}$, $\lim_{\kappa\to\infty} \lim_{\sigma \to 0{^+}}\lim_{\tau\to 0^+} \tilde f^{(t)}_{\kappa, \sigma, \tau, n}(x) = U^t \tilde f(x)$.
	\end{enumerate}
\end{cor}

\subsection{Numerical implementation}

In the context of the numerical approximations from \cref{sec:num_impl}, we modify the constructions in \cref{cor:pointwise} by replacing (i) $M_{f, \tau}$ and $A_{f, \sigma, \tau, n}$ by the finite-rank quantum observables $M_{f, \tau, l}$ and $A_{f, \sigma, \tau, n, l}$ obtained via~\eqref{eq:discrete_mult_op} and~\eqref{eq:discrete_a_fock}, respectively; and (ii) $F_{\kappa, \tau}$ by the projected feature map $F_{\kappa, \tau,d,m} \colon \mathbb T^N \to \mathcal H_{\tau,m}$, where $F_{\kappa, \tau,d,m}(x) = Z_{\tau,d,m} F_{\kappa, \tau}(x) / \lVert Z_{\tau,d,m} F_{\kappa, \tau}(x)\rVert_{\mathcal H_\tau}$.
As above, $F_{\kappa, \tau,d,m}$ lifts to a Fock-space-valued feature map $\tilde F_{\kappa, \tau,d,m}\colon \mathbb T^N \to F(\mathcal H_\tau)$, where
\begin{equation*}
	\tilde F_{\kappa, \tau,d,m}(x) = \sum_{n=1}^\infty w_n \xi_{\tau, n,d,m}^{\otimes n}.
\end{equation*}
Moreover, $F_{\kappa, \tau,d,m}$ and $\tilde F_{\kappa, \tau,d,m}$ induce quantum feature maps $\Xi_{\kappa, \tau,d,m}\colon \mathbb T^N \to S_*(\mathfrak B_\tau)$ and $\tilde\Xi_{\kappa, \tau,d,m}\colon \mathbb T^N \to S_*(\mathfrak F_\tau)$ analogously to $\Xi_{\kappa, \tau}$ and $\tilde\Xi_{\kappa, \tau}$, leading to approximations $\tilde f^{(t)}_{\kappa, \tau, d, m}$ and $\tilde f^{(t)}_{\kappa, \sigma, \tau, n, d, m}$ of $f^{(t)}_{\kappa, \tau}$ and $f^{(t)}_{\kappa, \sigma, \tau, n}$ from \cref{cor:pointwise}, respectively, given by
\begin{equation}
	\label{eq:pointwise_eval_d}
	\tilde f^{(t)}_{\kappa, \tau, d, m}(x) = \frac{\mathbb E_{\mathcal P^t_{\tau,m}(\Xi_{\kappa, \tau,d,m}(x))} M_{f, \tau}}{\mathbb E_{\mathcal P^t_{\tau,m}(\Xi_{\kappa, \tau,d,m}(x))} M_{\bm 1,\tau}},
	\quad \tilde f^{(t)}_{\kappa, \sigma, \tau, n, d, m}(x) = \frac{\mathbb E_{\tilde{\mathcal P}^t_{\tau,m}(\tilde\Xi_{\kappa, \tau,d,m}(x))} A_{f, \sigma, \tau, n}}{\mathbb E_{\tilde{\mathcal P}^t_{\tau,m}(\tilde\Xi_{\kappa, \tau,d,m}(x))} A_{\bm 1, \sigma, \tau, n}}.
\end{equation}
We further approximate $\tilde f^{(t)}_{\kappa, \sigma, \tau, d, m}$ and $\tilde f^{(t)}_{\kappa, \sigma, \tau, n, d, m}$ by $\tilde f^{(t)}_{\kappa,\sigma,\tau,d,m,l}$ and $\tilde f^{(t)}_{\kappa, \tau, n,d,m,l}$, computed from the quantum expectations of $M_{f,\tau,l}$ and $A_{f, \sigma, \tau, n, l}$, respectively, viz.
\begin{equation}
	\label{eq:pointwise_eval_d_l}
	\tilde f^{(t)}_{\kappa,\tau,d,m,l}(x) = \frac{\mathbb E_{\mathcal P^t_{\tau,m}(\Xi_{\kappa, \tau,d,m}(x))} M_{f, \tau,l}}{\mathbb E_{\mathcal P^t_{\tau,m}(\Xi_{\kappa, \tau,d,m}(x))} M_{\bm 1,\tau,l}},
	\quad \tilde f^{(t)}_{\kappa, \sigma, \tau, n,d,m,l}(x) = \frac{\mathbb E_{\tilde{\mathcal P}^t_{\tau,m}(\tilde\Xi_{\kappa, \tau,d,m}(x))} A_{f, \sigma, \tau, n, l}}{\mathbb E_{\tilde{\mathcal P}^t_{\tau,m}(\tilde\Xi_{\kappa, \tau,d,m}(x))} A_{\bm 1, \sigma, \tau, n, l}}.
\end{equation}
Note, in particular, that $\tilde f^{(t)}_{\kappa, \sigma, \tau, n,d,m,l}(x)$ is obtained via~\eqref{eq:ft_fock_d_num} for the quantum state $\tilde\Xi_{\kappa, \tau,d,m}(x)$.
The approximations $\tilde f^{(t)}_{\kappa, \tau, d, m, l}$ and $\tilde f^{(t)}_{\kappa,\sigma,\tau,d,m,l}$ converge pointwise on $X$ to $\tilde f^{(t)}_{\kappa,\tau}$ and $\tilde f^{(t)}_{\kappa, \sigma, \tau, n}$ in the iterated limits of $d\to\infty$ after $m,l \to \infty$ (see \cref{sec:num_impl}).

\subsection{Convergence analysis at fixed $d$}
\label{sec:convergence_torus}

As discussed in \cref{sec:finite_rank}, an interesting possibility of the tensor network framework is convergence as $n$ increases (and possibly other parameters vary), with the dimension parameter $d$ held fixed.
In this section, we study this possibility in the context of the scheme described in \cref{sec:pointwise} for pointwise approximation of bounded functions on $\mathbb T^N$ using classical and quantum feature maps built from von Mises distributions.
Specifically, the conjecture we wish to analyze is the pointwise convergence of $\tilde f^{(t)}_{\kappa, \sigma, \tau, n,d,m}$ from \eqref{eq:pointwise_eval_d} as $n, \frac{1}{\tau}, \frac{1}{\sigma} \to \infty$ jointly in an appropriately chosen manner and $\kappa$, $t$, $m$, and (importantly) $d$ held constant.

We focus on the simplest possible case of a circle rotation.
As it turns out, in this case we can set the smoothing parameter $\sigma = 0$ from the outset since the kernel integral operator $G_\sigma$ commutes with the Koopman operator $U^t$.
In \cref{sec:convergence_torus_proof}, we prove the following theorem.

\begin{thm}
	\label{thm:torus_rotation_convergence}
	Set $N=1$ and  $\Phi^t(\theta) = \theta + \alpha t \mod 2\pi$.
	Then, with the notation given above, for every $f \in L^\infty(\mathbb T)$ there are constants $c_1,c_2 \in \mathbb R_{>0}$ depending on $\kappa$, $d$, $\alpha$, $t$, and $p$ such that for every $\theta \in \mathbb T$, $\tilde f^{(t)}_{\kappa, 0, \tau, n,d,m}(\theta)$ from~\eqref{eq:pointwise_eval_d} converges as $\varepsilon \to 0^+$ to $ \mathbb E_{p_{\theta,2\kappa}} (U^t f)$ for $p_{\theta,2\kappa}$, and $(n,\tau) = (\lceil \frac{c_1}{\varepsilon}\rceil,c_2\varepsilon^2)$, and with $d=2$, $t$, $\kappa$, and $m \geq 2d + 1$ held constant.
\end{thm}

A similar argument to the $N=1$ case implies convergence for torus rotation dynamical systems of arbitrary dimension $N$ with $d$ at least $\lceil 1+\frac{4}{N} \rceil ^N$ and $n \sim \frac{1}{\varepsilon}$, $\tau \sim \varepsilon^2$.
We omit further discussion of this result in the interest of brevity.
In \cref{sec:convergence_general}, we will consider the challenges for analytically proving convergence for a generic smooth dynamical system.

\subsubsection{Proof of \cref{thm:torus_rotation_convergence}}
\label{sec:convergence_torus_proof}

Multiplicativity of the Koopman operator on $L^\infty(\mu)$ and the following lemma are the main components for proving convergence analytically for the circle rotation.
These two facts also require the estimate to be done with the infinity norm instead of the weaker two-norm which would otherwise be sufficient.

\begin{lem}
	\label{lem:pwr}
	Let $y, \tilde{y} \in \mathbb C$ and $n \in \mathbb N$.
	If $|y - \tilde{y}| \leq \frac{1}{n}$ then $|y^n - \tilde{y}^n| \leq e n\cdot \max\{1, \min\{|y|^n,|\tilde{y}|^n\}\} |y-\tilde{y}|$.
\end{lem}

\begin{proof}
	Without loss of generality $|y|^n \leq |\tilde{y}|^n$, and so
	\begin{align*}
		|y^n - \tilde{y}^n| & = \left\lvert \sum_{m=1}^n \binom{n}{m} y^{n-m} (\tilde{y} - y)^m \right\rvert \leq n |\tilde{y} - y| \sum_{m=1}^n \frac{(n-1)\cdots (n-m+1)}{m!} |y|^{n-m} \frac{1}{n^{(m-1)}} \\
		                    & \leq e n |y - \tilde{y}|\max\{1,|y|^n\} \sum_{m=1}^n \frac{(1-1/n)\cdots (1-(m-1)/n)}{m!}                                                                                       \\
		                    & \leq e n \cdot \max \{1, \min\{|y|^n, |\tilde{y}|^n\}\} |y - \tilde{y}|. \qedhere
	\end{align*}
\end{proof}

By translation invariance of the kernel $k_\tau$, the associated integral operator $\mathcal K_\tau\colon L^2(\mu) \to L^2(\mu)$ commutes with the Koopman operator $U^t$ of the circle rotation.
As result, for any of the approximation schemes in \cref{app:compact_approx,app:compact_res}, the regularized generator $V_\tau$ commutes with $\mathcal K_\tau$, and we may choose the Fourier functions $\{\phi_j\}_{j \in \mathbb Z}$ and the scaled Fourier functions $ \{ \psi_j = \sqrt{\lambda_\tau(j)} \varphi_j  \}_{j \in \mathbb Z}$ as eigenfunctions of $V_{\tau,m}$ and $W_{\tau,m}$, respectively, for any $\tau >0$ and $m \in \mathbb N$.
The Dirichlet energies of the Fourier functions are $\mathcal E_\tau(\phi_j) = 1/\lambda_\tau(j)$, giving the eigenfunctions $\zeta_{\tau,m,0}, \zeta_{\tau,m,1,j}, \zeta_{\tau,m,2}, \zeta_{\tau,m,3}, \zeta_{\tau,m,4},\ldots$ of $W_{\tau,m}$ as $\psi_{0,\tau}, \psi_{1,\tau}, \psi_{-1,\tau}, \psi_{2,\tau}, \psi_{-2,\tau},\ldots$, respectively, in accordance with our energy-based ordering from \cref{sec:spectral_approx}.
For simplicity of exposition, we will assume that $V_\tau$ is obtained via the generator compactification scheme from \cref{app:compact_approx}, so that the corresponding eigenfrequencies $\omega_{0,\tau,m}, \ldots, \omega_{4,\tau,m}$ are $0, \alpha \lambda_\tau(1), - \alpha\lambda_\tau(-1), 2\alpha \lambda_\tau(2), -2 \alpha \lambda_\tau(-2), \ldots$.
The analysis in the case of the resolvent compactification scheme from \cref{app:compact_res} proceeds in a similar manner after modification of the eigenfrequencies.
For any value of the dimension parameter $d \leq (m-1)/2$, the projection of $f = \sum_{j\in \mathbb Z} f_j \psi_j$ onto $\mathcal Z_{\tau,d,m} \subset \mathcal H_\tau$ is given by $Z_{\tau,d,m} f = \sum_{j=-d}^d f_j \psi_j$.
Moreover, for any $j_1, \ldots, j_n \in \mathbb Z$ with $\lvert j_i\rvert \leq d$, the lifted Koopman operator $\tilde U^t_\tau \colon F(\mathcal H_\tau) \to F(\mathcal H_\tau)$ on the Fock space satisfies
\begin{equation}
	\label{eq:eig_fock_2}
	\tilde U^t_{\tau,m} (\psi_{j_1} \otimes \cdots \otimes \psi_{j_n}) = e^{i\alpha t \sum_{r=-d}^d j_r \lambda_\tau(j_r)} \psi_{j_1} \otimes \cdots \otimes \psi_{j_n}.
\end{equation}

Next, set the evaluation point $x=0$ without loss of generality, define $p_\kappa := p_{0,\kappa}$, and note, using~\eqref{eq:classical_quantum_compat}, that
\begin{displaymath}
	\mathbb E_{p_{2\kappa}}(U^t f) = \mathbb E_{\Gamma(p_{2\kappa})}(\mathcal U^t f) = \langle U^{-t} p_{2\kappa}^{1/2}, M_f U^{-t} p_{2\kappa}^{1/2}\rangle_{L^2(\mathbb T)} = \langle \iota U^{-t} \tilde p_\kappa, M_f \iota U^{-t} \tilde p_\kappa\rangle_{L^2(\mathbb T)},
\end{displaymath}
where
\begin{displaymath}
	\tilde  p_\kappa = \frac{h_\kappa}{c_\kappa}, \quad h_\kappa(\theta) = e^{\kappa\cos \theta}, \quad c_\kappa = \lVert \iota h_\kappa\rVert_{L^2(\mathbb T)}.
\end{displaymath}
Moreover, we have $F_{\kappa, \tau}(0) = h_\kappa / \lVert h_\kappa\rVert_{\mathcal H_\tau}$ so that
\begin{displaymath}
	\tilde F_{\kappa, \tau,d,m}(0) = \sum_{n=1}^\infty w_n \tilde p_{\kappa, \tau, n,d}, \quad \tilde p_{\kappa, \tau, n,d} = \frac{1}{\lVert Z_{\tau,d,m} h_\kappa^{1/n}\rVert^n_{\mathcal H_\tau}} (Z_{\tau,d,m} (h_\kappa^{1/n}))^{\otimes n},
\end{displaymath}
giving
\begin{align*}
	\tilde f^{(t)}_{\kappa, 0, \tau, n,d,m}(0) & = \frac{\mathbb E_{\tilde\Xi_{\kappa, \tau,d,m}(0)} A_{f, 0, \tau, n}}{\mathbb E_{\tilde\Xi_{\kappa, \tau,d,m}(0)} A_{\bm 1, 0, \tau, n}} = \frac{\langle K_\tau^* \Delta_n^* \tilde U^{-t}_\tau \tilde F_{\kappa, \tau,d,m}(0), M_f K_\tau^*  \Delta_n^*\tilde U^{-t}_\tau \tilde F_{\kappa, \tau,d,m}(0)\rangle_{L^2(\mathbb T)}}{\langle K_\tau^*  \Delta_n^*\tilde U^{-t}_\tau \tilde F_{\kappa, \tau,d,m}(0),  K_\tau^*  \Delta_n^*\tilde U^{-t}_\tau \tilde F_{\kappa, \tau,d,m}(0)\rangle_{L^2(\mathbb T)}} \\
	                                           & = \frac{\langle K_\tau^*  \Delta_n^*\tilde U^{-t}_\tau \tilde p_{\kappa, \tau, n,d}, M_f K_\tau^*  \Delta_n^*\tilde U^{-t}_\tau  \tilde p_{\kappa, \tau, n,d}\rangle_{L^2(\mathbb T)}}{\langle K_\tau^*  \Delta_n^*\tilde U^{-t}_\tau \tilde p_{\kappa, \tau, n,d},  K_\tau^*  \Delta_n^*\tilde U^{-t}_\tau \tilde p_{\kappa, \tau, n,d}\rangle_{L^2(\mathbb T)}}.
\end{align*}

Our first goal is to estimate the difference between the $n^{th}$ root of $c_\kappa U^{-t}\tilde  p_\kappa$ and $c_\kappa^{1/n}U^{-t}_\tau Z_{\tau,d,m} (\tilde  p_\kappa^{1/n})$.
Due to multiplicativity of $U^{-t}$ on continuous functions, $\left(U^{-t}\tilde p_\kappa\right)^{1/n} = U^{-t}(\tilde p_\kappa^{1/n})$, and using the series representation of $I_j$ from~\eqref{eq:bessel_series}, we get
\begin{displaymath}
	c_\kappa^{1/n} (U^{-t}\tilde  p_\kappa)^{1/n}(\theta) = e^{\kappa \cos(\theta-\alpha t)/n} = \sum_{j=-\infty}^\infty I_{|j|}(\kappa/n)e^{ij(\theta-\alpha t)}.
\end{displaymath}
Similarly, using the eigenvalue relation $U^t_\tau \psi_j = e^{-i\alpha t j \lambda_\tau(j)} \psi_j$, we obtain
\begin{displaymath}
	c_\kappa^{1/n} U^{-t}_\tau Z_{\tau,d,m}(\tilde  p_\kappa^{1/n})(\theta) = \sum_{j=-d}^d I_{|j|}(\kappa/n)e^{ij(\theta-\alpha t)}.
\end{displaymath}
The series representation of $I_j$ also leads to the following useful inequality:
\begin{equation}
	\label{eq:bessel_sum}
	\left\lvert \sum_{j=d}^\infty I_{j}(\kappa) \right\rvert = \left\lvert \sum_{j=d}^\infty \sum_{r=0}^\infty \frac{(\kappa/2)^j (\kappa/2)^{2r}}{(r+j)! r!} \right\rvert \leq \left\lvert (\kappa/2)^d \sum_{j=d}^\infty \sum_{r=0}^\infty \frac{(\kappa/2)^{j-d} (\kappa/2)^{2r}}{(j-d)! (r+d)! r!} \right\rvert \leq e^{\kappa/2} I_d(\kappa).
\end{equation}

We may now estimate the difference of $n^{th}$ roots mentioned above.
For simplicity, we can assume that $n \geq \kappa$, however the following estimates do not rely on this being the case.
Using~\eqref{eq:bessel_uppper}, \eqref{eq:bessel_sum} and \cref{lem:pwr}, we obtain
\begin{multline*}
	\left\lvert c_\kappa^{1/n} (U^{-t}\tilde  p_\kappa)^{1/n}(\theta) -c_\kappa^{1/n} U^{-t}_\tau Z_{\tau,d,m}(\tilde  p_\kappa^{1/n})(\theta)\right\rvert\\
	\begin{aligned}
		 & =  \left\lvert \sum_{j=-\infty}^\infty I_{|j|}(\kappa/n) e^{-i\alpha j t} e^{ij \theta}  - \sum_{j=-d}^d I_{|j|}(\kappa/n) e^{i \alpha j \lambda_\tau(j) t} e^{i j \theta} \right\rvert                        \\
		 & \leq \left\lvert \sum_{j=-d}^d I_{|j|}(\kappa/n) e^{ij\theta} \left( e^{-i \alpha j t} - e^{-i \alpha j \lambda_\tau(j)t}\right) \right\rvert + 4 \left\lvert \sum_{j=d}^\infty I_{|j|}(\kappa/n) \right\rvert \\
		 & \leq 4 e^{\kappa/2n}I_d(\kappa/n) + \sum_{j=-d}^d I_{|j|}(\kappa/n) \left\lvert 1- e^{-i\alpha j t (1-\lambda_\tau(j))} \right\rvert                                                                           \\
		 & \leq 4 e^{\kappa/2n}I_d(\kappa/n) + 2 \max_{1 \leq j \leq d} \left\lvert \sin(\alpha j t (1-\lambda_\tau(j))/2)\right\rvert e^{\kappa/2n} I_0(\kappa/n)                                                        \\
		 & \leq 4 e^{\kappa/n} \left( I_d(\kappa/n) + \max_{1\leq j \leq d} |\sin(\alpha j t ( 1-\lambda_\tau(j))/2)|\right).
	\end{aligned}
\end{multline*}
Fix $\varepsilon >0$.
For $d \geq 2$ there exists $n \in \mathbb N$ such that
\begin{equation}
	\label{condition_on_n}
	I_d(\kappa/n) \leq \frac{ \kappa^d}{(2n)^d d!} e^{\kappa^2/4n^2} \leq \frac{\varepsilon c_\kappa}{n 8 e^{1+\kappa/n} \max\{1,c_\kappa e^\kappa\}}.
\end{equation}
Moreover, for fixed $\varepsilon >0$ and $\alpha, t \in \mathbb R$ there is a $\tau >0$ such that
\begin{equation}
	\label{condition_on_tau}
	\max_{1 \leq l \leq d} \left\lvert \sin(\alpha l t (1 - \lambda_\tau(l))/2)\right\rvert < \frac{\varepsilon c_\kappa}{8n e^{\kappa/n}\max\{1,c_\kappa e^\kappa\}}.
\end{equation}
Then for such an $n$ and $\tau$,
\begin{equation}
	\label{eq:sigma_root_diff}
	\left\lvert c_\kappa^{1/n} (U^{-t}\tilde  p_\kappa)^{1/n}(\theta) -c_\kappa^{1/n} U^{-t}_\tau Z_{\tau,d,m}(\tilde  p_\kappa^{1/n})(\theta)\right\rvert \leq \frac{\varepsilon c_\kappa}{n e \max\{1,c_\kappa e^\kappa\}}.
\end{equation}
Observe now that~\eqref{eq:eig_fock_2} implies
\begin{align*}
	\Delta_n^* \tilde{U}^{-t}_{\tau, n} \tilde p_{\kappa, \tau, n,d} (\theta) & = \left( \frac{1}{\lVert Z_{\tau,d,m}(h_\kappa^{1/n}) \rVert_{\mathcal H_\tau}} \sum_{j=-d}^d I_{|j|}(\kappa/n)e^{-i\alpha j \lambda_\tau(j) t} e^{ij\theta} \right)^n \\
	                                                                          & = \frac{1}{\lVert Z_{\tau,d,m}(h_\kappa^{1/n})\rVert_{\mathcal H_\tau} }U_\tau^{-t} Z_{\tau,d,m}(\tilde p_\kappa^{1/n})
\end{align*}
Using~\eqref{eq:sigma_root_diff}, it follows that
\begin{displaymath}
	\left \lVert  c_\kappa^{1/n} U^{-t}(\tilde p_\kappa^{1/n}) - c_\kappa^{1/n}\lVert Z_{\tau,d,m}(h_\kappa^{1/n})\rVert_{\mathcal H_\tau} \left(\Delta_n^* \tilde{U}^{-t}_{\tau, n}(\tilde p_{\kappa, \tau, n,d})\right)^{1/n} \right\rVert_{C(\mathbb T)} \leq \frac{\varepsilon c_\kappa}{n e\max\{1,c_\kappa e^\kappa\}}.
\end{displaymath}
Thus, for $\varepsilon \leq e\max\{1/c_\kappa,e^\kappa\}$ \cref{lem:pwr} can be applied,
\begin{multline*}
	\left \lVert c_\kappa U^{-t}\tilde p_\kappa - c_\kappa \lVert Z_{\tau,d,m}(h_\kappa^{1/n})\rVert^n_{\mathcal H_\tau} \Delta_n^* \tilde{U}_{\tau, n}^{-t}\tilde p_{\kappa, \tau, n,d} \right\rVert_{C(\mathbb T)} \\
	\begin{aligned}
		 & = \sup_{\theta \in \mathbb T} \left\lvert c_\kappa U^{-t}\tilde p_\kappa(\theta) - c_\kappa \lVert Z_{\tau,d,m}(h_\kappa^{1/n})\rVert^n_{\mathcal H_\tau}  \Delta_n^* \tilde{U}_{\tau, n}^{-t}\tilde p_{\kappa, \tau, n,d}(\theta) \right\rvert \\
		 & \leq \frac{\varepsilon e n c_\kappa \max\{1,\sup_{\theta \in \mathbb T} | c_\kappa \tilde p_\kappa(\theta)|\}}{n e\max\{1,c_\kappa e^\kappa\}}                                                                                                  \\
		 & = \frac{\varepsilon e n c_\kappa \max\{1,\sup_{\theta \in \mathbb T} | h_\kappa(\theta)|\}}{n e\max\{1,c_\kappa e^\kappa\}}                                                                                                                     \\
		 & \leq \varepsilon c_\kappa,
	\end{aligned}
\end{multline*}
giving
\begin{displaymath}
	\left \lVert U^{-t}\tilde p_\kappa - \lVert Z_{\tau,d,m}(h_\kappa^{1/n})\rVert^n_{\mathcal H_\tau} \Delta_n^* \tilde{U}_{\tau, n}^{-t}\tilde p_{\kappa, \tau, n,d} \right\rVert_{C(\mathbb T)} \leq \varepsilon.
\end{displaymath}
Taking $\eta = \iota U^{-t}\tilde  p_\kappa$ and $\tilde{\eta} = \lVert Z_{\tau,d,m}(h_\kappa^{1/n})\rVert^n_{\mathcal H_\tau}  K_\tau^* \Delta_n^* \tilde{U}_{\tau, n}^{-t}(\tilde p_{\kappa, \tau, n,d})$, we have $\lVert \eta - \tilde{\eta}\rVert_{L^\infty(\mathbb T)} \leq \varepsilon$ and
\begin{displaymath}
	\left\lvert \langle \eta, f \eta \rangle_{L^2(\mathbb T)} - \langle \tilde \eta, f \tilde \eta\rangle_{L^2(\mathbb T)}\right\rvert \leq 2 \lVert f\rVert_{L^\infty(\mathbb T)} \varepsilon,
\end{displaymath}
and so
\begin{align*}
	\left\lvert \mathbb E_{p_{2\kappa}}(U^t f) - \tilde f^{(t)}_{\kappa, 0, \tau, n,d,m}(0) \right\rvert
	 & = \left\lvert \langle \iota U^{-t} \tilde p_\kappa, f \iota U^{-t} \tilde p_\kappa\rangle_{L^2(\mathbb T)} -  \frac{\langle K_\tau^*\Delta_n^* \tilde U^{-t}_\tau \tilde p_{\kappa, \tau, n,d}, f K_\tau^* \Delta_n^*\tilde U^{-t}_\tau  \tilde p_{\kappa, \tau, n,d}\rangle_{L^2(\mathbb T)}}{\langle K_\tau^* \Delta_n^*\tilde U^{-t}_\tau \tilde p_{\kappa, \tau, n,d},  K_\tau^* \Delta_n^*\tilde U^{-t}_\tau \tilde p_{\kappa, \tau, n,d}\rangle_{L^2(\mathbb T)}}\right\rvert \\
	 & = \left\lvert \langle \iota U^{-t} \tilde p_\kappa, f \iota U^{-t} \tilde p_\kappa\rangle_{L^2(\mathbb T)} -  \frac{\langle \tilde\eta, f \tilde\eta\rangle_{L^2(\mathbb T)}}{\langle \tilde \eta, \tilde \eta\rangle_{L^2(\mathbb T)}}\right\rvert\\
	 & \leq 4 \lVert f \rVert_{L^\infty(\mathbb T)} \varepsilon
\end{align*}
since $\left\lvert \lVert \eta\rVert_{L^2(\mathbb T)} - \lVert \tilde \eta \rVert_{L^2(\mathbb T)} \right\rvert =\lvert 1 -\lVert\tilde{\eta}\rVert_{L^2(\mathbb T)} \rvert \leq 2 \varepsilon$.

Finally, we analyze the rate of convergence in terms of $n$ and $\tau$.
First, the requirement for $n$ in~\eqref{condition_on_n} implies
\begin{displaymath}
	n^{d-1} \geq \frac{4 \kappa^d e^{\kappa/n +\kappa^2/4n^2 +1}\max\{1,c_\kappa e^\kappa\}}{d! \varepsilon c_\kappa 2^{d-1}}.
\end{displaymath}
Defining $C_1(\kappa, d) =\left(4e^{\kappa + \kappa^2/4 +1} \max\left\{\frac{1}{c_\kappa}, e^\kappa \right\} \right)^{1/(d-1)}$, we have
\begin{equation*}
	C_1(\kappa,d) \geq \left(4e^{\kappa/n + \kappa^2/4n^2 +1} \max\left\{\frac{1}{c_\kappa}, e^\kappa \right\} \right)^{1/(d-1)}
\end{equation*}
for all $n \in \mathbb N$, and thus \eqref{condition_on_n} is satisfied if
\begin{equation*}
	n \geq \left(\frac{\kappa^d}{d!\varepsilon} \right)^{1/(d-1)} C_1(\kappa,d).
\end{equation*}
Second, since $\sin a < a$ for $a \geq 0$ and $\lambda_\tau(l)$ is a non-increasing function of $l$, \eqref{condition_on_tau} holds if
\begin{equation*}
	\lvert \alpha t d (1-\lambda_\tau(d)) \rvert < \frac{\varepsilon c_\kappa}{4n e^{\kappa/n} \max\{1, c_\kappa \}}.
\end{equation*}
Moreover, since $1-e^{-\tau d^p} \leq \tau d^p$ for $\tau \geq 0$, it is sufficient that
\begin{displaymath}
	\tau \leq \frac{\varepsilon}{|\alpha t| d^{1+p} n} C_2(\kappa) := \frac{\varepsilon}{|\alpha t| d^{1+p} n}  \frac{\min\left\{c_\kappa, e^{-\kappa}\right\}}{4e^{1+\kappa}}  \leq \frac{\varepsilon \min\left\{c_\kappa, e^{-\kappa}\right\}}{4 |\alpha t| d^{1+p} n e^{1+\kappa/n}}.
\end{displaymath}

With these inequalities we conclude how the joint $(n,\tau)$ limit must be taken.
Once $\varepsilon \leq e\max\{1/c_\kappa,e^\kappa\}$ and choosing
\begin{displaymath}
	n = \left\lceil C_1(\kappa,d) \left( \frac{\kappa^d}{d! \varepsilon}\right)^{1/(d-1)} \right\rceil, \quad \text{and} \quad \tau = C_2(\kappa) \frac{\varepsilon}{\lvert \alpha t \rvert d^{1+p} n},
\end{displaymath}
conditions \eqref{condition_on_n} for $n$ and \eqref{condition_on_tau} for $\tau$ are satisfied for any $d \geq 2$. \Cref{thm:torus_rotation_convergence} then follows for $d=2$ with a joint limit $n \sim \frac{1}{\varepsilon}$ and $\tau \sim \varepsilon^2$. \qed

\subsection{Towards a general convergence analysis}
\label{sec:convergence_general}

Let $(X, \mu,\Phi^t)$ be a measure-preserving, ergodic dynamical system as in \cref{sec:dyn_syst}.
Using the $L^\infty(\mu)$ norm for the main estimate in \cref{sec:convergence_torus_proof} is likely much stronger and more restrictive than is necessary for a general such system.
This choice is due to the ability to convert an estimate for $\left\lVert y^{1/n} -  \tilde{y}^{1/n} \right\rVert_{L^\infty(\mu)}$ into an estimate for $\left\lVert y -  \tilde{y} \right\rVert_{L^\infty(\mu)}$.
One encouraging feature is the convergence of $p^{1/n} \to 1$ for a positive element $p \in L^\infty(\mu)$ and the fact that $U^t \bm 1 = \bm 1 = K_\tau^*U^t_\tau \bm 1$.
As $n$ grows, $\bm 1$ becomes the main component of $p^{1/n}$ in $L^2$-norm and is preserved by the approximate Koopman operator.

In analogy with Sobolev-type estimates of $L^\infty$ norms, we may be able to obtain $L^\infty(\mu)$ estimates from $L^2(\mu)$ estimates such as
\begin{displaymath}
	\left\lVert U^t(p^{1/n}) - K_\tau^* \tilde{U}^{-t}_{\tau,1} Z_{\tau,d,m}(p^{1/n}) \right \rVert_{L^2(\mu)} \leq \left\lVert (U^tK_\tau^* - K_\tau^* U^{-t}_{\tau,1}) p_\perp^{1/n} \right\rVert_{L^2(\mu)} + \left\lVert p_\perp^{1/n} - Z_{\tau,d,m}(p_\perp^{1/n})\right\rVert_{L^2(\mu)}
\end{displaymath}
and of its derivatives, where $p_\perp^{1/n} = p^{1/n} - \langle \bm 1, p^{1/n}\rangle_{L^2(\mu)} \bm 1$.
This separates the estimate into two components, convergence of the smoothed Koopman operator, and the efficiency of approximating $p_\perp^{1/n}$ with $2d+1$ eigenfunctions of $U^t_\tau$.
In general, we expect a joint limit in $n$, $d$, and $\tau$ to be necessary.
Further analytical results from this approach are unlikely without specializing to specific families of dynamical systems.
We now turn to numerical experiments to verify the torus rotation case and explore more nontrivial dynamics.

\section{Numerical experiments}
\label{sec:experiments}

We present applications of the tensor network approximation framework described in the preceding sections to two dynamical systems on the 2-torus: an ergodic rotation (\cref{sec:torusrot}) and a Stepanoff flow \cite{Oxtoby53} (\cref{sec:stepanoff}).
These examples were chosen on the basis of exhibiting qualitatively different types of spectral characteristics, the former being a prototypical example of a pure-point-spectrum system while the latter is characterized by absence of non-constant continuous Koopman eigenfunctions.
Yet, despite their different spectral characteristics, both systems have the Haar probability measure on $\mathbb T^2$ as an ergodic invariant measure $\mu$, meaning that in both cases $G= X = \mathbb T^2$ and our method can be implemented using the characters (Fourier functions) of the group as basis functions.
This in turn allows assessing the behavior of our approach independently of errors due to data-driven computation of basis functions.

As in \cref{sec:rkha_examples}, for $j = (j_1, j_2) \in \mathbb Z^2$, we let $\gamma_j \in \widehat G$ be the Fourier function with wavenumbers $j_1$ and $j_2$ along the two torus dimensions, $\gamma_j(x) = e^{i j \cdot x}$ where $x = (\theta_1, \theta_2) \in [0, 2\pi)^2$.
We will sometimes use overarrows to distinguish multi-indices from integer-valued indices, i.e., $\vec\jmath \equiv j = (j_1, j_2)$.
In what follows, $\bm A^\dag$ will denote the complex-conjugate transpose of a matrix $\bm A$.
Technical details on numerical implementation and pseudocode are included in \ref{app:numerical}.

\subsection{Dynamical systems}
\label{sec:dyn_syst_examples}

\subsubsection{Torus rotation}
Our first example is an ergodic torus rotation generated by the smooth vector  field $\vec V \colon G \to TG$, where
\begin{equation}
	\label{eq:vec_torusrot}
	\vec V = \frac{\partial\ }{\partial \theta_1} + \alpha \frac{\partial\ }{\partial \theta_2}.
\end{equation}
Here, $\alpha$ is an irrational parameter which we set to $\sqrt{30} \approx 5.477$.
A quiver plot of the vector field for our choice of $\alpha$ is displayed in \cref{fig:vectorfield}(left).

The resulting flow is given by
\begin{equation}
	\label{eq:phi_rot}
	\Phi(x) = x + \alpha t \mod 2\pi,
\end{equation}
and has the Haar measure $\mu$ as its unique Borel ergodic invariant probability measure.
Moreover, the generator $V$ on $H$ has pure point spectrum with eigenfrequencies $\omega_{\vec\jmath} = j_1 + j_2 \alpha$ and corresponding eigenfunctions $u_{\vec\jmath} = \iota \gamma_{\vec \jmath} \in \mathfrak A \subset H$ with continuous representatives given by the characters $\gamma_{\vec \jmath} \in \widehat G$.
In this system, any pair of eigenfunctions $\{u_{\vec\jmath_1}, u_{\vec\jmath_2},\}$ with rationally independent eigenfrequencies $\omega_{\vec\jmath_1}$ and $\omega_{\vec\jmath_2}$ provides a conjugacy with a rotation system in the form of~\eqref{eq:rot_semiconj}.
The pair $ \{ \omega_{\vec\jmath_1}, \omega_{\vec\jmath_2} \}$ generates the point spectrum  of $V$ as an additive group, $\sigma_p(V) =  \{i (  r_1 \omega_{\vec\jmath_1} + r_2 \omega_{\vec\jmath_2} )\}_{r_1, r_2 \in \mathbb Z}$, and pointwise products of the corresponding eigenfunctions generate an orthonormal basis of $H$ as a multiplicative group, $\{ u_{\vec\jmath_1}^{r_1} u_{\vec\jmath_2}^{r_2} \}_{r_1,r_2 \in \mathbb Z}$.

It should be noted that despite the highly structured nature of this system, numerically approximating its point spectrum and associated eigenfunctions is non-trivial, for $\sigma_p(V)$ is a dense subset of the imaginary line and thus there are no isolated points in the spectrum (cf.\ the circle rotation example in \cref{sec:circlerot}, where all eigenvalues of $V$ are isolated).

\begin{figure}
	\includegraphics[draft=false, width=0.48\linewidth]{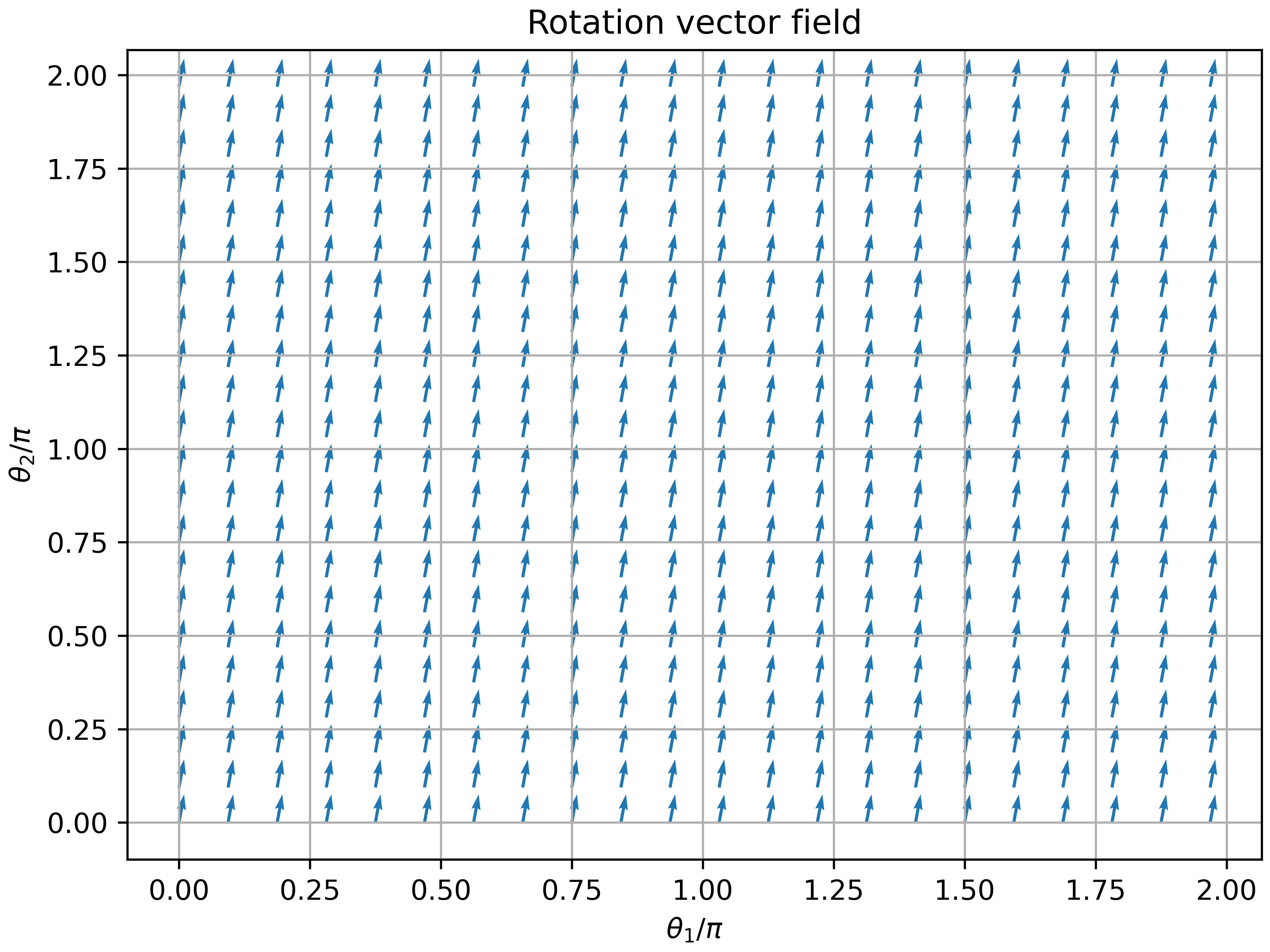}
	\hfill
	\includegraphics[draft=false, width=0.48\linewidth]{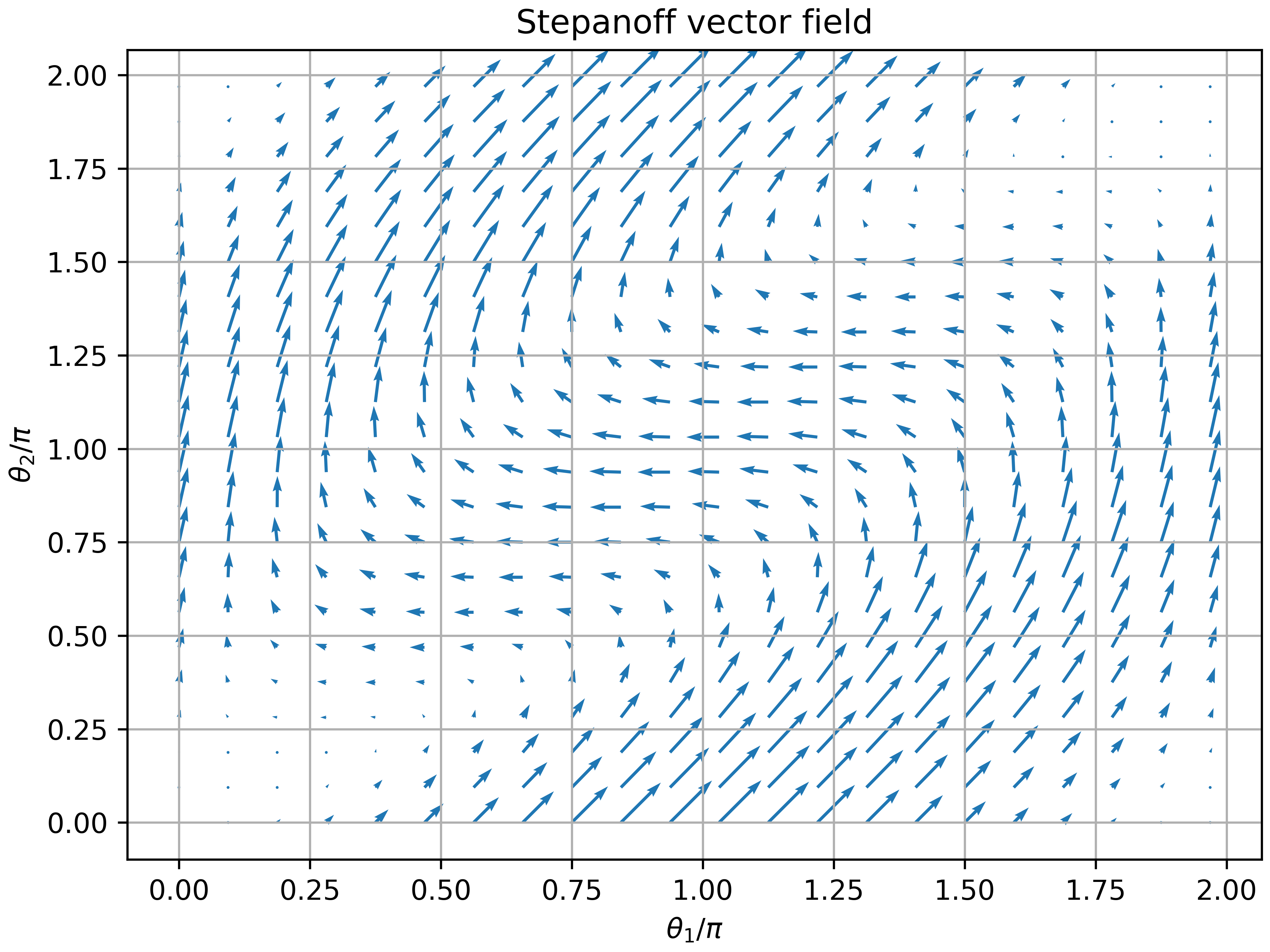}
	\caption{Quiver plots of the dynamical vector fields $\vec V$ for the ergodic torus rotation (left) and Stepanoff flow (right) examples.}
	\label{fig:vectorfield}
\end{figure}

\subsubsection*{Stepanoff flow}
Our second example comes from a class of Stepanoff flows on the 2-torus studied in a paper of Oxtoby \cite{Oxtoby53}.
The dynamical vector field $\vec V\colon G \to TG$ is defined in terms of smooth functions $V_1,V_2 \in C^\infty(\mathbb T^2)$ as
\begin{equation}
	\label{eq:vec_stepanoff}
	\begin{gathered}
		\vec V = V_1 \frac{\partial\ }{\partial \theta_1} + V_2 \frac{\partial\ }{\partial \theta_2}, \\
		V_1(x) = V_2(x) + (1- \alpha) (1 - \cos \theta_2), \quad V_2(x) = \alpha(1 - \cos(\theta_1 - \theta_2)),
	\end{gathered}
\end{equation}
where $\alpha$ is a real parameter.
One readily verifies that $\vec V$ has zero divergence with respect to the Haar measure $\mu$,
\begin{displaymath}
	\divr_\mu \vec V = \frac{\partial V_1}{\partial \theta_1} + \frac{\partial V_2}{\partial \theta_2} = 0,
\end{displaymath}
which implies that $\mu$ is an invariant measure under the associated flow $\Phi^t$.
Moreover, the system exhibits a fixed point at $x=0$.
In \cite{Oxtoby53} it is shown that the Haar measure is the unique invariant Borel probability measure of this flow that assigns measure 0 to the singleton set $ \{ 0 \} \subset G$ containing the fixed point.
In our experiments, we set $\alpha = \sqrt{20}$ leading to the vector field depicted in \cref{fig:vectorfield}(right).
Notice the characteristic ``S''-shape formed by the arrows as well as the presence of the fixed point at $x=0$.

Since any continuous, non-constant Koopman eigenfunction induces a semi-conjugacy with circle rotation (i.e., a version of~\eqref{eq:rot_semiconj} with $N=1$), the existence of the fixed point at $x=0$ implies that the system has no continuous Koopman eigenfunctions; i.e., it is topologically weak-mixing.
To our knowledge, there are no results in the literature on the measure-theoretic mixing properties of Stepanoff flows.
While topological mixing is, in general, independent of measure-theoretic mixing, the absence of continuous Koopman eigenfunctions implies at the very least that data-driven spectral computations are non-trivial for this class of systems (even if eigenfunctions exist in $L^2(\mu)$).

\subsection{Spectral decomposition}
\label{sec:spec_decomp}

In each example, we work with the RKHA $\mathcal H_\tau$ on $G$ induced from the family of weights $\lambda_\tau$ in~\eqref{eq:lambda_subexp}.
Fixing a maximal wavenumber parameter $J \in \mathbb N$ and setting $m = (2 J+ 1)^2 - 1$, we order the characters $\gamma_{\vec\jmath_0}, \ldots, \gamma_{\vec\jmath_m}$ in the set $ \widehat\Gamma_J = \{ \gamma_{\vec\jmath} \in \widehat G: \gamma_{\vec\jmath}(x) = e^{i \vec\jmath \cdot x}\colon \vec\jmath = (j_1, j_2): \lvert j_i \rvert \leq J \} \subset \widehat G$ in order of increasing $\lvert j_1 \rvert + \lvert j_2\rvert$.
This leads to the $m$-dimensional subspace $\mathcal H_{\tau,m} = \spn \widehat\Gamma_J \subset \mathcal H_\tau$ on which we define the approximate generator $W_{\tau,m}$ for each system.
For consistency with the notation for kernel eigenfunctions in \eqref{eq:k_eig}, we denote the ordered characters in $\widehat\Gamma_J$ as $\varphi_i \equiv \gamma_{\vec\jmath_i}$ and their $L^2$ equivalence classes as $\phi_i = \iota \varphi_i$.
We also let $\Lambda_{i,\tau} \equiv \lambda_\tau(\vec\jmath_i)$ be the corresponding eigenvalues with associated orthonormal basis vectors $\psi_{i, \tau} = \Lambda_{i,\tau}^{1/2}$ of $\mathcal H_\tau$.

In the experiments reported here, we set the RKHA parameters $p=0.75$, $\tau= 0.001$, and the maximal wavenumber parameter $J = 128$ corresponding to approximation space dimension $m = \text{16,641}$.  We build $W_{\tau,m}$ using \cref{alg:eigs}, described in \cref{app:compact_res}, with resolvent parameter $z=0.1$.
The numerical procedure solves an $m \times m$ matrix generalized eigenvalue problem~\eqref{eq:gev_mat} that approximates the eigendecomposition of an auxiliary skew-adjoint compact operator, $Q_{z,\tau} \in B(H)$.
The eigenvalues, $\beta_{j,z,\tau,m} \in i \mathbb R$, are then converted into eigenfrequencies, $\omega_{j,z,\tau,m}$, of $W_{\tau,m}$, and the corresponding generalized eigenvectors, $\bm {\tilde c}_j \in \mathbb C^m$, are mapped into coefficient vectors $\bm c_j = (c_{0j}, \ldots, c_{mj})^\top \in \mathbb C^{m+1}$, giving the expansions $\zeta_{j,z,\tau,m} = \sum_{i=0}^m c_{ij} \psi_{i, \tau}$ of $W_{z,\tau,m}$; see~\eqref{eq:vtilde_eig} and~\eqref{eq:zeta_expansion}.

In our examples, the components of the dynamical vector field $\vec V$ are polynomials in the characters $\gamma_j$, which allows us to compute the elements of the various operator matrices appearing in \cref{alg:eigs} analytically.
In the case of the rotation vector field \eqref{eq:vec_torusrot}, the characters are eigenfunctions of the generator and we immediately obtain
\begin{equation}
	\label{eq:gen_rot}
	\langle \varphi_r, \iota \vec V \cdot \nabla \varphi_s \rangle_H = i (j_1 + j_2 \alpha ) \delta_{rs},
\end{equation}
where $ (j_1, j_2) = \vec\jmath_s$.
In the case of the Stepanoff vector field \eqref{eq:vec_stepanoff}, we make use of the relations
\begin{displaymath}
	\cos \theta_2 = \frac{\gamma_{0,1}(x) + \gamma_{0,-1}(x)}{2}, \quad \cos(\theta_1 - \theta_2) = \frac{\gamma_{(1,-1)}(x) + \gamma_{(-1, 1)}(x)}{2}, \quad x = (\theta_1, \theta_2),
\end{displaymath}
together with the group structure of $\widehat G$, $\gamma_{\vec\jmath_r + \vec\jmath_s} = \gamma_{\vec\jmath_r} \gamma_{\vec\jmath_s}$, to compute
\begin{multline}
	\label{eq:gen_stepanoff}
	\langle \varphi_r, \iota \vec V \cdot \nabla \varphi_s \rangle_H \\
	\begin{aligned}
		 & = i j_1 \delta_{rs} - \frac{ij_1\alpha}{2}(\delta_{i_1,1+j_1}\delta_{i_2,j_2-1} + \delta_{i_1, j_1 -1}\delta_{i_2,j_2+1}) - \frac{ij_1(1-\alpha)}{2}(\delta_{i_1,j_1}\delta_{i_2,j_2+1} + \delta_{i_1,j_1}\delta_{i_2,j_2-1}) \\
		 & \quad + i j_2 \delta_{rs} - \frac{ij_2\alpha}{2}(\delta_{i_1,j_1+1}\delta_{i_2,j_2-1} + \delta_{i_1,j_1-1}\delta_{i_2,j_2+1}),
	\end{aligned}
\end{multline}
where $(i_1,i_2) = \vec\jmath_r$ and $(j_1,j_2)=\vec\jmath_s$.
The above formulas are sufficient to execute \cref{alg:eigs} for the torus rotation and Stepanoff flow.

We solve~\eqref{eq:gev_mat} using iterative methods, computing $n_\text{eig} = 1024$ eigenvalues $\beta_{j,z,\tau,m}$ of largest magnitude and the corresponding generalized eigenvectors $\bm{\tilde c}_{j,z,\tau,m}$ for each system.
Note that our choice of computing largest-magnitude eigenvalues $\beta_{j,z,\tau,m}$ is equivalent to computing lowest-magnitude nonzero eigenfrequencies $\omega_{j,z,\tau,m}$.
We order the computed eigenvalue/eigenfunction pairs for the two systems in order of increasing Dirichlet energy, computed by applying~\eqref{eq:dirichlet_eig} to the coefficient vectors $\bm{c}_{j,z,\tau,m}$.
Further details on numerical implementation can be found in \cref{app:numerical}.
For the remainder of this section we will sometimes use the abbreviated notation $\omega_j \equiv \tilde\omega_{j,z,\tau,m}$ and $\zeta_j \equiv \tilde\zeta_{j,z,\tau,m}$.

\Cref{fig:generator_spec} shows scatterplots of the $n_\text{eig}$ computed eigenfrequencies for the torus rotation (left) and Stepanoff flow (right) plotted against their corresponding Dirichlet energies.
Perhaps unsurprisingly, the generator spectrum for the torus rotation exhibits regular structure, with the eigenfrequencies visually forming spectral ``lines'' consistent with well-defined functional relationships between frequency and Dirichlet energy.
On the other hand, the Stepanoff spectrum appears markedly less structured, which could be the outcome of discretization of a continuum.

\begin{figure}
	\centering
	\includegraphics[draft=false, width=0.48\linewidth]{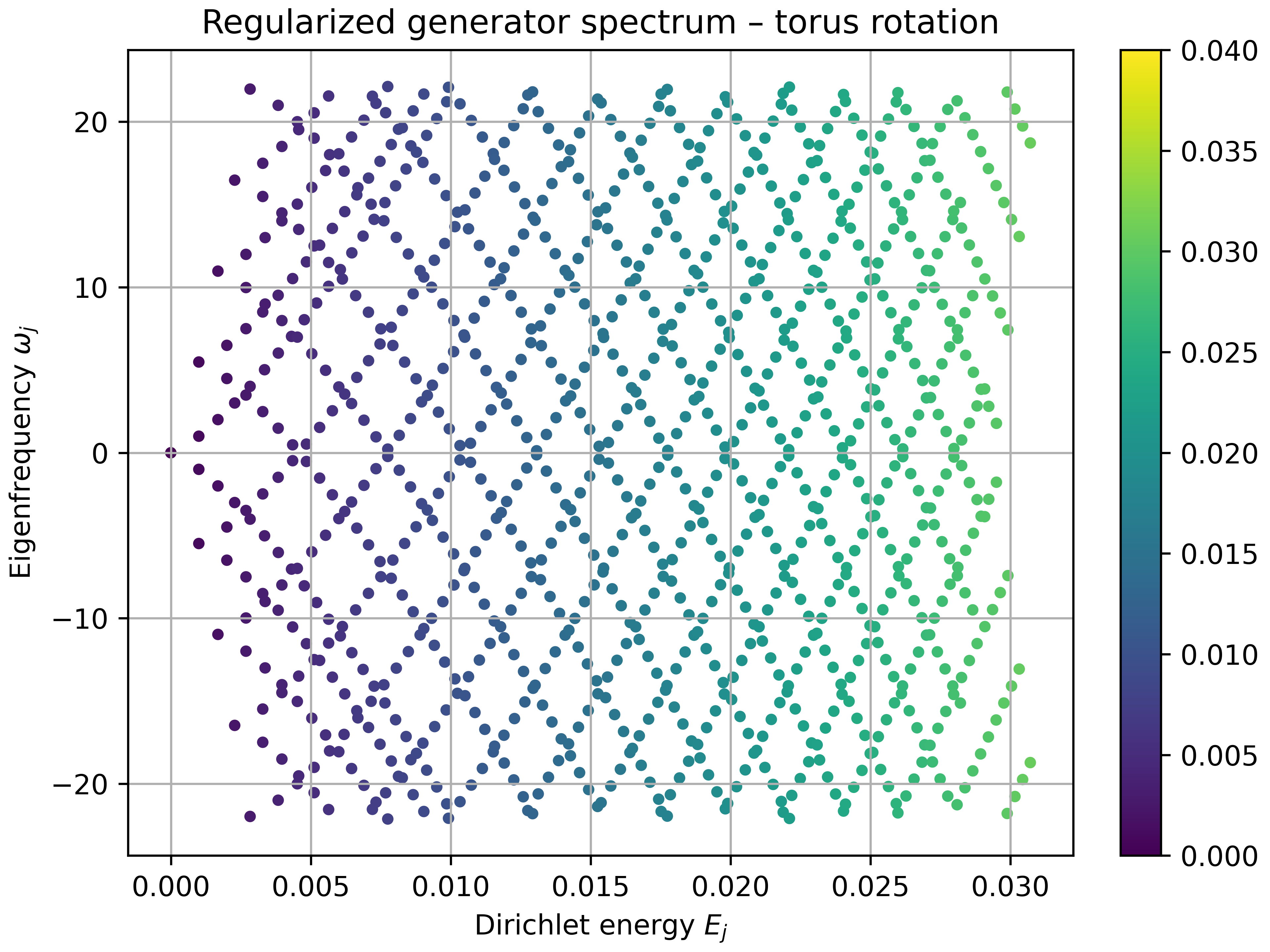}
	\hfill
	\includegraphics[draft=false, width=0.48\linewidth]{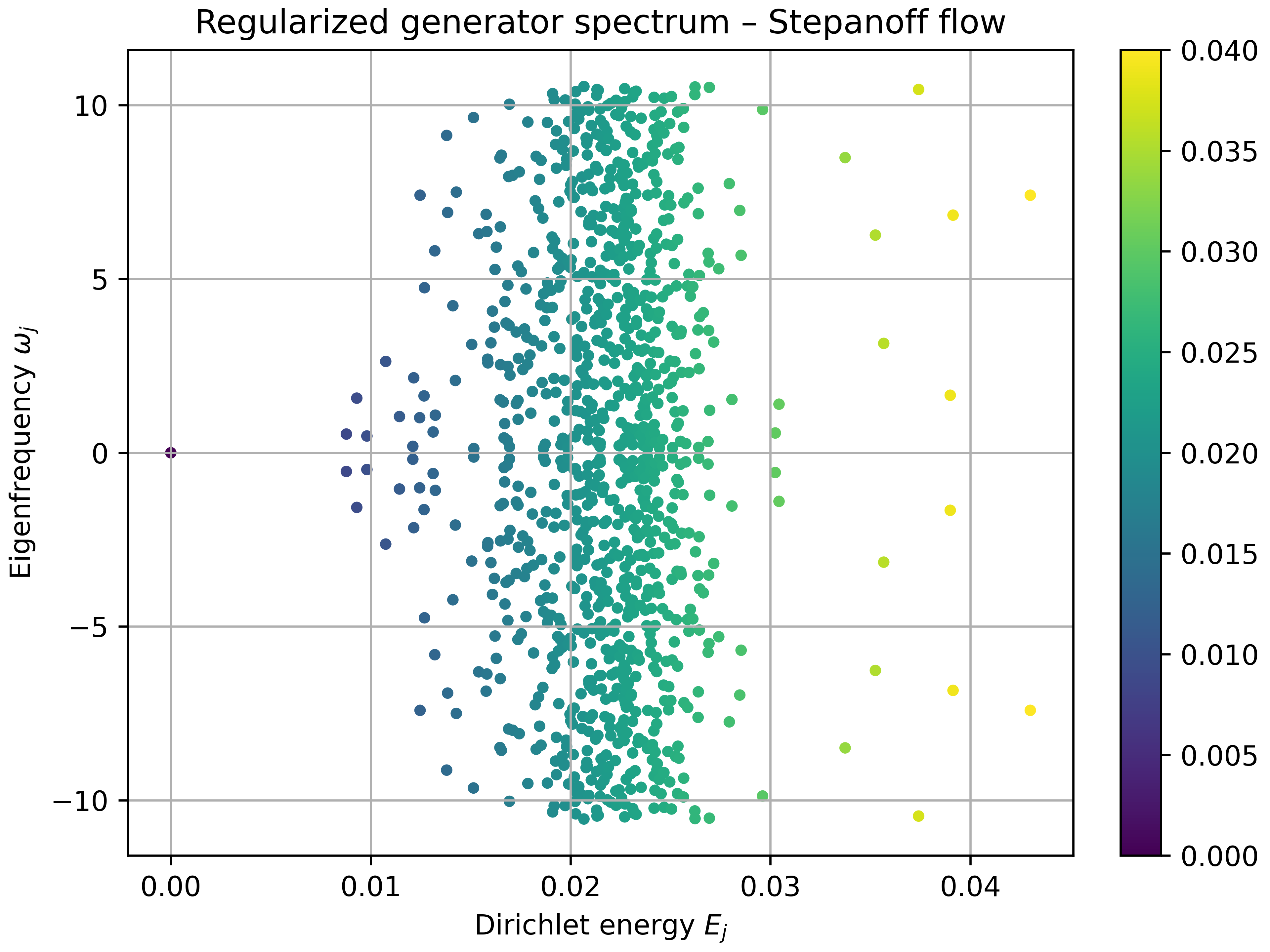}
	\caption{Spectrum of the regularized generator $W_{\tau,m}$ for the torus rotation (left) and Stepanoff flow (right).
		The eigenfrequencies $\omega_j$ are plotted versus the corresponding Dirichlet energies $E_j$ from the RKHA $\mathcal H_\tau$.
		The plotted points are colored by Dirichlet energy to facilitate comparison between the two cases.}
	\label{fig:generator_spec}
\end{figure}

In \cref{fig:generator_eig}, we plot the real parts of representative eigenfunctions $\zeta_j$ from the spectra in \cref{fig:generator_spec} on a regular $512 \times 512$ grid on $\mathbb T^2$.
In the case of the torus rotation (left-hand column), the plots well-approximate the planar wave patterns expected for real parts of Fourier functions.
In particular, eigenfunctions $\zeta_2$ (top left) and $\zeta_4$ (center left) exhibit wavenumber-1 oscillations along the $\theta_1$ and $\theta_2$ directions, which is consistent with the corresponding eigenfrequencies $\omega_2 \approx 1.00$ and $\omega_4 \approx 5.48$, respectively.
Meanwhile, $\zeta_{12}$ (bottom left) exhibits wavenumber-1 oscillations along both the $\theta_1$ and $\theta_2$ directions and a corresponding eigenfrequency $\omega_{12} \approx \omega_4 - \omega_2$, consistent with a product structure $\zeta_{12} \approx \zeta_2^* \zeta_4$.

\begin{figure}
	\centering
	\includegraphics[draft=false, width=0.48\linewidth]{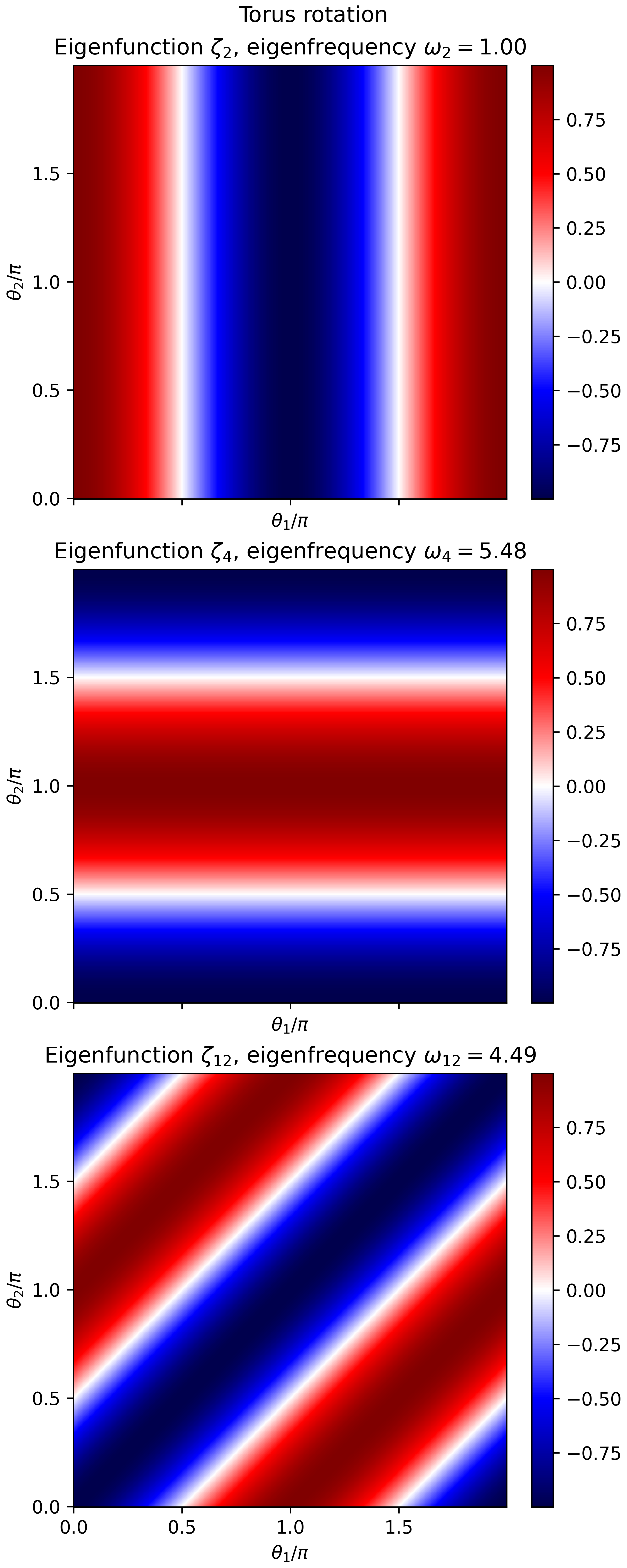}
	\hfill
	\includegraphics[draft=false, width=0.48\linewidth]{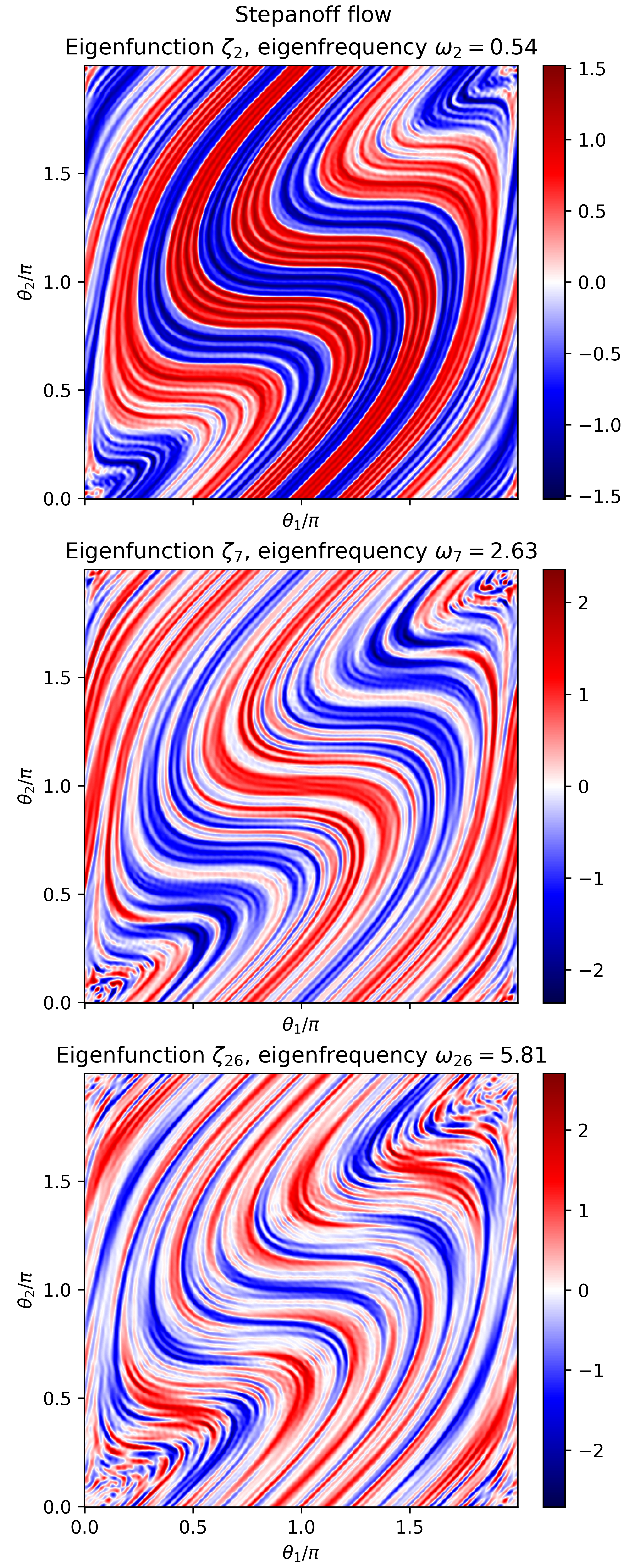}
	\caption{Real parts of representative eigenfunctions $\zeta_j$ of the regularized generator for the torus rotation (left) and Stepanoff flow (right).}
	\label{fig:generator_eig}
\end{figure}

Turning to the Stepanoff flow, the eigenfunctions in this case (right-hand column in \cref{fig:generator_eig}) exhibit considerably more intricate spatial structure than their rotation counterparts, in line with the fact that the unperturbed generator $V$ does not have non-constant continuous eigenfunctions.
The plotted eigenfunctions are the second, seventh, and 26th with respect to the Dirichlet energy ordering; i.e., are among the least oscillatory examples among the 1024 computed eigenfunctions.
In regions away from the fixed point at $x=0$, the level sets of the plotted eigenfunctions appear to be broadly aligned with the dynamical flow; compare, e.g., $\zeta_2$ in the top-right panel of \cref{fig:generator_eig} with the quiver plot of the Stepanoff vector field in \cref{fig:vectorfield}.
In regions closer to the fixed point, the eigenfunctions develop small-scale oscillations, whose presence is qualitatively consistent with the slowing down of the dynamical flow near the fixed point.
Note, in particular, that an orbit passing from a point $x $near the fixed point has to cross more eigenfunction wavefronts in a given time interval to produce consistent phase evolution $\zeta_j(\Phi^t(x)) \approx e^{i\omega_j t} \zeta_j(x) $ with more rapidly evolving orbits away from the fixed point.

In directions transverse to orbits, all computed eigenfunctions exhibit significant oscillatory behavior.
For instance, even eigenfunction $\zeta_2$, which has the smallest nonzero Dirichlet energy among eigenfunctions in our computed set, exhibits an apparent wavenumber-$4$ oscillation along the circle $\theta_1=\theta_2$ (bottom-left to top-right corner in the $[0,2\pi)^2$ periodic box), as opposed to the wavenumber-1 eigenfunctions from the rotation system.
As we will see in \cref{sec:evo} below, the strongly oscillatory nature of the eigenfunctions computed from the Stepanoff system will have repercussions on their ability to represent smooth, isotropic observables in EDMD-type approximations.

\subsection{Dynamical evolution experiments}
\label{sec:evo}

We use the dynamical systems and spectral decomposition results from \cref{sec:dyn_syst_examples,sec:spec_decomp} to benchmark the efficacy of our tensor network scheme based on the Fock space $F(\mathcal H_\tau)$ in approximating the Koopman evolution of observables relative to (i) the ``true'' model, using the Koopman operator implemented as a composition map by an analytically known flow (the torus rotation), or a numerical flow obtained by an ordinary differential equation (ODE) solver (applied to the Stepanoff system); (ii) a ``classical'', EDMD-type approximation using the eigenfunctions from \cref{sec:spec_decomp} as basis functions; and (iii) a ``quantum mechanical'' approximation that approximates pointwise evaluation of the forecast observable as expectation of a smoothed multiplication operator with respect to a quantum state.
The latter is effectively a tensor network model utilizing only the $n=1$ grading of the Fock space; see \cref{sec:von_mises}.
Pseudocode for the classical, quantum mechanical, and Fock space approximations is listed in \Cref{alg:classical,alg:qm}, and~\ref{alg:fock}, respectively.

In all experiments, we consider the evolution of an observable $p_{\bm \mu, \bm \kappa} \in C^\infty(\mathbb T^2)$ from the von Mises family of probability density functions on the 2-torus, \eqref{eq:von_mises} with $N=2$.
The following two reasons motivate our use of observables in this family as our forecast observables:
\begin{itemize}[wide]
	\item For any admissible value of the parameters $\bm \mu$ and $\bm \kappa$, $p_{\bm \mu, \bm \kappa}$ is a strictly positive function, and is thus suitable for testing the ability of approximation schemes to preserve positivity.
	\item For any admissible $\bm \mu$ and $\bm \kappa = (\kappa, \kappa)$, $\kappa>0$, $p_{\bm \mu,\bm \kappa}$ is an isotropic function, i.e., $p_{\bm \mu, \bm \kappa}(x) = p_{\bm \mu, \bm \kappa}(x')$ whenever $x - x' = (y, y) \mod 2\pi$.
	      Given their highly anisotropic nature, it is expected that the approximate Koopman eigenfunctions for the Stepanoff flow do not provide an efficient basis for representing $p_{\bm \mu, \bm \kappa}$ for such parameter values.
	      As a result, approximating the evolution of $p_{\bm \mu, \bm \kappa}$ constitutes a challenging test of prediction techniques utilizing approximate Koopman eigenfunctions.
\end{itemize}

In the sequel, we let $\bm{\hat p}_{\bm \mu, \bm \kappa} = (\langle \phi_0, p_{\bm \mu, \bm \kappa}\rangle_H, \ldots,\langle \phi_m, p_{\bm \mu, \bm \kappa}\rangle_H)^\top$ be the vector in $\mathbb C^{m+1}$ containing the leading $m + 1$ Fourier coefficients of $p_{\bm \mu, \bm \kappa}$, determined via~\eqref{eq:von_mises_fourier}.
Moreover, using the coefficient vectors $\bm{c}_j = (c_{0j}, \ldots, c_{m,j})^\top \in \mathbb C^{m+1}$ and the corresponding eigenfrequencies $\omega_{j,z,\tau,m}$ we define, for an integer $d \leq (n_\text{eig} - 1)/2$, the $(m+1)\times (2d + 1)$ matrix $\bm C_d = (\bm{c}_0, \ldots, \bm{c}_{2d})$, and the $(2d + 1) \times (2d + 1)$ diagonal unitary matrix $\bm U^t_d = [U^t_{ij}]$ with diagonal entries $U^t_{jj} = e^{i \omega_{j,z,\tau,m} t}$.
As in \cref{sec:num_impl}, $\mathcal Z_{\tau,d,m} = \spn \{\zeta_{0,z,\tau,m}, \ldots, \zeta_{2d,z,\tau,m} \}$ will be the subspace of $\mathcal H_{\tau,m}$ spanned by the leading $2d+1$ eigenfunctions of $W_{\tau,m}$ and $Z_{\tau,d,m}$ the orthogonal projection onto $\mathcal Z_{\tau,d,m}$.
With these definitions, the four types of predictive models for $p_{\bm \mu, \bm \kappa}$ that we examine below are as follows.

\subsubsection{True model}

The true model employs the Koopman evolution $f^{(t)}_\text{true} := p_{\bm \mu, \bm \kappa} \circ \Phi^t$ associated with the dynamical flow $\Phi^t\colon G \to G$.
As mentioned above, in the case of the torus rotation $\Phi^t$ is known analytically, $\Phi^t(x) = x + (1, \alpha) t \mod 2 \pi$.
In the case of the Stepanoff flow, $\Phi^t$ is approximated numerically by means of an ODE solver (see \ref{app:numerical} for further information).

\subsubsection{Classical approximation}

For a dimension parameter $d_\text{cl} \leq (n_\text{eig} - 1)/2$, the classical approximation utilizes the unitary evolution operators $U^t_{\tau,m}$ to approximate $f^{(t)}_\text{true}$ by
\begin{displaymath}
	f^{(t)}_\text{cl} := U^t_{\tau,m} Z_{\tau,d_\text{cl},m} p_{\bm \mu, \bm \kappa} = \sum_{j=0}^{2d_\text{cl}} e^{i \omega_{j,z,\tau,m}t} \langle \zeta_{j,z,\tau,m}, p_{\bm \mu, \bm \kappa} \rangle_{\mathcal H_\tau} \zeta_{j,z,\tau,m}.
\end{displaymath}
We may expand
\begin{displaymath}
	\zeta_{j,z,\tau,m} = \sum_{i=0}^{m} c_{ij} \psi_{i,\tau} = \sum_{i=0}^{m} \frac{c_{ij}}{\Lambda_{i,\tau/2}} K_\tau \phi_i,
\end{displaymath}
giving
\begin{displaymath}
	\langle \zeta_{j,z,\tau,m}, p_{\bm \mu, \bm \kappa}\rangle_{\mathcal H_\tau} = \sum_{i=0}^{m} \frac{\overline{c_{ij}}}{\Lambda_{i,\tau/2}}\langle K_\tau \phi_j, p_{\bm \mu, \bm \kappa} \rangle_{\mathcal H_\tau} = \sum_{i=0}^{m} \frac{\overline{c_{ij}}}{\Lambda_{i,\tau/2}}\langle \phi_j, K_\tau^* p_{\bm \mu, \bm \kappa} \rangle_H = \bm{c}_j^\dag \bm \Lambda_{\tau/2}^{-1} \bm{\hat f},
\end{displaymath}
where $\bm{\hat f} = \bm{\hat p}_{\bm \mu, \bm \kappa}$ and $\bm \Lambda_{\tau/2}$ is the ${m + 1} \times {m + 1}$ diagonal matrix with diagonal entries $\Lambda_{0,\tau/2}, \ldots, \Lambda_{m,\tau/2}$.
Thus, defining the vector-valued function $\bm \zeta_d\colon G \to \mathbb C^{2d+1} $ as $\bm \zeta_d(x) = (\zeta_{0,z,\tau,m}(x), \ldots, \zeta_{2d,z,\tau,m}(x))^\top$, we can succinctly express $f^{(t)}_\text{cl}$ as
\begin{equation}
	\label{eq:f_cl}
	f^{(t)}_\text{cl} = (\bm C^\dag_{d_\text{cl}}  \bm{\hat f})^\top  \bm U^t_{d_\text{cl}} \bm \zeta_{d_\text{cl}}.
\end{equation}

\subsubsection{Quantum mechanical approximation}
\label{sec:qm_model}

Following the approach described in \cref{sec:quantum,sec:pointwise}, the quantum mechanical approximation represents $p_{\bm \mu, \bm \kappa}$ by the quantum mechanical observable $M_{f,\tau,l} \in \mathfrak B_\tau$ from~\eqref{eq:discrete_mult_op} for $f = \iota p_{\bm \mu, \bm \kappa}$, and employs the quantum feature maps from \cref{sec:pointwise} to approximate pointwise evaluation at $x \in \mathbb T^2$ by the vector state $\Xi_{\kappa_\text{eval},\tau,d,m}(x) \in S_*(\mathfrak B_\tau)$ for a concentration parameter $\kappa_\text{eval} > 0$.
Quantum states $\rho \in S_*(\mathfrak B_\tau)$ (including $\Xi_{\kappa_\text{eval},\tau,d,m}(x)$) evolve under the induced action $\mathcal P^t_{\tau,m}$ of the unitary $U^t_{\tau,m}$, i.e., $\mathcal P^t_{\tau,m} \rho = U^{t*}_{\tau,m} \rho U^t_{\tau,m}$.

As discussed in \cref{sec:num_impl}, the quantum observable $M_{f,\tau,l}$ is built by numerical quadrature associated with a finite-dimensional discretization $H_l$ of the Hilbert space $H$.
Let $x^{(l)}_0, \ldots, x^{(l)}_{l^2-1} $ be the nodes of a uniform $l \times l$ grid on $\mathbb T^2$ (with grid spacing $2\pi / l$).
Let $\mu_l$ be the corresponding sampling measure from~\eqref{eq:sampling_meas}, and $H_l = L^2(\mu_l)$ and $\mathfrak A_l = L^\infty(\mu_l)$ the associated Hilbert space and algebra of discretely sampled classical observables, respectively.
Let also $K_{\tau,l}\colon H_l \to \mathcal H_\tau$ be the integral operator from~\eqref{eq:k_op_discrete}.
We use the values $p_{\bm \mu, \bm \kappa}(x^{(l)}_0), \ldots, p_{\bm \mu, \bm \kappa}(x^{(l)}_{l^2-1}) \in \mathbb R$ and $K_{\tau,l}$ to build the quantum observable $M_{f,\tau,l} \in \mathfrak B_\tau$ from~\eqref{eq:discrete_mult_op}.

Below, we will be making use of the $l\times l$ diagonal matrix $\bm M$ with diagonal entries $\tilde f(x_0^{(l)}), \ldots, \tilde f(x_{l^2 -1}^{(l)})$ and the $l^2 \times (m+1)$ matrix $\bm \Phi_m = [\Phi_{ij}]$, whose entries contain the values of the leading $m+1$ Fourier functions on the points $x^{(l)}_i$; i.e., $\Phi_{ij} = \phi_j(x^{(l)}_i)$.
We also introduce a vector-valued function $\bm{\hat{F}}_\kappa\colon \mathbb T^2 \to \mathbb C^{m+1}$, $\bm{\hat{F}}_{\kappa}(x) = \bm{\hat p}_{x,(\kappa, \kappa)}$, that returns, up to proportionality, the leading $m+1$ Fourier coefficients of the feature vectors $F_{\kappa, \tau}(x)$.

With these definitions, and for an integer parameter $d_\text{qm} \leq (n_\text{eig}-1)/2$, the quantum mechanical model approximates the evolution $f^{(t)}_\text{true}$ by $f^{(t)}_\text{qm} \in C(\mathbb T^2)$ using the first formula in~\eqref{eq:pointwise_eval_d_l}, i.e.,
\begin{align}
	\nonumber f^{(t)}_\text{qm}(x)
	                 & = \frac{\mathbb E_{\mathcal P^t_{\tau,m}(\Xi_{\kappa_\text{eval},\tau,d_\text{qm},m}(x))}M_{f,\tau,l}}{\mathbb E_{\mathcal P^t_{\tau,m}(\Xi_{\kappa_\text{eval},\tau,d_\text{qm},m}(x))}M_{\bm 1,\tau,l}}                                                                                                     \\
	\label{eq:ft_qm} & \equiv \frac{\langle K_{\tau,l}^* U^{t*}_{\tau,m} F_{\kappa_\text{eval},\tau,d_\text{qm},m}(x), (\pi_l f_l) K_{\tau,l}^* U^{t*}_{\tau,m} F_{\kappa_\text{eval},\tau,d_\text{qm},m}(x) \rangle_{H_l}}{\lVert \langle K_{\tau,l}^* U^{t*}_{\tau,m} F_{\kappa_\text{eval},\tau,d_\text{qm},m}(x)\rVert^2_{H_l}}.
\end{align}
In matrix notation, \eqref{eq:ft_qm} can be equivalently expressed as
\begin{equation}
	\label{eq:ft_qm_matrix}
	f^{(t)}_\text{qm}(x) = \frac{\bm\xi_{t,x}^\dag \bm M \bm\xi_{t,x}}{\lVert \bm\xi_{t,x}\rVert^2},
	\quad \bm \xi_{t,x} = \bm \Phi_m \bm C_{d_\text{qm}} \bm U^{t\dag}_{d_\text{qm}} \bm C^\dag_{d_\text{qm}} \bm{\hat{F}}_{\kappa_\text{eval}}(x).
\end{equation}

\subsubsection{Fock space approximation}
\label{sec:fock_space_model}

The Fock space approximation implements the scheme described in \cref{sec:tensornet} using the quantum observable $A_{f, \sigma, \tau, n, l}$ from~\eqref{eq:discrete_a_fock} and the feature map $\tilde \Xi_{\kappa_\text{eval},\tau,d,m}$ derived from the projections of the roots of the feature vectors $F_{\kappa_\text{eval},\tau}(x)$ onto $\mathcal Z_{\tau,d,m}$.

Setting an integer parameter $d_\text{Fock} \leq (n_\text{eig}-1)/2$ and using the above feature maps, the Fock space approximation $f^{(t)}_\text{Fock} \in C(\mathbb T^2)$, obtained via the second formula in~\eqref{eq:pointwise_eval_d_l}, becomes
\begin{align}
	\nonumber
	f^{(t)}_\text{Fock}(x)
	                  & = \frac{\mathbb E_{\tilde{\mathcal P}^t_{\tau,m}(\tilde \Xi_{\kappa_\text{eval},\tau, n,d_\text{Fock},m}(x))}A_{f, \sigma, \tau, n, l}}{\mathbb E_{\tilde{\mathcal P}^t_{\tau,m}(\tilde\Xi_{\kappa_\text{eval},\tau,d_\text{Fock},m}(x))}A_{\bm 1,n,\tau,l}} \\
	\label{eq:f_fock} & \equiv \frac{\langle (G_\sigma K_\tau^* U^{t*}_{\tau,m}F_{\tau, n,d,m}(x))^n,(\pi_l f_l) (G_\sigma K_\tau^* U^{t*}_{\tau,m}\tilde F_{\tau,d,m}(x))^n\rangle_{H_l}}{\lVert  (G_\sigma K_\tau^* U^{t*}_{\tau,m}\tilde F_{\tau,d,m}(x))^n\rVert_{H_l}^2}.
\end{align}
Defining $\bm{\hat{F}}_{\kappa, n}:G \to \mathbb C^m$ as $\bm{\hat{F}}_{\kappa,n}(x) = \bm{\hat p}_{x,(\kappa, \kappa)/n}$, and using $\bm v^{\bm . n}$ to denote the elementwise $n$-th power of a vector $\bm v \in \mathbb C^{l^2}$, we can express~\eqref{eq:f_fock} in matrix form as
\begin{equation}
	\label{eq:ft_fock_matrix}
	f^{(t)}_\text{Fock}(x) = \frac{\bm\xi_{n,t,x}^\dag \bm M \bm\xi_{n,t,x}}{\lVert \bm\xi_{n,t,x}\rVert^2},
	\quad \bm \xi_{n,t,x} = (\bm \Phi_m \bm \Lambda_\sigma \bm C_{d_\text{Fock}} \bm U^{t\dag}_{d_\text{Fock}} \bm C^\dag_{d_\text{Fock}} \bm{\hat{F}}_{\kappa_\text{eval},n}(x))^{\bm{.}n}.
\end{equation}
Note that, up to proportionality constants that cancel upon taking the ratio, the numerator and denominator in~\eqref{eq:f_fock} are evaluations of the tensor network in \cref{fig:network}.

\subsubsection{Torus rotation}
\label{sec:torusrot}

In our torus rotation experiments, the forecast observable is an anisotropic von Mises density $p_{\bm \mu,\bm \kappa}$ with $\bm \mu = (0, 0)$ and $\bm \kappa = (\kappa_1, \kappa_2) \equiv (1, 6)$.
Since $\kappa_2 > \kappa_1$, $p_{\bm \mu, \kappa}$ has higher sharpness (faster rate of decay) in the $\theta_2$ dimension than along $\theta_1$, resulting in an oblate appearance of its level sets when visualized in the $(\theta_1, \theta_2)$ periodic domain.
Correspondingly, well-approximating $p_{\bm \mu, \bm \kappa}$ by truncated Fourier series requires more coefficients (higher bandwidth) along the $\theta_2$ dimension than $\theta_1$.

\Cref{fig:evo_torusrot} shows scatterplots of the evolution of $p_{\bm \mu, \bm \kappa}$ under the true model, and the classical, quantum mechanical, and Fock space approximations for representative evolution times $t \in \{ 0, 1, 2, 4 \}$, plotted on a $256 \times 256$ uniform grid on $\mathbb T^2$.
The approximate models use the parameters $d_\text{cl} = d_\text{qm} = d_\text{Fock} = 128$, $l = 512$, $\kappa_\text{eval} = 500$, $\sigma = 0$, and $n=10$.
Note that working with the same values of $d_\text{cl}$, $d_\text{qm}$, and $d_\text{Fock}$ allows us to compare approximations originating from the same set of eigenfunctions and eigenfrequencies.
Moreover, since $\kappa_\text{eval} \gg \max \{ \kappa_1, \kappa_2 \}$ (and as we discuss momentarily, the initial condition $p_{\bm \mu,  \bm \kappa}$ does not experience deformation of its level sets under rotation dynamics), the quantum states utilized by the quantum mechanical and Fock space models are expected to well-approximate pointwise evaluation of the forecast observables.

\begin{figure}
	\centering
	\includegraphics[draft=false, width=\textwidth]{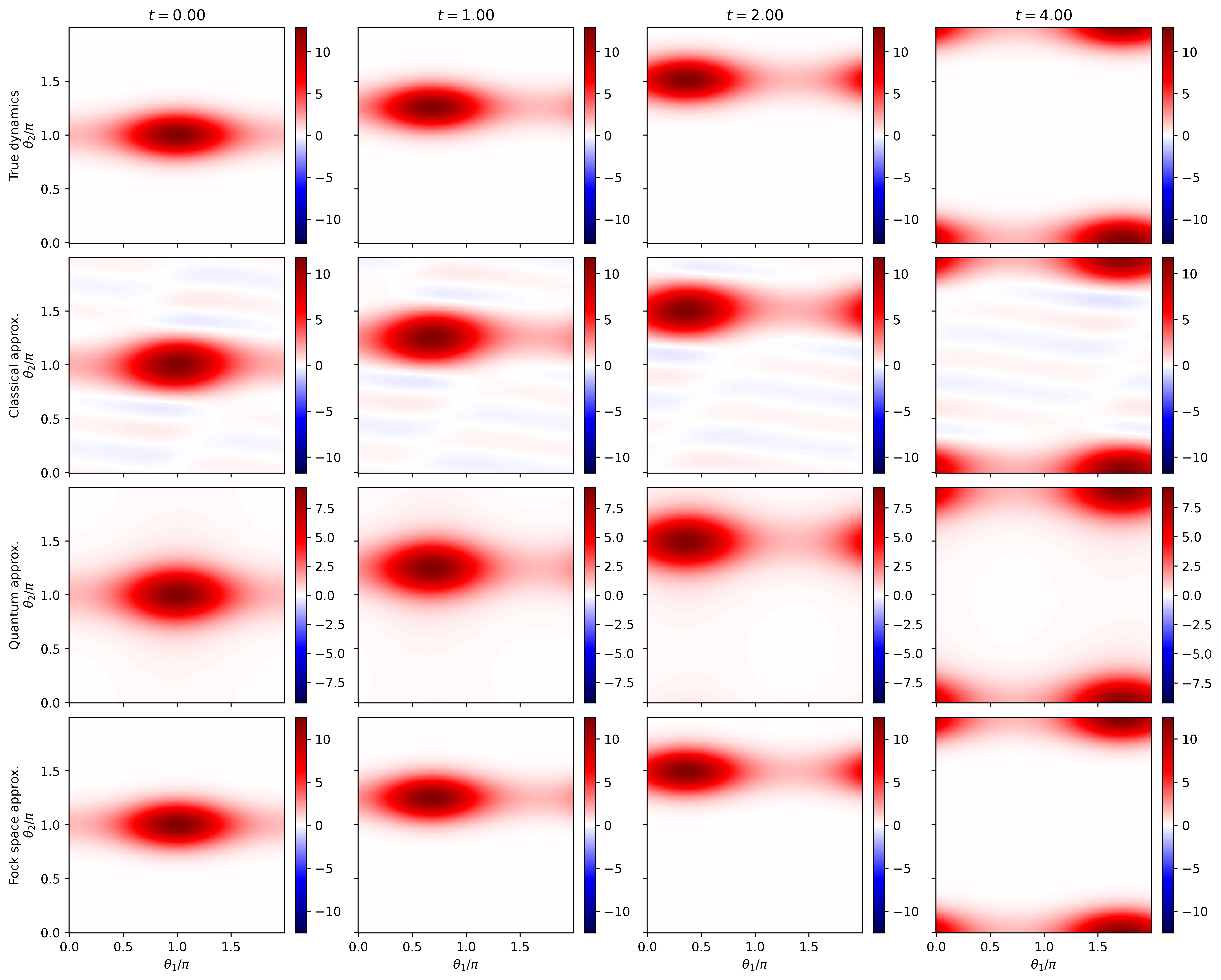}
	\caption{Time evolution of the anisotropic von Mises density $p_{\bm \mu, \bm \kappa}$ with $\bm \mu = (0, 0)$ and $\bm \kappa = (1, 6)$ under the ergodic torus rotation ($f^{(t)}_\text{true}$; top row), the classical approximation ($f^{(t)}_\text{cl}$; second row from top), the quantum mechanical approximation ($f^{(t)}_\text{qm}$; third row from top), and the Fock space approximation ($f^{(t)}_\text{Fock}$; bottom row).
		Columns from left-to-right show snapshots of the evolution times $t=0, 1, 2, 4$.}
	\label{fig:evo_torusrot}
\end{figure}

As expected from the analytically known flow map~\eqref{eq:phi_rot}, under the true model $f^{(t)}_\text{true}$ the initial condition $f^{(0)}_\text{true} = p_{\bm \mu, \bm \kappa}$ is advected as a rigid body along lines of (irrational) slope $\alpha_2/\alpha_1$.
This behavior is qualitatively captured by all three models, but there are notable differences in the details of the three approximations:
\begin{itemize}[wide]
	\item Being an orthogonal projection method, the classical approximation $f^{(t)}_\text{cl}$ is not positivity preserving (see \cref{sec:feature_extraction}).
	      Indeed, oscillations of $f^{{(t)}}_\text{cl}$ to negative values are evident in the results shown in the second row of \cref{fig:evo_torusrot}.
	      Notice that these oscillations take place primarily along the $\theta_2$ direction.
	      This is consistent with the fact that $p_{\bm \mu, \bm \kappa}$ is more susceptible to Fourier truncation errors along $\theta_2$ than $\theta_1$ (since $\kappa_2>\kappa_1$).
	\item The quantum mechanical approximation $f^{(t)}_\text{qm}$ is positivity preserving (see \cref{sec:quantum}), which eliminates the oscillations to negative values seen in $f^{(t)}_\text{cl}$.
	      In some applications, e.g., modeling sign-definite physical quantities, positivity preservation is a highly desirable, or even essential, feature.
	      However, as can be seen in the third row of \cref{fig:evo_torusrot}, $f^{(t)}_\text{qm}$ does not recover as accurately the pointwise value of $f^{(t)}_\text{true}$ near its peak, and generally appears to be more diffusive than $f^{(t)}_\text{cl}$.
	\item The Fock space approximation $f^{(t)}_\text{Fock}$ is positivity-preserving similarly to $f^{(t)}_\text{qm}$, but also overcomes the shortcomings of the latter method in reproducing the extremal values of $f^{(t)}_{\text{true}}$.
	      Moreover, the Fock space scheme has access to approximation spaces lying in the tensor product space $\mathcal H_\tau^{\otimes n}$.
	      As we have already seen in \cref{sec:circlerot} in the context of the circle rotation, the higher dimensionality of these spaces may avoid the apparent positivity preservation and extremal value reproduction tradeoffs seen in the classical and quantum mechanical approximation schemes.
\end{itemize}

Next, to study the accuracy of the three approximation schemes in more detail, in \cref{fig:err_torusrot} we show plots of their errors, $f^{(t)}_\text{cl} - f^{(t)}_\text{true}$, $f^{(t)}_\text{qm} - f^{(t)}_\text{true}$, and $f^{(t)}_\text{Fock} - f^{(t)}_\text{true}$, relative to the true evolution.
At least for $t \leq 2$, the results verify the qualitative observations made above; namely, $f^{(t)}_\text{cl}$ performs better than $f^{(t)}_\text{qm}$ in terms of reconstruction of the extremal values of $f^{(t)}$ but fails to preserve positivity, and $f^{(t)}_\text{Fock}$ provides ``the best of both worlds'' in the sense of being positivity-preserving and yielding the best reconstruction accuracy of extremal values.
At the same time, the $t=4$ results in \cref{fig:err_torusrot} indicate that $f^{(t)}_\text{Fock}$ has faster error growth than the other methods.
This can be understood from the fact that the amplification to $\mathcal H_\tau^{\otimes n}$ generates frequencies as high as $n=10$ times the maximal eigenfrequency corresponding to the approximation space $\mathcal Z_{\tau,d_\text{Fock},m}$.
Any errors in the eigenfrequencies may thus be amplified, causing a build-up of phase errors over time.

\begin{figure}
	\centering
	\includegraphics[draft=false, width=\textwidth]{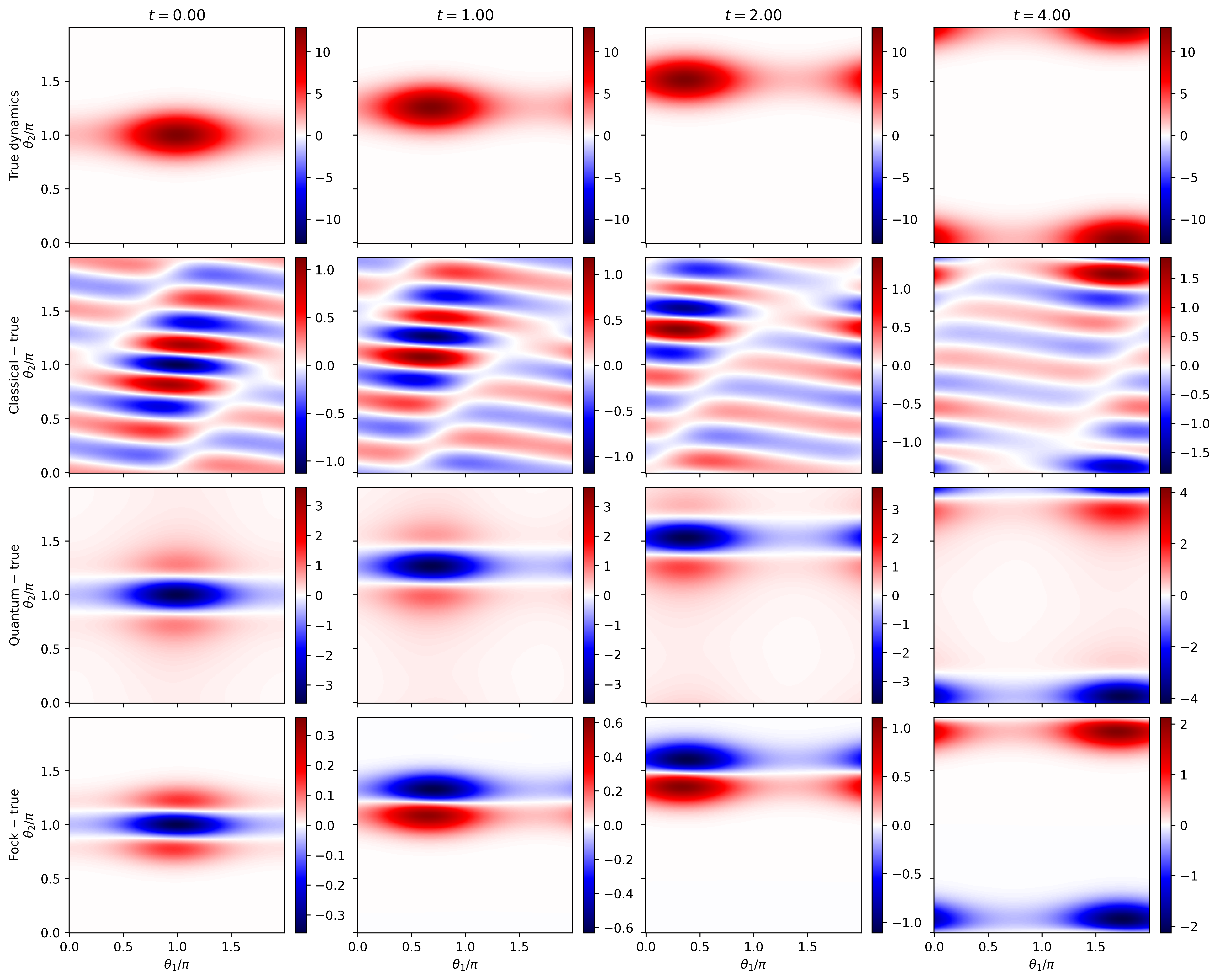}
	\caption{Errors in the classical, quantum mechanical, and Fock space approximations from \cref{fig:evo_torusrot} (second from top to bottom rows, respectively) relative to the true evolution.
		The true evolution is plotted in the first row for reference.}
	\label{fig:err_torusrot}
\end{figure}

A possible remedy to such issues would be to set $d_\text{Fock}$ and/or $n$ in a $t$-dependent manner, with values determined in an offline training procedure.
Here, we have opted not to explore such additional tuning steps in order to more directly highlight the benefits and shortcomings of each method.

\subsubsection{Stepanoff flow}
\label{sec:stepanoff}

For our Stepanoff flow experiments, we consider the isotropic von Mises density function $p_{\bm \mu, \bm \kappa}$ with $\bm \mu = (0,0)$ and $\bm \kappa = (1, 1)$ as the forecast observable.
We use the approximation parameters $d_\text{cl} = d_\text{qm} = d_\text{Fock} = 256$, $l = 1024$, $\kappa_\text{eval} = 500$, $\sigma = 0$, and $n=10$.
\Cref{fig:evo_stepanoff,fig:err_stepanoff} show representative dynamical evolution and error results for time $t \in \{ 0, 1, 2, 4 \}$ similarly to \cref{fig:evo_torusrot,fig:evo_stepanoff}, respectively.

\begin{figure}
	\centering
	\includegraphics[draft=false, width=\textwidth]{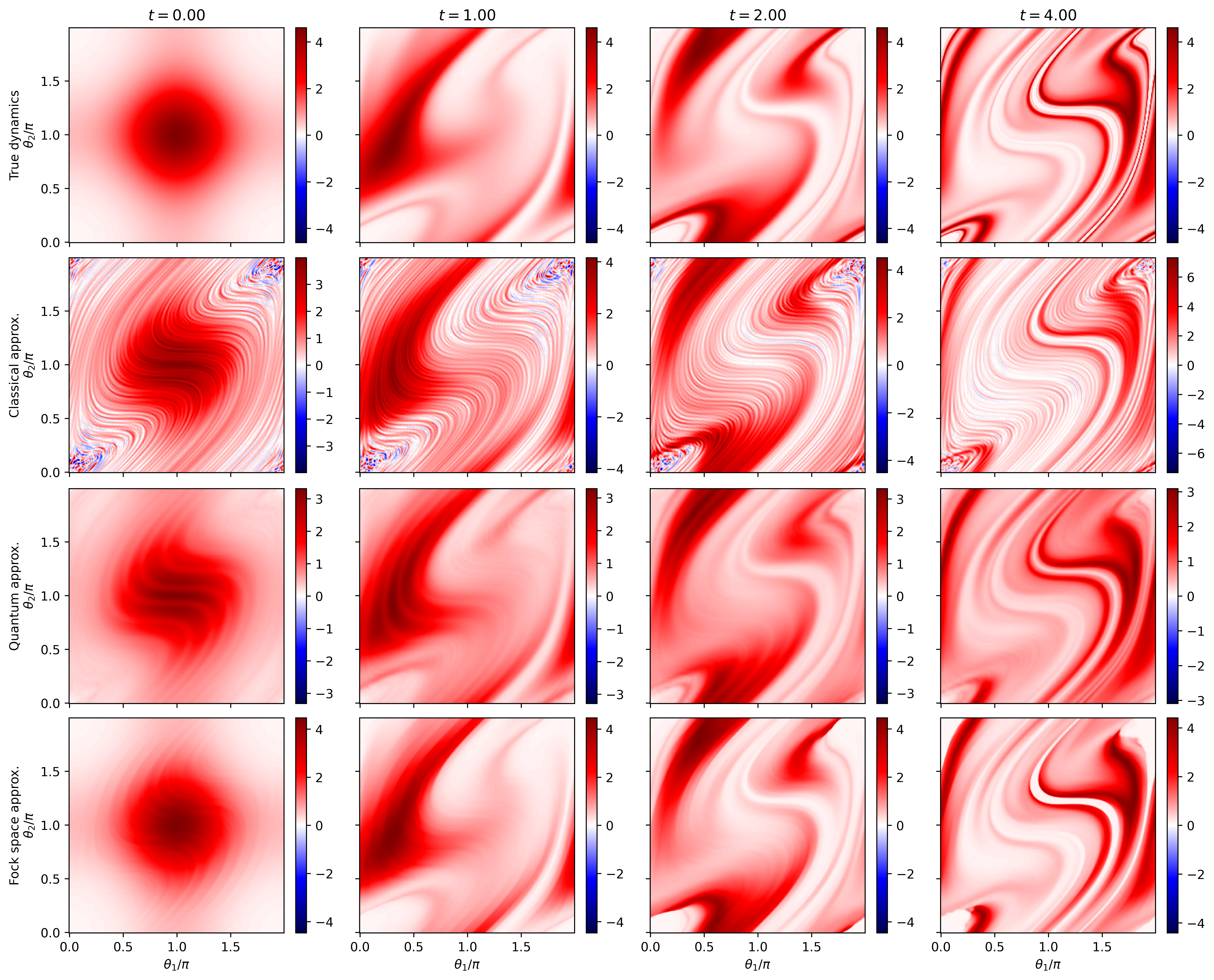}
	\caption{Time evolution of the isotropic von Mises density $p_{\bm \mu, \bm \kappa}$ with $\bm \mu = (0, 0)$ and $\bm \kappa = (1, 1)$ under the Stepanoff flow (top row), and the classical quantum mechanical, and Fock space approximations (second to last rows).
		The plots are organized similarly to \cref{fig:evo_torusrot}}.
	\label{fig:evo_stepanoff}
\end{figure}

\begin{figure}
	\centering
	\includegraphics[draft=false, width=\textwidth]{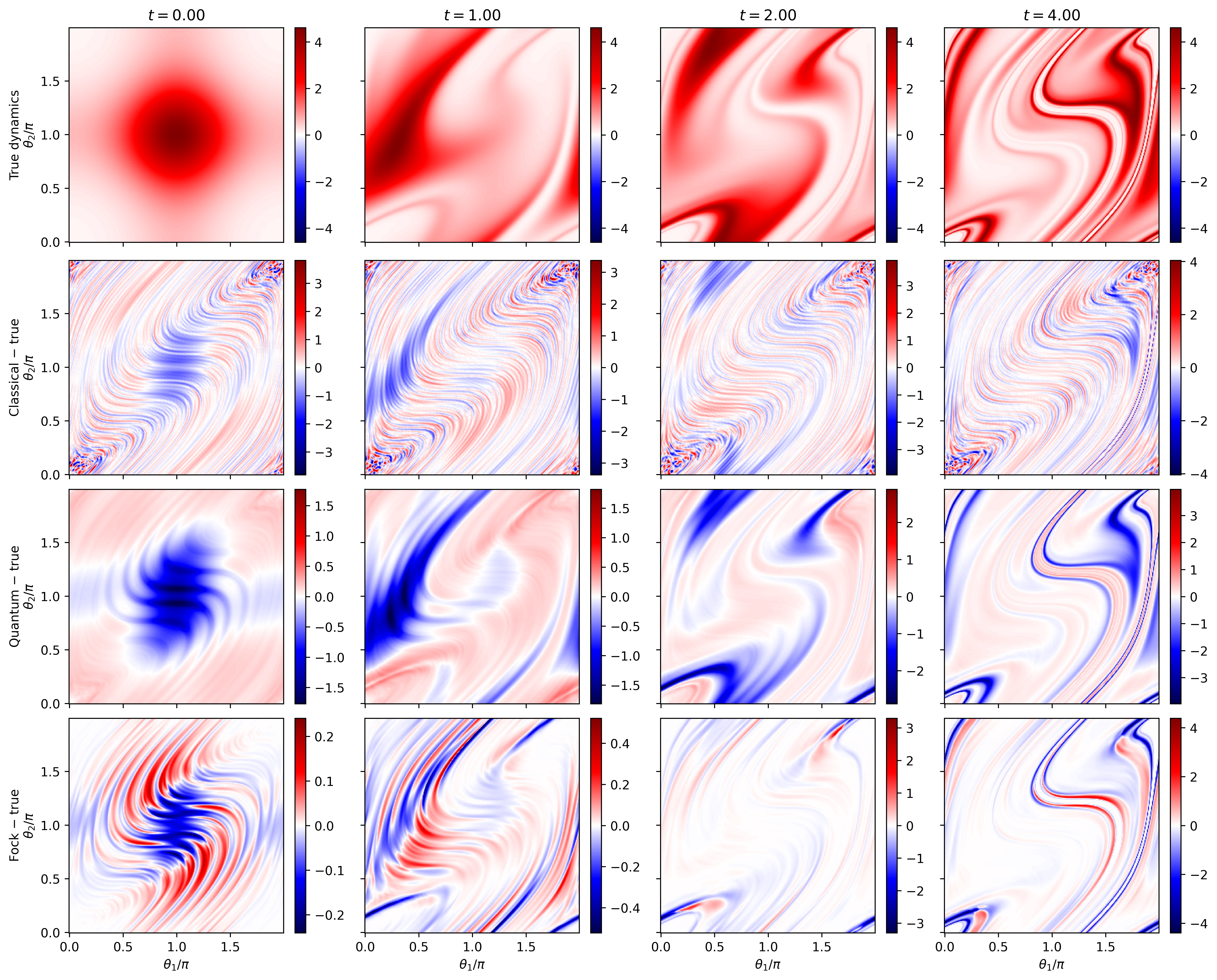}
	\caption{Errors in the classical, quantum mechanical, and Fock space approximations from \cref{fig:evo_stepanoff} (second from top to bottom rows, respectively) relative to the true evolution.
		The true evolution is plotted in the first row for reference.}
	\label{fig:err_stepanoff}
\end{figure}

First, before delving into the behavior of the dynamics and approximations at $t>0$, it is worthwhile examining the $t=0$ results to highlight that even reconstruction of the initial condition $f^{(0)}_\text{true}$ is non-trivial using the basis of approximate Koopman eigenfunctions computed for this system (as alluded to in \cref{sec:spec_decomp}).
This is especially evident in the case of the classical approximation, where the initial reconstruction $f^{(0)}_\text{cl}$ in \cref{fig:evo_stepanoff} has clear imprints of the characteristic sigmoid structure of the eigenfunction level sets (see right-hand column of \cref{fig:generator_eig}) and exhibits significant excursions to negative values in regions near the fixed point $x=(0,0)$.
This is despite the fact that the true initial condition $f^{(0)}$ is a smooth, isotropic function (and thus might be considered a ``nice'' observable) and the number eigenfunctions used, $2 d_\text{cl} + 1 = 513$, is arguably not insignificant.
The quantum mechanical reconstruction, $f^{(0)}_\text{qm}$, improves reconstruction quality by smoothing and eliminating negative values, but the sigmoid pattern inherited from the eigenfunctions is still clearly visible, and as in the case of the torus rotation the amplitude of extremal values is not well-reproduced.
The Fock space reconstruction $f^{(0)}_\text{Fock}$, on the other hand, offers a solid improvement over both its classical and quantum mechanical counterparts.
Even though some undulations attributable to the form of the eigenfunctions are still present in $f^{(0)}_\text{Fock}$, they are considerably weaker than those seen in $f^{(0)}_\text{cl}$ and $f^{(0)}_\text{qm}$.
Moreover, as can be seen from the error plots in \cref{fig:err_stepanoff}, the amplitude of extremal values of $f^{(0)}$ is reproduced with significantly higher accuracy.

Next, turning to the $t>0$ results, a comparison of the top rows of \cref{fig:evo_torusrot,fig:evo_stepanoff} clearly illustrates the higher dynamical complexity of the Stepanoff flow over the torus rotation, with the former producing stretching and folding of the initially circular level sets of $p_{\bm \mu,\bm \kappa}$ versus the rigid translation seen in the latter.
All three approximations track the large-scale behavior of $f^{(t)}_\text{true}$ from the true system, although the classical approximation $f^{(t)}_\text{cl}$ continues to be affected by imprinting of eigenfunction the patterns as discussed above for $t=0$.
Again, $f^{(t)}_\text{Fock}$ provides the best overall prediction results, particularly for $t \leq 2$, demonstrating the benefits of Fock space amplification.

Inspecting the $t=2, 4$ results in \cref{fig:evo_stepanoff,fig:err_stepanoff}, we see that the largest errors in $f^{(t)}_\text{Fock}$ are concentrated in regions such as the vicinity of the fixed point, where the approximation appears to be rounded off (visually ``whitened out'') to near-zero values.
In separate calculations, we have verified that at such points $x \in \mathbb T^2$ the state vector $\tilde F_{\kappa_\text{eval}, \sigma, \tau, n, d, m}(x) \in F(\mathcal H_\tau)$ becomes an extremely localized function upon application of $\Delta_n^*$.
Thus, the forecast $f^{(t)}_\text{Fock}(x)$ is dominated by an exceedingly small number of sampled values $p_{\bm \mu, \bm \kappa}(x_i^{(l)})$ of the forecast observable.
In other words, the approximation $f^{(t)}_\text{Fock}(x)$ at these points appears to be affected by some type of overfitting.
In \cref{fig:evo_stepanoff,fig:err_stepanoff}, the quantum mechanical approximation $f^{(t)}_\text{qm}(x)$ does not seem to be impacted by the same issue.
Since $f^{(t)}_\text{qm}$ is essentially an $n=1$, $\sigma=0$, version of the Fock space scheme, decreasing $n$ and/or increasing $\sigma$ would be a way of restoring smoothness of the prediction $f^{(t)}_\text{Fock}(x)$ under changes in $x$.

Of course, sensitive dependence on initial conditions is a hallmark of complex dynamics, so the fact that $f^{(t)}_\text{Fock}(x)$ appears to have higher sensitivity on $x$ than other methods is not necessarily a disadvantage of the tensor network scheme.
Nonetheless, the presence of the whitened-out regions in \cref{fig:evo_stepanoff} indicates that $t$-dependent tuning of $n$/ $\sigma$ or modification of the quantum observable $A_{f, \sigma, \tau, n}$ used for prediction may be beneficial.
We comment on possibilities for the latter approach in \cref{sec:conclusions}.

\section{Concluding remarks}
\label{sec:conclusions}

In this paper, we developed a framework for approximating the statistical evolution of continuous observables of measure-preserving, ergodic dynamical systems based on a combination of techniques from Koopman/transfer operator theory and many-body quantum theory.
For this framework, the working hypothesis was that the multiplicative structure of approximate eigenfunctions of Koopman operators contains information about the underlying dynamical system that is not utilized by classical Koopman operator techniques.
The principal elements of this scheme are (i) spectral regularization for consistently approximating the skew-adjoint Koopman generator (which is typically unbounded and non-diagonalizable) by a family of skew-adjoint, diagonalizable operators, $V_\tau$, on $L^2(\mu)$; (ii) lift of the evolution of observables and probability densities under the dynamics generated by $V_\tau$ to the setting of a Fock space $F(\mathcal H_\tau)$ generated by a reproducing kernel Hilbert algebra (RKHA) of observables, $\mathcal H_\tau$.
Notable features of this construction are:
\begin{enumerate}[wide]
	\item The lifted Koopman/transfer operators act multiplicatively on the tensor algebra $T(\mathcal H_\tau) \subset F(\mathcal F_\tau)$.
	      This means that their eigenvalues are generative as an additive group, and the corresponding eigenfunctions are generative as a multiplicative group under the tensor product.
	\item The RKHA $\mathcal H_\tau$ is simultaneously an RKHS, a coalgebra, and a Banach algebra with respect to the pointwise product of functions.
	      This structure enables a mapping of probability densities on state space into quantum state vectors in the Fock space that project non-trivially to all of its gradings $\mathcal H_\tau^{\otimes n} \subset F(\mathcal H_\tau)$.
	      Moreover, classical observables $f \in L^\infty$ are mapped into quantum observables $A_{f, \sigma, \tau, n}$ acting on the Fock space.
	      The observables $A_{f, \sigma, \tau, n}$ are built as smoothed multiplication operators, amplified to act on any of the tensor product spaces $\mathcal H^{\otimes (n+1)}$, $n \in \mathbb N$, via the comultiplication operator $\Delta \colon \mathcal H_\tau \to \mathcal H_\tau \otimes \mathcal H_\tau$ of the RKHA.
	\item The lifted system on the Fock space is projected onto finite-dimensional subspaces generated by tensor products of eigenfunctions of $V_\tau$.
	      The resulting quantum state evaluation can be expressed as a tree tensor network (\cref{fig:network}), and is positivity preserving and asymptotically consistent in appropriate limits.
	      Importantly, the tensor network approach provides a route for algebraically building high-dimensional approximation spaces from a modest number of approximate Koopman/transfer operator eigenfunctions.
\end{enumerate}

In broad terms, our approach can be thought of as (i) distributing an initial sharp probability density function $\varrho \in \mathcal H_\tau$ into a linear combination of tensor products $\xi^{1/n} \otimes \cdots \otimes \xi^{1/n} \in \mathcal H_\tau^{\otimes n} $ of coarser functions by the $n$-th roots of the quantum mechanical state vector $\xi \propto \varrho^{1/2}$; (ii) dynamically evolving each of the factors $\xi^{1/n}$; and (iii) recombining the results using into a sharp, time-evolved quantum state that approximates the evolution of densities under the transfer operator of the dynamical system.
Our analysis characterized the convergence properties of the scheme (\cref{thm:tau_conv,thm:torus_rotation_convergence}), and its efficacy was assessed with numerical prediction experiments involving an irrational rotation and a Stepanoff flow on $\mathbb T^2$.
In both cases, approximations utilizing the Fock space were found to perform better than conventional subspace projection methods in terms of approximation accuracy, consistent with the working hypothesis, and with the added benefit of being positivity preserving.

In terms of future work, the tensor network approach presented in this paper provides a flexible approximation framework with several possible directions for further development.
First, as noted in \cref{rk:roots}, the scheme based on the $n$-th roots $\xi^{1/n}$ is one of many possibilities of lifting the state vector $\xi \in \mathcal H_\tau$ to the Fock space.
From the point of view of asymptotic convergence, it suffices to identify any set of functions $\xi^{(1)}, \ldots, \xi^{(n)}$ whose pointwise product $\Delta_n^* (\xi^{(1)} \otimes \ldots \otimes \xi^{(n)})$ recovers $\xi$.
Second, the amplification scheme leading to the quantum observables $A_{f, \sigma, \tau, n}$ that act on individual gradings $\mathcal H^{\otimes(n+1)}$ could be generalized to yield quantum observables acting on multiple gradings, e.g., through (possibly infinite) linear combinations $\sum_n c_n A_{f, \sigma, \tau, n}$.
Exploring these more general schemes would be a fruitful avenue of future work that could potentially improve the performance of the method at large evolution time $t$ and/or in regions of state space such as fixed points with highly oscillatory eigenfunctions (see \cref{sec:evo}).
The tensor network we utilized and the finite cut off made resembles a tensor train along with a finite rank approximation (see \cite{Oseledets11}).
Future work may focus on optimizations from tensor trains for the tensor network approach to Koopman operators.
A related research direction would be to characterize the structure of the resulting tree tensor networks using methods from renormalization group theory.

Finally, even though the methods presented in this paper are already beneficial in the context of classical numerical computation, their common mathematical foundation with many-body quantum physics motivates the development of efficient implementations on quantum computing platforms.
Here, a challenge stems from the fact that the comultiplication operators $\Delta$ employed in the scheme are not unitary, and thus cannot be directly implemented using quantum logic gates.
Recent work on representing non-unitary operators in tensor networks through orthogonal forms \cite{Ran20} may be a promising way of addressing these challenges.
More generally, we are hopeful that the work presented in this paper will stimulate fruitful interactions between the fields of classical dynamics and many-body quantum theory.

\section*{Acknowledgments}

Dimitrios Giannakis acknowledges support from the U.S.\ Department of Defense, Basic Research Office under Vannevar Bush Faculty Fellowship grant N00014-21-1-2946, the U.S.\ Office of Naval Research under MURI grant N00014-19-1-242, and the U.S.\ Department of Energy under grant DE-SC0025101.
Mohammad Javad Latifi Jebelli and Michael Montgomery were supported as postdoctoral fellows from the first two of these grants.
Philipp Pfeffer is supported by the project no.~P2018-02-001 ``DeepTurb -- Deep Learning in and of Turbulence'' of the Carl Zeiss Foundation, Germany.
J\"org Schumacher is supported by the European Union (ERC, MesoComp, 101052786).
Views and opinions expressed are however those of the authors only and do not necessarily reflect those of the European Union or the European Research Council.
Neither the European Union nor the granting authority can be held responsible for them.

\begin{appendices}

	\section{List of symbols}
	\label{app:symbols}

	\Cref{tab:notation} provides a summary of the main symbols related to dynamics, vector spaces, and operators used in this paper.

	\begin{table}[ht]
		\caption{Summary of symbols.}
		\label{tab:notation}
		\centering
		\begin{tabular}{ll}
			\hline\hline
			\textbf{Dynamics}                  & \textbf{Description}                                                        \\
			\hline
			$X$                                & Smooth finite dimensional compact manifold                                  \\
			$\mu$                              & Invariant probability measure on $X$                                        \\
			$\Phi^t$                           & Measure preserving flow on $X$                                              \\
			\hline                                                                                                           \\
			\textbf{Vector spaces}             & \textbf{Description}                                                        \\
			\hline
			$\mathcal H_\tau$                  & Reproducing kernel Hilbert algebra on $X$                                   \\
			$\zeta_{j,\tau}$                   & Eigenfunctions of $W_\tau$                                                  \\
			$\omega_{j,\tau}$                  & Eigenvalues of $W_\tau$                                                     \\
			$T(\mathcal H_\tau)$               & Tensor algebra generated by $\mathcal H_\tau$                               \\
			$F(\mathcal H_\tau)$               & Fock space construction of $\mathcal H_\tau$                                \\
			$\xi=\rho^{1/2}$                   & State vector in $\mathcal H_\tau$                                           \\
			$\eta_{\tau, n}$                   & State vector in $\mathcal H_\tau^{\otimes n}$                               \\
			$\xi_{\tau, n,d}$                  & Truncated state vector in $F(\mathcal Z_{\tau,d})$                          \\
			$\mathcal Z_{\tau,d}$              & Finite dimension truncation of $\mathcal H_\tau$                            \\
			$F(\mathcal Z_{\tau,d})$           & Fock space of finite dimensional truncation of $\mathcal H_\tau$            \\
			$\overrightarrow{j}$               & Multi-index $(j_1,\cdots,j_N)$                                              \\
			$\zeta_{\overrightarrow{j},\tau}$  & Eigenfunctions of $\tilde{U}^t_\tau$                                        \\
			$\omega_{\overrightarrow{j},\tau}$ & Eigenvalues of $\tilde{U}^t_\tau$                                           \\
			$p_{\bm \mu, \bm \kappa}$          & Von Mises distribution on an $N$-dimensional torus                          \\
			\hline                                                                                                           \\
			\textbf{Operators}                 & \textbf{Description}                                                        \\
			\hline
			$U^t$                              & Koopman operator associated to $\Phi^t$                                     \\
			$V$                                & Generator of the Koopman operator                                           \\
			$k_\tau$                           & Reproducing kernel of $\mathcal H_\tau$                                     \\
			$\mathcal K_\tau$                  & Kernel operator associated to $k_\tau$ and $\mu$                            \\
			$V_\tau$                           & Smoothed generator $\mathcal K_\tau V \mathcal K_\tau$                      \\
			$T_\tau$                           & Kernel operator associated to $k_\tau$ from $L^2(\mu)$ to $\mathcal H_\tau$ \\
			$W_\tau$                           & Smoothed generator $T_\tau V T_\tau^*$                                      \\
			$U^t_\tau$                         & Smoothed Koopman operator generated by $V_\tau$                             \\
			$\Delta$                           & Comultiplication operator on $\mathcal H_\tau$                              \\
			$\tilde{U}^t_\tau$                 & Amplification of smoothed Koopman operator                                  \\
			$\tilde{\mathcal U}^t_\tau$        & Adjoint action on observables under $\tilde{U}^t_\tau$                      \\
			$A_{f, \sigma, \tau, n}$           & Amplified observable associated to $f$ on $F(\mathcal H_\tau)$              \\
			$\mathbb E_{\rho}$                 & Expectation with respect to the density operator $\rho$                     \\
			\hline\hline
		\end{tabular}
	\end{table}

	\section{Spectral approximation techniques}
	\label{app:spectral_approx}

	In this appendix, we give an overview of approximation techniques for the Koopman generator $V$ from the papers \cites{DasEtAl21,GiannakisValva24,GiannakisValva25}.
	These methods utilize two alternative approximation strategies---one that approximates $V$ by a family of skew-adjoint compact operators (\cref{app:compact_approx}) and another one that approximates it by unbounded operators with compact resolvent (\cref{app:compact_res}).
	Both approaches satisfy \crefrange{prty:V1}{prty:V6} assumed for the tensor network schemes described in this paper.

	\subsection{Compact approximations}%
	\label{app:compact_approx}

	The scheme of \cite{DasEtAl21} produces a family of skew-adjoint compact operators $V_\tau\colon H \to H$ by smoothing $V$ by conjugation with the integral operators $\mathcal K_\tau$.
	These operators are defined as $V_\tau = \mathcal K_{\tau/2} V \mathcal K_{\tau/2}$, giving $W_\tau = K_\tau^* V K_\tau $ for the corresponding operators on the RKHA $\mathcal H_\tau$ which are also skew-adjoint and compact.
	It is shown \cite{DasEtAl21}*{Proposition~19} that under analogous assumptions to \crefrange{prty:K1}{prty:K6}, as $\tau \to 0^+$ $V_\tau$ converges strongly to $V$ on $D(V)$.
	For skew-adjoint operators this implies strong resolvent convergence (e.g., \cite{Oliveira09}*{Propositions~10.1.8, 10.1.8}) and thus strong dynamical convergence, as required.
	Moreover, skew-adjointness and compactness of $V_\tau$ means that its eigendecomposition can be computed stably and consistently using finite-rank compressions of the operator.

	Fixing $\tau>0$ and choosing $m \in \mathbb N$ such that $\Lambda_{m + 1,\tau} \neq \Lambda_{m,\tau}$, we have that the orthogonal projection $\Pi_m \colon H \to H$ with $\ran \Pi_m = \spn \{ \phi_0, \ldots, \phi_{m} \}$ projects onto a union of eigenspaces of $\mathcal K_\tau$.
	Our-finite rank approximations to $V_\tau$ are given by $V_{\tau,m} = \Pi_m \mathcal K_{\tau/2} V \mathcal K_{\tau/2} \Pi_m$, and are concretely represented by $(m + 1) \times (m + 1)$ skew-symmetric matrices $\bm V_{\tau,m} = [V^{(\tau)}_{ij}]$ with elements $V^{(\tau)}_{ij} = \Lambda_{\tau/2, i} \langle \phi_i, V \phi_j \rangle_H \Lambda_{\tau/2,j}$.
	Similarly, using the basis vectors $\psi_{j,\tau}$ from~\eqref{eq:psi}, we define the subspaces $\mathcal H_{\tau,m} = \spn \{ \psi_{0,\tau}, \ldots, \psi_{m,\tau} \} \subset \mathcal H_\tau$, the orthogonal projections $\Psi_{\tau,m} \colon \mathcal H_\tau \to \mathcal H_\tau$ with $\ran \Psi_{\tau,m} = \mathcal H_{\tau,m}$, and the compressed operators $W_{\tau,m} = \Psi_{\tau,m} W_\tau \Psi_{\tau,m}$.
	By construction, the matrix elements $W_{ij} = \langle \psi_{i,\tau}, W_\tau \psi_{j,\tau}\rangle_{\mathcal H_\tau}$ are equal to $V_{ij}$.

	To compute these matrix elements, we take advantage of~\eqref{eq:v_dot_grad} in conjunction with the fact that the $\phi_j$ have smooth representatives $\varphi_j \in \mathcal H_\infty$, giving
	\begin{equation*}
		V_{ij} = \Lambda_{\tau/2,i}\langle \iota \varphi_i, \iota \vec V \cdot \nabla \varphi_j \rangle_H \Lambda_{\tau/2,j}.
	\end{equation*}
	In \cite{DasEtAl21}, the directional derivatives $\vec V \cdot \nabla \varphi_j$ are approximated using finite-difference methods along sampled dynamical trajectories.
	Another approach \cite{GiannakisValva25} takes advantage of the kernel integral representation of $\varphi_j$ (see~\eqref{eq:k_eig}) to evaluate these derivatives via automatic differentiation.
	Specifically, we have
	\begin{equation*}
		\vec V \cdot \nabla \varphi_j = \frac{1}{\Lambda_{j,\tau}} \int_X \vec V \cdot \nabla k_\tau(\cdot, x) \phi_j(x) \, d\mu(x),
	\end{equation*}
	and the function $\vec V \cdot \nabla k_\tau(\cdot, x)$ can be computed via automatic differentiation of the kernel section $k_\tau(\cdot, x)$ at any point $x' \in G$ where the tangent vector $\vec V(x') \in T_{x'}G$ is known.
	In scenarios with known equations of motion, this will include every point $x' \in G$.

	Finally, computation of $V_{ij}$ requires a means of evaluating $\langle \cdot, \cdot \rangle_H$ inner products.
	A common data-driven approach (adopted, e.g., in \cite{DasEtAl21}) is to approximate the invariant measure $\mu$ by a time average along a dynamical trajectory.
	By ergodicity, such an approximation converges for $\mu$-a.e.\ initial point $x_0 \in G$ of that trajectory.
	In systems with so-called observable invariant measures (e.g., \cite{Blank17}), convergence holds for initial conditions $x_0$ in a set of positive ambient (Haar) measure.
	In the examples of \cref{sec:experiments}, the components of $\vec V$ in a frame of $G$-invariant vector fields are polynomials in the characters $\gamma$ of the group.
	This allows evaluation of the matrix elements $V_{ij}$ analytically.

	Next, we compute the eigenvalues and corresponding eigenvectors of $V_{\tau,m}$ and $W_{\tau,m}$ by solving the matrix eigenvalue problem
	\begin{displaymath}
		\bm V_{\tau, m} \bm c_j = i \omega_{j,\tau,m} \bm c_j,
	\end{displaymath}
	where the eigenvectors $\bm c_j = (c_{0j}, \ldots, c_{mj})^\top \in \mathbb C^{m + 1}$ are orthonormal with respect to the standard inner product on $\mathbb C^{m + 1}$.
	The elements of these eigenvectors are expansion coefficients of eigenvectors $u_{j,\tau,m} = \sum_{i=0}^{m} c_{ij} \phi_j \in H_\tau$ of $V_{\tau,m}$ and eigenvectors $\zeta_{j,\tau,m} = \sum_{i=0}^{m} c_{ij} \psi_{i,\tau} \in \mathcal H_\tau(X)$ of $W_{\tau,m}$, both with corresponding eigenfrequencies $\omega_{j,\tau,m}$ (cf.~\eqref{eq:w_tau_eig}).
	Note that the eigenvectors $\zeta_{j,\tau,m}$ are everywhere-defined functions (vs.\ equivalence classes of $\mu$-a.e.\ defined functions in the case of $u_{j,\tau,m}$) and can be manipulated as function objects in computational implementations (allowing, e.g., for automatic differentiation and lazy evaluation).

	Following \cite{DasEtAl21}, we order eigenvectors and eigenfrequencies in order of increasing Dirichlet energy $\mathcal E_\tau(u_{j,\tau,m})$.
	The latter, can be computed directly from the $\bm c_j$ through the formula
	\begin{equation}
		\label{eq:dirichlet_eig}
		\mathcal E_\tau(u_{j,\tau,m}) = \frac{\sum_{i=1}^{m} \lvert c_{ij}\rvert^2\Lambda_{i,\tau}^{-1}}{\sum_{i=0}^{m} \lvert c_{ij}\rvert^2}.
	\end{equation}
	With this choice, we build the approximation spaces in subsequent steps of our approach so as to contain functions of high regularity with respect to $\mathcal E_\tau$.

	\subsection{Approximations with compact resolvent}
	\label{app:compact_res}

	Rather than approximating the generator $V$ by compact operators, the schemes of \cite{GiannakisValva24,GiannakisValva25} approximate it by a two-parameter family of skew-adjoint operators $V_{z,\tau} \colon D(V_{z,\tau}) \to H$, $z,\tau>0$, with \emph{compact resolvent}.
	An advantageous aspect of this form of approximation is that infinite-rank operators with compact resolvents have genuinely discrete spectra.
	This is in contrast to infinite-rank compact operators such as $V_\tau$ from \cref{app:compact_approx}, which exhibit accumulation of the eigenvalues at 0 causing potentially high sensitivity of the spectrum to approximation parameters such as $\tau$.
	Infinite-rank operators with compact resolvents are also unbounded, and thus structurally closer in some sense to the unperturbed generator.
	In this subsection, we give an overview of the resolvent compactification scheme proposed in \cite{GiannakisValva25}.

	For $z>0$, define the resolvent function $r_z\colon i \mathbb R \to \mathbb C$ on the imaginary line as
	\begin{displaymath}
		r_z(i\omega) = \frac{1}{z - i \omega}.
	\end{displaymath}
	This function gives the resolvent $R_z\colon H \to H$ of the generator via the Borel functional calculus,
	\begin{displaymath}
		R_z = r_z(V) = \frac{1}{z - V}.
	\end{displaymath}

	Let $\tilde H = \{f \in H: \langle \bm 1, f \rangle_H = 0 \}$ be closed subspace of $H$ consisting of zero-mean functions with respect to the invariant measure $\mu$.
	Let also $\tilde \Pi = \Id - \langle \bm 1, \cdot \rangle_H \bm 1 \in B(H)$ be the orthogonal projection mapping into $\tilde H$.
	Since $V \bm 1 = \bm 0$, $\tilde H$ is an invariant subspace of $V$.
	The scheme of \cite{GiannakisValva25} considers the skew-adjoint, bounded operator $Q_z\colon \tilde H \to \tilde H$,
	\begin{displaymath}
		Q_z = q_z(V\rvert_{\tilde H}) = (R_z^* V R_z)\rvert_{\tilde H},
	\end{displaymath}
	obtained from the function $q_z\colon i \mathbb R \to \mathbb C$ (again via the Borel functional calculus), where
	\begin{displaymath}
		q_z(i\omega) = \frac{i\omega}{z^2 + \omega^2}, \quad \ran q_z = i\left[-\frac{1}{2z}, \frac{1}{2z}\right].
	\end{displaymath}
	While $q_z$ is not an invertible function, its restriction $\tilde q_z$ on the subset $\Omega_z := i(-\infty, -z] \cup i[z, \infty)$ of the imaginary line \emph{is} invertible, with inverse $\tilde q_z^{-1} \colon i[-(2z)^{-1}, (2z)^{-1}] \setminus \{ 0 \} \to i \mathbb R$,
	\begin{equation}
		\label{eq:q_inv}
		\tilde q_z^{-1}(i\omega) = i \frac{1 + \sqrt{1 - 4 z^2 \omega^2}}{2\omega}.
	\end{equation}
	As a result, $\tilde V_z := \tilde q_z^{-1}(Q_z)$ is a skew-adjoint operator that reconstructs the generator $V$ on the spectral domain $\Omega_z$.
	We extend $\tilde V_z$ to a skew-adjoint operator $V_z \colon D(V_z) \to H$, where $D(V_z) \subset H$ and $V_z = \tilde V_z \tilde \Pi$.
	It is shown \cite{GiannakisValva25}*{Theorem~5} that as $z \to 0^+$, $V_z$ converges to $V$ in strong resolvent sense.
	In what follows, we let $\overline{\tilde q_{z}^{-1}}\colon i \mathbb R \to i \mathbb R$ be any continuous extension of $\tilde q_z^{-1}$ on $i \mathbb R \setminus \{0\}$.

	Next, for a choice of kernels $k_\tau$ satisfying \crefrange{prty:K1}{prty:K6}, we approximate $Q_z$ by the compact operators $Q_{z,\tau} \colon \tilde H \to \tilde H$, where
	\begin{displaymath}
		Q_{z,\tau} f = R_z^* \mathcal K_{\tau/2} V \mathcal K_{\tau/2} R_z f \equiv R_z^* V_\tau R_z f, \quad f \in \tilde H.
	\end{displaymath}
	It can be shown that, as $\tau \to 0^+$, $Q_{z,\tau}$ converges strongly to $Q_z$.
	Thus, $V_{z,\tau} \colon D(V_{z,\tau}) \to H$ with $V_{z,\tau} = \tilde V_{z,\tau} \tilde \Pi$ and $\tilde V_{z,\tau} := \overline{\tilde q_{z}^{-1}}(Q_{z,\tau})$ is a family of skew-adjoint operators with compact resolvents that converge in strong resolvent sense to $V$ in the iterated limits $z\to 0^+$ after $\tau \to 0^+$.
	In particular, the two-parameter family $ \{ V_{z,\tau} \}_{z,\tau>0}$ satisfies \crefrange{prty:V1}{prty:V6}, with the $\tau \to 0^+$ limit in \cref{prty:V6} replaced by the iterated limit just mentioned.

	Let
	\begin{equation}
		\label{eq:q_eig}
		Q_{z,\tau} u_{j,z,\tau} = \beta_{j,z,\tau} u_{j,z,\tau}, \quad \beta_{j,z,\tau} \in i \mathbb R,
	\end{equation}
	be an eigendecomposition of $Q_{z,\tau}$, where the eigenvalues/eigenvectors are indexed by $j \in \mathbb N$, and  $ \{ u_{j,z,\tau} \}_{j \in \mathbb N}$ is an orthonormal basis of $\tilde H$ (such a basis exists by skew-adjointness and compactness of $Q_{z,\tau}$).
	Similarly to the solutions of~\eqref{eq:w_eig}, the eigenvalues/eigenvectors corresponding to nonzero eigenvalues come in complex-conjugate pairs
	and we can choose an ordering satisfying $\beta_{2j, \tau, z} = \overline{\beta_{2j-1,z,\tau}}$, $u_{2j,z,\tau}^* = u_{2j-1,z,\tau}$ for $j \in \mathbb N$.
	Applying $\overline{\tilde q_{z}^{-1}}$ to the eigenvalues of $Q_{z,\tau}$ leads to an eigendecomposition of $\tilde V_{z,\tau}$,
	\begin{equation}
		\label{eq:vtilde_eig}
		\tilde V_{z,\tau} u_{j,z,\tau} = i \omega_{j,z,\tau} u_{j,z,\tau}, \quad i \omega_{j,z,\tau} = \overline{\tilde q_{z}^{-1}}(\beta_{j,z,\tau}).
	\end{equation}
	We arrive at an eigendecomposition of $V_{z,\tau}$ by appending the zero-frequency eigenpair $(\omega_{0,z,\tau}, \zeta_{0,z,\tau}) = (0, \bm 1)$ to the eigenpairs $\{(\omega_{j,z,\tau}, \zeta_{j,z,\tau})\}_{j \in \mathbb N}$ of $\tilde V_{z,\tau}$.
	Note that unlike the compact operator $V_\tau$ from \cref{app:compact_approx}, whose eigenvalues cluster at 0, the spectrum of $\tilde V_{z,\tau}$ consists entirely of isolated eigenvalues of finite multiplicity.
	In particular, $\tilde V_{z,\tau}$ has eigenvalues of minimal modulus $\lvert\omega_{j,z,\tau}\rvert$, corresponding to eigenvalues of $Q_{z,\tau}$ of maximal modulus $\lvert \beta_{j,z,\tau}\rvert$.
	This allows the use of iterative numerical algorithms to obtain low-frequency eigenvalues/eigenvectors of $\tilde V_{z,\tau}$ by computing large-modulus eigenvalues/eigenvectors of $Q_{z,\tau}$.

	\begin{rk}
		\label{rk:q_inv}
		In general, the spectrum of $Q_{z,\tau}$ need not be a subset of the interval $i[-(2z)^{-1}, (2z)^{-1}]$ that contains the spectrum of $Q_z$.
		Still, in the applications presented in this paper we have not observed any cases where the inclusion $\sigma(Q_{z,\tau}) \subset i[-(2z)^{-1}, (2z)^{-1}]$ is violated.
		With such an inclusion, we have $\overline{\tilde q_{z}^{-1}}(\beta_{j,z,\tau}) = \tilde q_z^{-1}(\beta_{j,z,\tau})$ and an explicit choice of extension $\overline{\tilde q_{z}^{-1}}$ is not needed.
	\end{rk}

	The approach proposed in \cite{GiannakisValva25} solves~\eqref{eq:q_eig} by solving an equivalent generalized eigenvalue problem,
	\begin{equation}
		\label{eq:gev}
		V_\tau \tilde u_{j,z,\tau} = \beta_{j,z,\tau} (z^2 - V^2) \tilde u_{j,z,\tau}, \quad \tilde u_{j,z,\tau} \in D(V^2),
	\end{equation}
	with $V_\tau \in B(H)$ defined as in \cref{app:compact_approx}.
	Since $\ran V_\tau \subset H_\tau$, from the above it follows that the generalized eigenvectors $\tilde u_{j,z,\tau}$ corresponding to nonzero $\beta_{j,z,\tau}$ lie in $H_\tau$.
	We may therefore order these solutions in increasing order of Dirichlet energy $\mathcal E_\tau(\tilde u_{j,z,\tau})$.
	One readily verifies that the vectors $u_{j,z,\tau} = (z- V) \tilde u_{j,z,\tau}$ corresponding to nonzero eigenvalue $\beta_{j,z,\tau}$ are eigenvectors of $Q_{z,\tau}$ and can thus be normalized to form an orthonormal set in $\tilde H$.

	Analogously to $W_\tau$ from \cref{app:compact_approx}, we define the skew-adjoint operator $W_{z,\tau} = T_\tau V_{z,\tau} T_\tau^*$ with domain $D(W_{z,\tau}) = T_\tau(D(V_{z,\tau})) \oplus \mathcal H_\tau(X)^\perp$ and eigendecomposition
	\begin{displaymath}
		W_{z,\tau} \zeta_{j,z,\tau} = i \omega_{j,z,\tau} \zeta_{j,z,\tau}, \quad \zeta_{j,z,\tau} = T_\tau u_{j,z,\tau},
	\end{displaymath}
	where the eigenfunctions $\zeta_{j,z,\tau}$ form an orthonormal basis of $\mathcal H_\tau(X)$.

	Let $\bm{\tilde V}_m = [\tilde V_{ij}]$, $\tilde V_{ij} = \langle \phi_i, V \phi_j \rangle_H$, be the $m \times m$ matrix representing the generator projected onto the subspace of $\tilde H$ spanned by $\{\phi_1, \ldots, \phi_m\}$.
	Let also $\bm I_m$ be the $m \times m$ identity matrix.
	Approximate solutions of the generalized eigenvalue problem~\eqref{eq:gev} can be computed using similar techniques to those described in \cref{app:compact_approx} for $V_\tau$.
	These methods result in a matrix generalized eigenvalue problem
	\begin{equation}
		\label{eq:gev_mat}
		\bm{\tilde V}_{\tau, m} \bm{\tilde c}_j = i \beta_{j,z,\tau,m} \bm B_m \bm{\tilde c}_j,
	\end{equation}
	where $\bm{\tilde V}_{\tau, m}$ is an $m \times m$ regularized generator matrix defined as in \cref{app:compact_approx}, $\bm B_m$ is the positive-definite $m \times m$ matrix
	\begin{equation*}
		\bm B_m = \bm A^\top_m \bm A_m, \quad \bm A_m = (z \bm I_m - \bm{\tilde V}_m),
	\end{equation*}
	and the generalized eigenvectors $\bm{\tilde c}_j = (\tilde c_{1j}, \ldots, \tilde c_{mj} )^\top$ give the expansion coefficients of $\tilde u_{j,z,\tau,m}$ with respect to the $\phi_i$ basis vectors of $\tilde H$.

	The matrix elements of $\bm{\tilde V}_m$ (and thus $\bm A_m$ and $\bm B_m$) can be computed by means of analogous automatic differentiation and/or ergodic averaging methods as in the computation of $V_{ij}$.
	Dirichlet energies $\mathcal E_\tau(\tilde u_{j,z,\tau,m})$ can similarly be computed by applying~\eqref{eq:dirichlet_eig} to the generalized eigenvectors $\bm{\tilde c_j}$.
	The corresponding eigenvectors $u_{j,z,\tau}$ of $Q_{z,\tau}$ are then approximated by
	\begin{equation*}
		u_{j,z,\tau,m} = \sum_{i=1}^m c_{ij} \phi_i, \quad c_{ij} = \sum_{k=1}^m (z \delta_{ik} - V_{ik}) \tilde c_{kj},
	\end{equation*}
	Letting $\bm c_0 = (1, 0, \ldots, 0)^\top \in \mathbb C^{m + 1}$ and $\bm c_j = (0, c_{1j}, \ldots c_{mj})^\top \in \mathbb C^{m + 1}$ for $j \in \{1, \ldots, m\}$, we define the skew-adjoint operator $W_{z, \tau, m} \colon \mathcal H_{\tau, m} \to \mathcal H_{\tau, m}$ to have eigendecomposition
	\begin{equation*}
		W_{z, \tau, m} \zeta_{j,z,\tau,m} = i \omega_{j,z,\tau,m}, \quad j \in \{0, \ldots, m\},
	\end{equation*}
	where
	\begin{equation}
		\label{eq:zeta_expansion}
		\zeta_{j,z,\tau,m} = \sum_{i=0}^m c_{ij} \phi_i.
	\end{equation}
	The eigenpairs $(\omega_{j,z,\tau,m}, \zeta_{j,z,\tau,m})$ are approximation of $(\omega_{j,z,\tau}, \zeta_{j,z,\tau})$ that converge as $m \to \infty$ in the norm of $\mathcal H_\tau$.

	Pseudocode for the spectral decomposition technique summarized in this section, as applied to the 2-torus examples in \cref{sec:experiments}, can be found in \cref{alg:eigs}.

	\section{Numerical implementation}%
	\label{app:numerical}

	This appendix provides pseudocode for the numerical schemes employed in the experiments described in \cref{sec:experiments}.
	This includes \cref{alg:eigs} (eigendecomposition of the Koopman generator), \cref{alg:classical} (prediction using the true model), and \cref{alg:classical,alg:qm,alg:fock} (prediction using the classical, quantum mechanical, and Fock space approximations, respectively).
	These algorithms were implemented in Python.
	Code reproducing the results in \crefrange{fig:generator_spec}{fig:err_stepanoff} can be found in the repository \url{https://dg227.github.io/NLSA}.
	Our implementation uses JAX \cite{JAX18} as the main numerical library.
	We found that JAX's functional programming features, such as composable function transformations and lazy evaluation, lend themselves well to computational implementation of the operator-theoretic methods in the paper, including eigendecomposition of the regularized generator and evaluation of quantum states.
	Additional notes on numerical implementation are as follows:
	\begin{itemize}[wide]
		\item To solve the generalized eigenvalue problem~\eqref{eq:gev_mat} we used the \texttt{eigs} iterative solver from ScipPy, using JAX to implement the action of the $\bm V_m$ and $\bm B_m$ matrices on vectors.
		      To our knowledge, no JAX-native generalized eigenvalue solvers are available in the public domain at the time of writing this article.
		\item Following solution of~\eqref{eq:gev_mat}, our code proceeds by implementing the eigenfunctions $\tilde\zeta_{j,z,\tau,m}$, the quantum feature maps $\Xi_{\kappa_\text{eval}, \tau, d, m}$ and $\Xi_{\kappa_\text{eval},\tau, n, d, m}$, and the discrete multiplication operators $\pi_l f$ associated with the prediction observables as Python function objects.
		      The prediction functions $f^{(t)}_\text{qm}$ and $f^{(t)}_\text{Fock}$ are also built as function objects by operating on the implementations of the eigenfunctions, feature maps, and multiplication operators.
		      These computations are lazy, in the sense that no computation of intermediate function values is performed prior to formation of $f^{(t)}_\text{qm}$ and  $f^{(t)}_\text{Fock}$.
		      Once formed, the prediction functions are vectorized (to allow simultaneous evaluation and multiple prediction points), and just-in-time (JIT) compiled using JAX to yield efficient implementations capable of running on hardware accelerators such as GPUs.
		\item Numerical integrations of the Stepanoff system~\eqref{eq:vec_stepanoff} (providing the ``true model'' data in \cref{fig:evo_stepanoff,fig:err_stepanoff}) were obtained using the \texttt{Dopri5} solver from the JAX-based Diffrax library \cite{Kidger21}.
	\end{itemize}

	\begin{algorithm}
		\caption{Eigendecomposition of the regularized generator $W_{z, \tau, m}$ in \cref{app:compact_res} for the ergodic rotation and Stepanoff flow on $\mathbb T^2$.
			The notation $(a, \bm v)$ with $a \in \mathbb C$ and $\bm v \in \mathbb C^m$ in Steps~3 and 9 represents the $(m+1)$-dimensional vector obtained by appending $a$ to the elements of $\bm v$.
			See \cref{rk:q_inv} regarding application of the function $\tilde q_z^{-1}$ in Step~8.}
		\label{alg:eigs}
		\begin{algorithmic}[1]
			\REQUIRE Resolvent parameter $z>0$; RKHA parameters $p \in (0,1)$, $\tau>0$; maximal Fourier wavenumber $J \in \mathbb N$; number of eigenfrequency pairs $d \leq ((2J+1)^2 - 1)/2$.
			\ENSURE With $m= (2J+1)^2 - 1$: Eigenfrequencies $\bm \omega = (\omega_0, \ldots, \omega_{2d}) \in \mathbb R^{2d+1}$; eigenvectors $\bm C = (\bm c_0, \ldots, \bm c_{2d+1}) \in \mathbb C^{(m + 1) \times (2d+1)}$; Dirichlet energies $\bm E = (E_0, \ldots, E_{2d}) \in \mathbb R^{(2d+1)}$; eigenfunctions $ \{ \zeta_0, \ldots, \zeta_{2d} \in \mathcal H_\tau \}$.
			\STATE Form the index set $\mathcal J = \{ (j_1, j_2): j_1,j_2 \in \{ -J, \ldots, J \} \}$ with elements $(j_1, j_2)$ ordered in order of increasing $\lvert j_1 \rvert + \lvert j_2 \rvert$.
			Set $\tilde{\mathcal J} = \mathcal J  \setminus \{(0, 0)\}$.
			\STATE Form the generator matrix $\bm{\tilde V}_m \in \mathbb R^{m\times m} = [V_{rs}]_{r,s: \vec j_r, \vec j_s \in \tilde{\mathcal J}}$ using~\eqref{eq:gen_rot} (ergodic rotation) or~\eqref{eq:gen_stepanoff} (Stepanoff flow).
			\STATE Compute the RKHA inverse weight vectors $\bm{\tilde \lambda}_{\tau,m} \in \mathbb R^m = [e^{-\tau (\lvert j_1 \rvert^p + \lvert j_2 \rvert^p)/2}]_{(j_1, j_2) \in \tilde{\mathcal J}}$ and $\bm \lambda_{\tau, m} = (\lambda_0, \ldots, \lambda_m) \equiv (1, \bm{\tilde \lambda}_{\tau, m})$.
			\STATE Form the regularized generator matrix $\bm{\tilde V}_{\tau,m} \in \mathbb R^{m \times m} = \bm{\tilde \Lambda}_{\tau, m} \bm{\tilde V}_m \bm{\tilde \Lambda}_{\tau, m}$, where $\bm{\tilde \Lambda}_{\tau, m} = \texttt{diag}(\bm{\tilde \lambda}_{\tau, m})$.
			\STATE Form the mass matrix $\bm B_m \in \mathbb R^{m \times m} = z^2 \bm I_m - \bm {\tilde V}^T_m \bm {\tilde V}_m $, where $\bm I_m$ is the $m\times m$ identity matrix.
			\STATE Compute $2d$ eigenvalues $\beta_1, \ldots, \beta_{2d} \in \mathbb C$ of largest modulus from the generalized eigenvalue problem~\eqref{eq:gev_mat} and the corresponding eigenvectors $\bm{\tilde c}_1, \ldots, \bm{\tilde c}_{2d} \in \mathbb C^m$.
			\STATE Compute the Dirichlet energies $E_j$ of the eigenvectors $\bm{\tilde c}_j$ using~\eqref{eq:dirichlet_eig}.
			Sort $(\beta_1, \ldots, \beta_{2d})$, $(\bm{\tilde c}_1, \ldots, \bm{\tilde c}_{2d})$, and $(E_1, \ldots, E_{2d})$ in order of increasing Dirichlet energy $E_j$.
			\STATE Compute the eigenfrequencies $\omega_r$ corresponding to $\beta_r$ using~\eqref{eq:q_inv}.
			\STATE Compute the vectors $\bm c_1, \ldots, \bm c_d \in \mathbb C^{r+1}$, where $\bm c_r = (0, (z \bm I_m - \bm V_m) \bm{\tilde c}_r)$.
			Normalize the $\bm c_r$ to unit 2-norm.
			\STATE For each $\bm c_r = (c_{0r}, \ldots, c_{mr})^\top$, form the function $\zeta_r \colon \mathbb T^2 \to \mathbb C$, where $\zeta_r(x) = \sum_{s: \vec j_s \in  \mathcal J} c_{sr} \lambda_s e^{i j_s \cdot x} $.
			\STATE Set $\omega_0 = 0$, $\bm c_0 = [\delta_{\vec j,(0,0)}]_{\vec j\in \mathcal J}$, $E_0 = 0$, $\zeta_0(x) = 1$.
			\RETURN $\bm \omega = (\omega_0, \ldots, \omega_{2d})$, $\bm C = (\bm c_0, \ldots, \bm c_{2d})$, $\bm E = (E_0, \ldots, E_{2d})$, $ \{ \zeta_0, \ldots, \zeta_{2d} \}$.
		\end{algorithmic}
	\end{algorithm}

	\begin{algorithm}
		\caption{\label{alg:classical}Classical prediction using~\eqref{eq:f_cl}.}
		\begin{algorithmic}[1]
			\REQUIRE  Eigenfrequencies $\bm \omega \in \mathbb R^{2d+1}$, eigenvectors $\bm C \in \mathbb C^{(m+1) \times (2d+1)}$, and eigenfunctions $ \{ \zeta_0, \ldots, \zeta_{2d} \in \mathcal H_\tau \}$ from \cref{alg:eigs}; column vector of Fourier coefficients $\bm{\hat f} \in \mathbb C^m$ of prediction observable (ordered according to the index set $\mathcal J$ from \cref{alg:eigs}); evolution time $t \in \mathbb R$; evaluation point $x \in \mathbb T^2$.
			\ENSURE Value $y= f^{(t)}_\text{cl}(x) \in \mathbb R$ of classical prediction function.
			\STATE Form unitary evolution matrix $\bm U^t \in \mathbb C^{(2d+1) \times (2d+1)} = \texttt{diag}([e^{i\omega_j t}]_{j\in \{ 0, \ldots, 2d +1 \}})$.
			\STATE Compute the column vector $\bm \zeta = (\zeta_0(x), \ldots, \zeta_{2d}(x))^\top$ of eigenfunction values.
			\RETURN $y = (\bm C^\dag \bm{\hat f})^\top  \bm U^t \bm \zeta$.
		\end{algorithmic}
	\end{algorithm}

	\begin{algorithm}
		\caption{\label{alg:qm}Quantum mechanical prediction using~\eqref{eq:ft_qm_matrix}.}
		\begin{algorithmic}[1]
			\REQUIRE  Eigenfrequencies $\bm \omega \in \mathbb R^{2d+1}$, eigenvectors $\bm C \in \mathbb C^{(m + 1) \times (2d+1)}$, and eigenfunctions $ \{ \zeta_0, \ldots, \zeta_{2d} \in \mathcal H_\tau \}$ from \cref{alg:eigs}; sharpness parameter $\kappa_\text{eval} >0$ for quantum feature map; values $\bm f = (\tilde f(x_0^{(l)}), \ldots, \tilde f(x^{(l)})_{l^2-1}) \in \mathbb C^{l^2}$ of prediction observable on $l\times l$ grid on $\mathbb T^2$; evolution time $t \in \mathbb R$; evaluation point $x \in \mathbb T^2$.
			\ENSURE Value $y= f^{(t)}_\text{qm}(x) \in \mathbb R$ of quantum mechanical prediction function.
			\STATE Form unitary evolution matrix $\bm U^t \in \mathbb C^{(2d+1) \times (2d+1)} = \texttt{diag}([e^{i\omega_j t}]_{j\in \{ 0, \ldots, 2d +1 \}})$.
			\STATE Form the matrix $\bm \Phi \in \mathbb C^{l^2 \times (m + 1)} = [\Phi_{ij}]$ of Fourier function values, where $\Phi_{ij} = \phi_j(x^{(l)}_i)$.
			\STATE Compute the column vector $\bm \xi_{t,x} \in \mathbb C^{l^2} = \bm \Phi \bm C \bm U^{t\dag} \bm C^\dag \bm{\hat{F}}_{\kappa_\text{eval}}(x)$ using $\bm{\hat{F}}_{\kappa_\text{eval}}$ from \cref{sec:qm_model}.
			\STATE Return $y = \bm\xi_{t,x}^\dag \bm M \bm\xi_{t,x} / \bm \xi_{t,x}^\dag \bm \xi_{t,x}$.
		\end{algorithmic}
	\end{algorithm}

	\begin{algorithm}
		\caption{\label{alg:fock}Fock space prediction using~\eqref{eq:ft_fock_matrix}. The index set $\mathcal J$ in Step~2 is built similarly to Step~1 of \cref{alg:eigs}.
			The notation $\bm v^{.n}$ represents the elementwise $n$-th power of column vector $\bm v$.}
		\begin{algorithmic}[1]
			\REQUIRE  Eigenfrequencies $\bm \omega \in \mathbb R^{2d+1}$, eigenvectors $\bm C \in \mathbb C^{(m + 1) \times (2d+1)}$, and eigenfunctions $ \{ \zeta_0, \ldots, \zeta_{2d} \in \mathcal H_\tau\}$ from \cref{alg:eigs}; Fock space grading $n \in \mathbb N$; kernel smoothing parameter $\sigma > 0$; sharpness parameter $\kappa_\text{eval} >0$ for quantum feature map; values $\bm f = (\tilde f(x_0^{(l)}), \ldots, \tilde f(x^{(l)})_{l^2-1}) \in \mathbb C^{l^2}$ of prediction observable on $l\times l$ grid on $\mathbb T^2$; evolution time $t \in \mathbb R$; evaluation point $x \in \mathbb T^2$.
			\ENSURE Value $y= f^{(t)}_\text{qm}(x) \in \mathbb R$ of Fock space prediction function.
			\STATE Form unitary evolution matrix $\bm U^t \in \mathbb C^{(2d+1) \times (2d+1)} = \texttt{diag}([e^{i\omega_j t}]_{j\in \{ 0, \ldots, 2d +1 \}})$.
			\STATE Compute the RKHA inverse weight vector $\bm \lambda_\sigma \in \mathbb R^{m + 1} = [e^{-\sigma (\lvert j_1 \rvert^p + \lvert j_2 \rvert^p)/2}]_{(j_1, j_2) \in \mathcal J}$ and form the associated kernel smoothing matrix $\bm \Lambda_\sigma \in \mathbb C^{(2d+1) \times (2d+1)} = \texttt{diag}(\bm \lambda_\sigma)$.
			\STATE Form the matrix $\bm \Phi \in \mathbb C^{l^2 \times (m + 1)} = [\Phi_{ij}]$ of Fourier function values, where $\Phi_{ij} = \phi_j(x^{(l)}_i)$.
			\STATE Compute the column vector $\bm \xi_{n,t,x} \in \mathbb C^{l^2} = (\bm \Phi \bm \Lambda_\sigma \bm C \bm U^{t\dag} \bm C^\dag \bm{\hat{F}}_{\kappa_\text{eval},n}(x))^{.n}$ using $\bm{\hat{F}}_{\kappa_\text{eval,n}}$ from \cref{sec:fock_space_model}.
			\STATE Return $y = \bm\xi_{n,t,x}^\dag \bm M \bm\xi_{n,t,x} / \bm \xi_{n,t,x}^\dag \bm \xi_{n,t,x}$.
		\end{algorithmic}
	\end{algorithm}

\end{appendices}


% \bibliography{bibliography_dg}

\end{document}